\documentclass[a4paper,11pt,openright,twoside]{book}
\usepackage[english]{babel}
\usepackage[utf8]{inputenc}
\usepackage[T1]{fontenc}
\usepackage{lmodern}
\usepackage{latexsym}
\usepackage{a4}
\usepackage{amsfonts}
\usepackage{bm}
\usepackage{amssymb}
\usepackage{amsthm} 
\usepackage{color}
\usepackage{eucal}
\usepackage{amscd}
\usepackage{amsmath,mathtools}
\usepackage{array}
\usepackage{setspace}
\usepackage{listings}
\usepackage{graphicx}
\usepackage{xifthen}
\usepackage{todonotes}
\usepackage{caption}
\usepackage{subcaption}
\usepackage{csquotes}
\usepackage{cases}
\usepackage{multirow}

\usepackage{pgf,tikz}
\usetikzlibrary{matrix, calc, arrows,shapes,decorations}
\usepackage{pgfplots} 
\usepackage{circuitikz}
\usetikzlibrary{decorations.markings}
\tikzset{middlearrow/.style={
		decoration={markings,
			mark= at position 0.5 with {\arrow{#1}} ,
		},
		postaction={decorate}
	}
}

\newcommand{\sqdiamond}[1][fill=black]{\tikz [x=1ex,y=1ex,line width=.1ex,line join=round, yshift=1ex, scale=2] \draw  [#1]  (1,1) -- (.5,1.5) -- (1,2) -- (1.5,1.5) -- (1,1) -- cycle;}%
\newcommand{\MyDiamond}[1][fill=black]{\mathop{\raisebox{-0.275ex}{$\sqdiamond[#1]$}}}

\newcommand{\tikzcircle}[2][red,ultra thick]{\tikz[baseline=-.7ex]\draw[#1,radius=#2] (0,0) circle (.15) ;}%

\newcommand{\tikzdot}[2][fill=black]{\tikz[baseline=-.7ex]\draw[#1,radius=#2] (0,0) circle (0.04) ;}%

\newcommand{\arro}[1][fill=black]{\tikz [x=1ex,y=1ex,line width=.1ex,line join=round, yshift=1ex, scale=2]\draw[#1] (-0.5,0.5) to (0.5,-0.5);}%
\newcommand{\SEARROW}[1][fill=black]{\mathop{\raisebox{-0.275ex}{$\arro[#1]$}}}

\DeclareMathAlphabet{\mathcal}{OMS}{cmsy}{m}{n}

\usepackage[
backend=biber,
style=numeric,
sorting=nyt,
maxbibnames=4]{biblatex}
\newboolean{english}
\setboolean{english}{true} % mettere "true" o "false" per versione in inglese o in italiano
\newboolean{corelatore}
\setboolean{corelatore}{false} % mettere "true" o "false" per versione con o senza co-relatore

\usepackage[left=3cm,
top = 3cm, 
bottom = 3cm,
right= 3cm,
]{geometry} %customize geometries of the page output
\usepackage{fancyhdr}
\usepackage{emptypage}
\usepackage{textcomp} %need this for \textbigcircle
\usepackage{enumitem}
\usepackage{accents}

\pgfplotsset{
	discard if not/.style 2 args={
		x filter/.code={
			\edef\tempa{\thisrow{#1}}
			\edef\tempb{#2}
			\ifx\tempa\tempb--
			\else
			
			\fi
		}
	}
}
\pgfplotsset{compat=1.14}%1.16 %avoid warning
\usepackage[lined,boxed,commentsnumbered]{algorithm2e}

\usepackage{hyperref}
\hypersetup{colorlinks = true,
	linkcolor = black,
	urlcolor  = black,
	citecolor = black,
	anchorcolor = black,
	hypertexnames=false}

\allowdisplaybreaks[4]

\providecommand{\red}{\textcolor{red}}
\definecolor{myblue}{rgb}{0,0,0.6}

\definecolor{pastelgreen}{rgb}{0.47, 0.87, 0.47}

\definecolor{amcol}{rgb}{0.8,0,0}

\definecolor{rvcol}{rgb}{1,0,0}

\definecolor{lmigcol}{rgb}{0.8,0.3,0}

\newcommand*{\N}[1]{\left\|#1\right\|}
\newcommand*{\abs}[1]{\left|#1\right|}

\newcommand{\jmp}[1]{[\![#1]\!]}    
\newcommand{\mvl}[1]{\{\!\!\{#1\}\!\!\}}

\newcommand{\Uu}[1]{{\mathbf{#1}}}
\newcommand{\IN}{\mathbb{N}}\newcommand{\IP}{\mathbb{P}}
\newcommand{\IR}{\mathbb{R}}

\newcommand{\be}{{\Uu e}}\newcommand{\bn}{{\Uu n}}
\newcommand{\bx}{{\Uu x}}
\newcommand{\by}{{\Uu y}}\newcommand{\bz}{{\Uu z}}

\newcommand{\calA}{{\mathcal A}}
\newcommand{\calE}{{\mathcal E}}
\newcommand{\calF}{{\mathcal F}}
\newcommand{\calL}{{\mathcal L}}
\newcommand{\calO}{{\mathcal O}}
\newcommand{\calT}{{\mathcal T}}
\newcommand{\calM}{{\mathcal M}}

\newcommand{\tand}{\text{ and }}
\newcommand{\tfor}{\text{ for }}

\newcommand{\Th}{{(\calT_h)}}

\newcommand{\mi}{{\boldsymbol{i}}}
\newcommand{\mj}{{\boldsymbol{j}}}
\newcommand{\mk}{{\boldsymbol{k}}}

\newcommand\refb{b_J}
\newcommand{\QT}{{\mathbb{Q\!T}}}

\newcommand{\bw}{{\bm{w}}}
\newcommand{\bbeta}{{\boldsymbol \beta}}
\newcommand{\bk}{{\bm{K}}}
\newcommand{\bi}{{\bm{i}}}
\newcommand{\bj}{{\bm{j}}}
\newcommand{\br}{{\bm{r}}}

\newcommand{\tayt}{{\mathsf{T}}}

\allowdisplaybreaks[4]

\DeclareMathOperator*{\esssup}{sup\, ess}

\newtheorem{theorem}{Theorem}[section]
\newtheorem{lemma}[theorem]{Lemma}
\newtheorem{prop}[theorem]{Proposition}
\newtheorem{rem}[theorem]{Remark}

\theoremstyle{definition}
\newtheorem{Def}{Definition}[section]

\newcommand{\vertiii}[1]{{\left\vert\kern-0.25ex\left\vert\kern-0.25ex\left\vert #1 
		\right\vert\kern-0.25ex\right\vert\kern-0.25ex\right\vert}}

\newenvironment{flushleftequation*}
{\begin{equation*}\begin{lrbox}{\flusheqbox}$\displaystyle}
		{$\end{lrbox}\makebox[\displaywidth][l]{\usebox{\flusheqbox}}\end{equation*}\ignorespacesafterend}
\newsavebox{\flusheqbox}

\usepackage{comment}
\usepackage{floatrow}
\usepackage{csvsimple} 
\usepackage{siunitx} %for scientific numbers \num
\usepackage{tabularx}
\usepackage{cleveref} 
\usepackage{booktabs}

\definecolor{pscol}{rgb}{0,0.6,0}

\pgfplotscreateplotcyclelist{paulcolors}{%
	teal,every mark/.append style={solid,fill=teal},mark=*,very thick,mark size=3pt\\%
	orange,every mark/.append style={solid,fill=orange},mark=square*,very thick,mark size=3pt\\%
	violet,every mark/.append style={solid,fill=violet},mark=dd*,very thick,mark size=3pt\\%
	teal,densely dashed,every mark/.append style={solid,fill=teal},mark=*,very thick,mark size=3pt\\%
	orange,densely dashed,every mark/.append style={solid,fill=orange},mark=square*,very thick,mark size=3pt\\%
	violet,densely dashed,every mark/.append style={solid,fill=violet},mark=diamond*,very thick,mark size=3pt\\%
	teal,dash dot,every mark/.append style={solid,fill=teal},mark=*,very thick,mark size=3pt\\%
	orange,dash dot,every mark/.append style={solid,fill=orange},mark=square*,very thick,mark size=3pt\\%
	violet,dash dot,every mark/.append style={solid,fill=violet},mark=diamond*,very thick,mark size=3pt\\%
}
\pgfplotsset{
	discard if not/.style 2 args={
		x filter/.code={
			\edef\tempa{\thisrow{#1}}
			\edef\tempb{#2}
			\ifx\tempa\tempb
			\else
			
			\fi
		}
	}
}
\pgfplotsset{compat=1.14}%1.16 %avoid warning

\newcommand{\markchapterintro}[1]{\markboth{#1}{}}

\begin{document}
	
	\thispagestyle{empty}
	\space
	\begin{center}
		\Large{Università degli Studi di Pavia}\\
		\vspace{0.2cm}
		\normalsize{Dipartimento di Matematica “Felice Casorati”\\
			\vspace{0.2cm}
			Laurea magistrale in Matematica}
	\end{center}
	\[\]
	\begin{center}
		\includegraphics[width=0.35\textwidth]{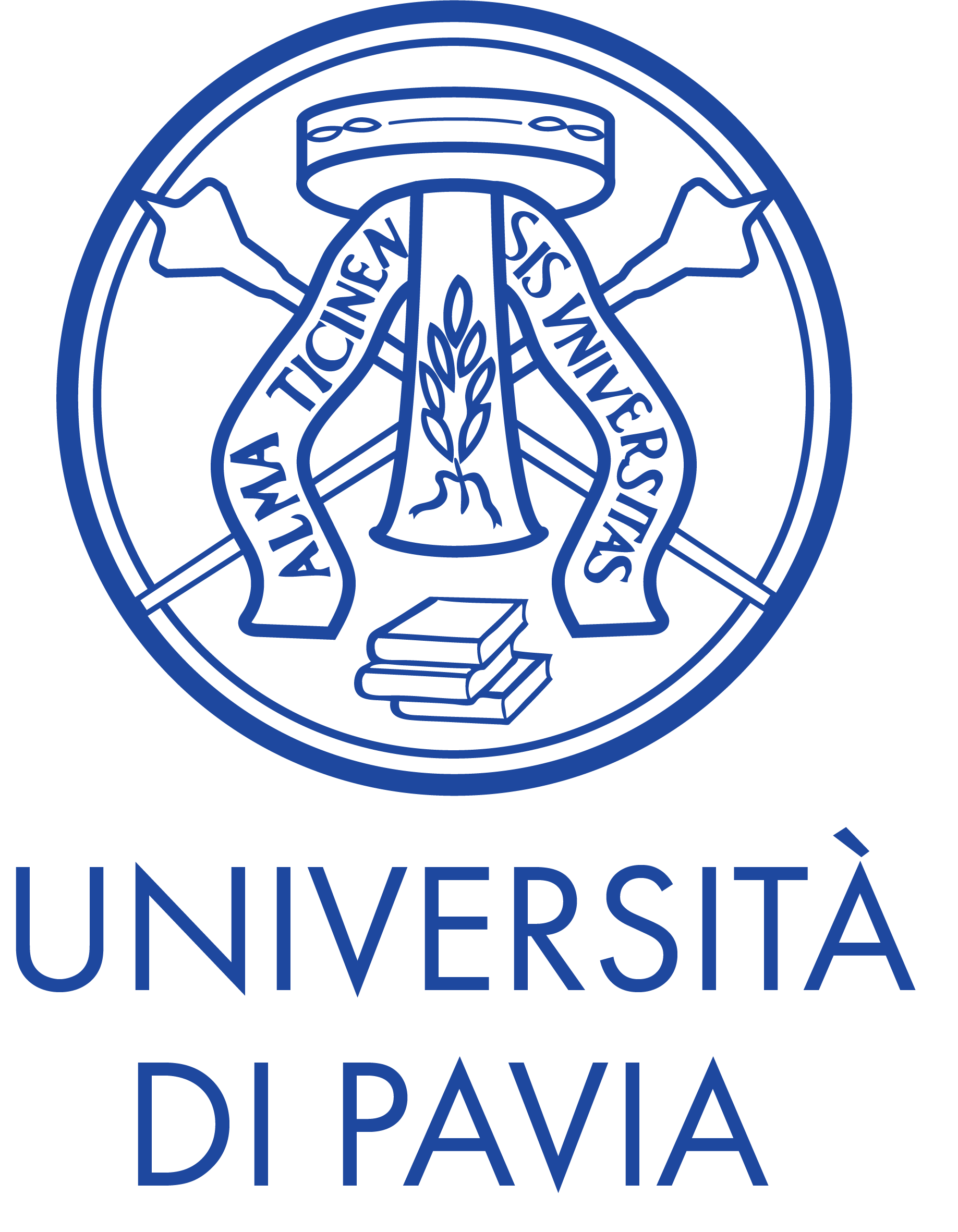}
	\end{center}
	\[\]
	\space
	\begin{center}
		
		{\textbf{\LARGE  A quasi-Trefftz discontinuous Galerkin method}}
		
		\vspace{0.3cm}
		{\textbf{\LARGE 	for the homogeneous diffusion-advection-reaction}}
		
		\vspace{0.3cm}
		{\textbf{\LARGE  equation 
				with piecewise-smooth coefficients}}
		
	\end{center}
	\singlespace
	\[\]
	%\begin{center}
	%	\textbf{Tesi di Laurea Magistrale in Matematica}
	%\end{center}
	\[\]\[\]\[\]
	\begin{flushleft}
		\ifthenelse{\boolean{english}}{Relatore:\\}{Relatore:\\}
		\textbf{Prof. Andrea Moiola}
		\ifthenelse{\boolean{corelatore}}{
			\ifthenelse{\boolean{english}}{\\Correlatore (Co-Supervisor):\\}{\\Correlatore:\\}
			\textbf{Nome Prof}
		}
		{}
	\end{flushleft}
	\[\]\[\]
	\begin{flushright}
		Tesi di Laurea di \\
		\textbf{Chiara Perinati}\\
		Matricola 508357
	\end{flushright}
	\[\]\[\]\[\]
	\begin{center}
		Anno Accademico 2022/2023
	\end{center}

	\newpage\null\thispagestyle{empty}
	
	\newpage\pagenumbering{roman}
	
	\newpage
	\thispagestyle{plain}
	\vspace*{1cm}
	\begin{center}
		\textbf{\Large Abstract}
	\end{center}
	We describe and analyse a quasi-Trefftz discontinuous Galerkin method for solving boundary value problems for the homogeneous diffusion-advection-reaction equation with piece\-wise\--smooth coefficients.
	
	The discontinuous Galerkin (DG) weak formulation is derived including an interior penalty parameter and using the classical upwind numerical fluxes. We compare three different formulations which arise from different choices of the symmetrization parameter, including the Symmetric Interior Penalty Galerkin (SIPG) one. DG methods may require a higher number of degrees of freedom compared to continuous ones, especially in the case of high-order schemes; combining them with quasi-Trefftz functions allows for a reduction in the number of degrees of freedom.
	
	Trefftz schemes are high-order Galerkin methods whose discrete functions satisfy exactly the underlying  partial differential equation (PDE) in each mesh element (for example, harmonic polynomials for the Laplace equation).
	Since a family of local exact solutions is needed,
	Trefftz basis functions can be easily computed  for many PDEs that are linear, homogeneous and with piecewise-constant coefficients. However, if the equation has varying coefficients, in general, exact solutions are unavailable, hence  the construction of  discrete Trefftz spaces is not possible.
	
	The quasi-Trefftz methods have been introduced to overcome this limitation. They rely on  discrete spaces of functions that are not exact solutions but elementwise “approximate solutions” of the PDE in the sense of Taylor polynomials:  the Taylor polynomial of a given order of their image under the partial differential operator vanishes at the barycentre of a mesh element. 
	
	A space-time quasi-Trefftz DG method has recently been studied in \cite{imbert2023space}, extending
	to the case of smoothly varying coefficients the space-time Trefftz DG scheme for the acoustic wave equation of \cite{moiola2018space}.
	Since it has shown excellent results, we propose a related method that can be applied to the second-order diffusion-advection-reaction elliptic equation.
	
	We choose polynomial quasi-Trefftz basis functions, whose coefficients can be computed with a simple algorithm, which is initialized assigning a  sort of  Cauchy conditions and is based on the Taylor expansion of the coefficients of the PDE. 
	The main advantage of Trefftz and quasi-Trefftz schemes over more classical ones is the higher accuracy for comparable numbers of degrees of freedom. 
	We prove that the dimension of the quasi-Trefftz space is smaller than the dimension of the polynomial space of the same degree and that yields the same optimal convergence rates as the full polynomial space.
	The quasi-Trefftz DG method is well-posed, consistent and stable and we prove its high-order convergence.
	We present some numerical experiments in two dimensions that show excellent properties in terms of approximation and convergence rate.
	\newpage
	\thispagestyle{plain}
	\vspace*{1cm}
	\begin{center}
		\textbf{\Large Sommario}
	\end{center}
	In questo elaborato descriviamo e analizziamo un metodo quasi-Trefftz discontinuous Galerkin per risolvere problemi al contorno per l'equazione di diffusione-trasporto-reazione omogenea con coefficienti  lisci a tratti. 
	
	La formulazione debole discontinuous Galerkin (DG) include  i classici flussi numerici upwind, un parametro di interior penalty e un parametro di simmetrizzazione che porta a tre diverse formulazioni, tra cui quella del metodo Symmetric Interior Penalty Galerkin (SIPG). I metodi DG spesso richiedono un maggior numero di gradi di libertà rispetto a quelli continui, soprattutto nel caso di metodi di alto ordine; combinandoli con funzioni quasi-Trefftz otteniamo una riduzione del numero di gradi di libertà. 
	
	I metodi Trefftz sono metodi di alto ordine di tipo Galerkin  le cui funzioni test e trial  sono soluzioni esatte dell'equazione alle derivate parziali (EDP) in ogni elemento della mesh (ad esempio, polinomi armonici per l'equazione di Laplace). Poiché sono necessarie soluzioni locali esatte, le funzioni Trefftz di base possono essere facilmente calcolate per molte EDP che sono lineari, omogenee e con coefficienti costanti a tratti. Tuttavia, se l'equazione ha coefficienti variabili, in generale, le soluzioni esatte non sono disponibili, quindi la costruzione di spazi discreti Trefftz non è possibile.
	
	I metodi quasi-Trefftz sono stati introdotti per superare questa limitazione. Si basano su spazi discreti di funzioni che non sono soluzioni esatte ma “soluzioni approssimate” dell'EDP in ogni elemento. Più precisamente, il polinomio di Taylor di un dato ordine della loro immagine tramite l'operatore differenziale parziale si annulla nel baricentro di un elemento della mesh. 
	
	Un metodo spazio-tempo quasi-Trefftz DG  è stato recentemente studiato in \cite{imbert2023space}, estendendo al caso di coefficienti variabili il metodo spazio-tempo Trefftz DG per l'equazione delle onde acustiche analizzato in \cite{moiola2018space}. Poiché ha mostrato eccellenti risultati, proponiamo un metodo correlato per l'equazione ellittica del secondo-ordine di diffusione-trasporto-reazione.
	
	Scegliamo delle funzioni quasi-Trefftz di base polinomiali, i cui coefficienti possono essere calcolati con un semplice algoritmo, che viene inizializzato assegnando una sorta di condizioni di Cauchy e si basa sull'espansione di Taylor dei coefficienti dell'EDP. Il principale vantaggio dei metodi Trefftz e quasi-Trefftz rispetto a quelli più classici è la maggiore accuratezza per un numero comparabile di gradi di libertà. Dimostriamo che la dimensione dello spazio quasi-Trefftz è inferiore alla dimensione dello spazio polinomiale dello stesso grado e che si ottengono gli stessi ordini di convergenza ottimali dello spazio polinomiale completo. Il metodo quasi-Trefftz  DG è ben posto, consistente e stabile e dimostriamo la convergenza di ordine alto. Presentiamo alcuni esperimenti numerici in due dimensioni che mostrano ottime proprietà in termini di approssimazione e ordini di convergenza.
	
	\tableofcontents
	
	\clearpage
	\pagenumbering{arabic}
	\addcontentsline{toc}{chapter}{Introduction} 
	
	\chapter*{Introduction}
	\markchapterintro{Introduction}
	
	Trefftz methods are numerical schemes for approximating solutions of boundary value problems. 
	While Standard Finite Element methods seek to approximate a solution of a partial differential equation (PDE) by piecewise polynomials, Trefftz methods rely on problem-dependent bases,  incorporating some properties of the PDE into the test and trial spaces, hence also into the discrete solution.
	Trefftz schemes were proposed in 1926 by  Erich Trefftz \cite{trefftz1926gegenstuck} and are high-order Galerkin methods whose trial and test functions are exact solutions of the governing PDE in each mesh element. For example, if we consider the Laplace equation, instead of piecewise polynomials, a natural choice of approximating functions is given by piecewise harmonic polynomials. 
	
	Since local exact solutions are required, in general the PDEs considered are  linear, homogeneous, i.e., with no source term, and with piecewise-constant coefficients, for which Trefftz basis functions can be easily computed. 
	However, if the equation has varying coefficients, in general, exact solutions are unavailable, making the construction of discrete Trefftz spaces impossible. 
	
	Quasi-Trefftz methods have been introduced to solve this problem: they use discrete spaces of functions that are not exact solutions but elementwise “approximate solutions” of the PDE. More precisely, quasi-Trefftz functions satisfy the following property: the Taylor polynomial of a given order of their image under the partial differential operator vanishes at the barycentre of a mesh element. 
	
	Trefftz methods proved successful especially for time-harmonic wave propagation (Helm\-holtz equation) using plane waves as basis functions \cite{hiptmair2016survey} and have been extended in \cite{imbert2014generalized} to the case of smoothly varying coefficients using Generalized Plane Waves (GPWs), which are approximate solutions of the Helmholtz equation in the sense of Taylor polynomials.
	
	An extension of a space-time Trefftz DG scheme for the acoustic wave equation of \cite{moiola2018space} to the case of smoothly varying coefficients has recently been studied in \cite{imbert2023space} using the quasi-Trefftz method: 
	since it has shown excellent results, we propose a related method that can be applied to the second-order elliptic equation.
	In this thesis we describe and analyse a quasi-Trefftz discontinuous Galerkin method for solving boundary value problems for the homogeneous diffusion-advection-reaction equation with piecewise-smooth coefficients.
	
	We choose polynomial quasi-Trefftz basis functions, whose coefficients can be computed with a simple algorithm, which is initialized assigning a sort of  Cauchy conditions and is based on the Taylor expansion of the coefficients of the PDE. 
	
	The main advantage of Trefftz and quasi-Trefftz schemes over more classical ones is the higher accuracy for comparable numbers of degrees of freedom.
	The quasi-Trefftz space, compared to the full polynomial space of the same degree, provides the same convergence rates with respect to the mesh size and has a much smaller dimension, especially for large polynomial degrees.
	
	Trefftz and quasi-Trefftz methods are often formulated in a discontinuous Galerkin (DG) framework. 
	DG methods are based on discontinuous piecewise polynomial functions and they can be easily combined with Trefftz and quasi-Trefftz functions due to the decoupling of the basis construction of elements.
	DG schemes often require a higher number of degrees of freedom compared to continuous ones, especially in the case of high-order schemes and combining them with quasi-Trefftz functions allows for a reduction in the number of degrees of freedom.
	
	For PDEs with diffusion, DG methods originated from the use of an interior penalty to weakly enforce continuity conditions  as in \cite{arnold1982interior}. See \cite{arnold2002unified} for a unified analysis of DG methods for the Laplace equation.
	For advection–reaction problems, it is common to use upwind numerical fluxes in the DG formulation \cite{brezzi2004discontinuous}.
	DG scheme for diffusion-advection-reaction problems that combine these two approaches have been studied in \cite{ houston2002discontinuous, ayuso2009discontinuous,di2011mathematical}. 
	
	We now outline the structure of the thesis.
	
	In Chapter \ref{Chapter1} we  describe the boundary value problem for the diffusion-advection-reaction equation and analyse the classical variational formulation and its well-posedness. 
	We introduce the abstract framework for the analysis of the discontinuous Galerkin methods, in particular the non-conforming error analysis based on \cite[Chapter 1]{di2011mathematical}, which relies on three key ingredients: consistency, discrete coercivity and boundedness.
	
	The discrete setting is introduced in Chapter \ref{Chapter2} by stating assumptions on the mesh sequences and defining the typical tools of DG approximations such as averages, jumps and broken function spaces.
	
	We derive the discontinuous Galerkin variational formulation in Chapter \ref{Chapter3}. 
	We consider the interior penalty method, in its symmetric (SIPG), incomplete (IIPG) and non-symmetric (NIPG) variants
	to handle the diffusion term
	and a DG scheme with upwind numerical fluxes to handle the advection-reaction terms, following mainly \cite{di2011mathematical} and \cite{riviere2008discontinuous}.
	
	This method is consistent by construction and in Chapter \ref{Chapter4} we prove the discrete coercivity and the boundedness with respect to a mesh-dependent norm. 
	These results imply the well-posedness of the DG method and the quasi-optimality in such norm for any  discrete space that is a subspace of the full polynomial space.
	Furthermore, we study the convergence analysis for the full polynomial space.
	
	In Chapter \ref{Chapter5} we define  the polynomial quasi-Trefftz space for a general linear homogeneous PDE with sufficiently smooth coefficients and prove its optimal approximation properties. 
	We focus on the quasi-Trefftz space for the homogeneous diffusion-advection-reaction equation and prove the $h$-convergence of the quasi-Trefftz DG scheme.
	We describe an algorithm for the construction of a basis  in dimension $d\in\mathbb{N}$, $d\ge1$.
	
	The quasi-Trefftz DG method has been implemented in the two-dimensional case and some numerical experiments are illustrated in Chapter \ref{Chapter6} to validate the theoretical results, compare the quasi-Trefftz space with the full polynomial space and assess some other properties of the method.
	
	We draw some conclusions and outline some future developments in Chapter \ref{Chapter7}.

	\chapter{Model problem}\label{Chapter1}
	In this chapter we present the boundary value problem for the diffusion-advection-reaction equation and we introduce the abstract framework to build discontinuous Galerkin approximations of the model problem.
	In Section \ref{abstractsetting} we describe the continuous variational setting and we consider the Galerkin method and non-conforming methods  for approximating a variational problem. We follow Di Pietro and Ern \cite{di2011mathematical} for the non-conforming analysis.
	In Section \ref{problema} we describe the model problem, its classical variational formulation  and prove its well-posedness.
	\section{Abstract setting}\label{abstractsetting}
	In this section we introduce the abstract framework for the analysis of the discontinuous Galerkin methods.
	We start with the variational setting, then we consider the Galerkin method and non-conforming methods and then we  recall some classical function spaces that we will use in the analysis.
	\subsection{Variational problems and Galerkin}
	Let $V$ be a real Hilbert space equipped with the norm $\N{\cdot}_V$, let $\calA: V\times V\to\mathbb{R}$ be a bilinear form on $V$ and $\calF: V\to \mathbb{R}$ be a linear functional on $V$.
	Consider the following abstract variational problem:
	\begin{equation}\label{abstractproblem}
		\text{Find } u\in V \text{ such that } \calA(u,v)=\calF(v) \text{ for all } v\in V.
	\end{equation}
	We recall the definitions of coercivity and continuity of a bilinear form $\calA$ and of  continuity of a linear form $\calF$.
	\begin{Def}[Coercivity]
		A bilinear form $\calA$ defined on $\left(V,\N{\cdot}_V\right)$ is \textit{coercive} if there exists $\alpha>0$ such that
		\begin{equation*}
			\calA(v,v)\ge \alpha\N{v}_V^2 \quad 	\forall v\in V.
		\end{equation*}
	\end{Def}
	\begin{Def}[Continuity]
		A bilinear form $\calA$ defined on $\left(V,\N{\cdot}_V\right)$ is \textit{continuous} if there exists $M>0$ such that
		\begin{equation*}
			\abs{\calA(v,w)}\leq M\N{v}_V\N{w}_V \quad \forall v,w\in V.
		\end{equation*}
	\end{Def}
	\begin{Def}[Continuity]
		A linear form $\calF$ defined on $\left(V,\N{\cdot}_V\right)$ is \textit{continuous} if there exists $M>0$ such that
		\begin{equation*}
			\abs{\calF(v)}\leq M\N{v}_V \quad \forall v\in V.
		\end{equation*}
	\end{Def}
	Sufficient conditions for the well-posedness of the problem \eqref{abstractproblem} are given by the Lax--Milgram Theorem:
	\begin{theorem}[Lax--Milgram]\label{Lax-Milgram}
		Given the problem \eqref{abstractproblem}, assume that the bilinear form $\calA$ is continuous and coercive on $V$ and that the linear functional $\calF$ is continuous on $V$.
		Then, there exists a unique solution $u\in V$ of the variational problem \eqref{abstractproblem}.
	\end{theorem}
	
	In order to approximate a variational problem one can use the Galerkin method, which consists in restricting the variational problem \eqref{abstractproblem} to a finite-dimensional subspace of $V$. Let $V_h\subseteq V$ be a finite-dimensional subspace of $V$, called \textit{discrete space}, the Galerkin method has the following form:
	\begin{equation}\label{abstractproblemdiscrete}
		\text{Find } u_h\in V_h \text{ such that } \calA(u_h,v_h)=\calF(v_h) \text{ for all } v_h\in V_h.
	\end{equation}
	Note that the equation is the same, the only differences are that the discrete solution $u_h$ is searched only in $V_h$ and that the variational equation is satisfied only for test functions $v_h\in V_h$. 
	Observe that, if the bilinear form $\calA$ is coercive and continuous on $ V$ and the linear functional $ \calF$ is continuous on $V$, then they satisfy these properties also on $V_h$, since $V_h\subseteq V$.
	Hence, if the continuous problem \eqref{abstractproblem} is well-posed due to the Lax--Milgram Theorem, then also the discrete one admits a unique solution.
	We can now study the error between the exact solution and the approximate one.
	An error estimate is provided by the Céa's Lemma:
	\begin{lemma}[Céa]
		Under the assumption of the Lax--Milgram Theorem, let $u$ be the solution of \eqref{abstractproblem} and $u_h$ be the solution of  \eqref{abstractproblemdiscrete}. Then, the following error bound holds
		\begin{equation*}
			\N{u-u_h}_V\leq \frac{M}{\alpha} \inf_{v_h\in V_h}\N{u-v_h}_V.
		\end{equation*}
	\end{lemma}
	This inequality is known as the \textit{quasi-optimality inequality}.
	\subsection{Non-conforming methods}\label{nonconforming}
	Another way of approximating the problem \eqref{abstractproblem} is with non-conforming methods. We will consider discontinuous Galerkin approximations that belong to this class of schemes.
	For this part we rely on the abstract non-conforming error analysis presented in \cite[Section 1.3]{di2011mathematical}.
	In the case of non-conforming methods the discrete space $V_h$ is a finite-dimensional space that is not contained in $V$.
	The discrete  problem has this form:
	\begin{equation}\label{abstractproblemdg}
		\text{Find } u_h\in V_h \text{ such that } \calA_h(u_h,v_h)=\calF_h(v_h) \text{ for all } v_h\in V_h,
	\end{equation}
	with $V_h \not \subseteq V$, a discrete bilinear form $\calA_h: V_h\times V_h\to \mathbb{R}$ and a discrete linear form $\calF_h:V_h\to\mathbb{R}$.
	Observe that a sufficient condition for the well-posedness of the discrete problem \eqref{abstractproblemdg} is the coercivity of the bilinear form $\calA_h$ on $V_h$. 
	
	\begin{Def}[Discrete coercivity]
		Let $\vertiii{\cdot}$ be a norm defined on $V_h$.
		The discrete bilinear form $\calA_h$ enjoys \textit{discrete coercivity} if it is coercive on $V_h$, i.e., if there exists $\alpha>0$ such that
		\begin{equation*}
			\calA_h(v_h,v_h)\ge \alpha\vertiii{v_h}^2 \quad 	\forall v\in V_h.
		\end{equation*}
	\end{Def}
	In order to have an error estimate analogous to the Céa's Lemma we will need three properties satisfied: consistency, discrete coercivity and boundedness. 
	
	Consistency means that the exact solution $u$ of \eqref{abstractproblem} satisfies the discrete variational equation \eqref{abstractproblemdg}.
	To check this property we need to evaluate $\calA_h$ on the exact solution $u$ but the discrete bilinear form is defined only on $V_h\times V_h$.
	For this reason we assume that exists a subspace $V_{*}\subseteq V$ such that $u\in V_{*}$ and such that we can extend the bilinear form $\calA_h$ to $V_{*}\times V_h$.
	\begin{Def}[Consistency]
		The discrete problem \eqref{abstractproblemdg} is \textit{consistent} if the exact solution $u\in V_{*}$ satisfies
		\begin{equation*}
			\calA_h(u,v_h)=\calF_h(v_h) \text{ for all } v_h\in V_h.
		\end{equation*}
	\end{Def}
	We note that consistency is equivalent to the \textit{Galerkin orthogonality property}:
	\begin{equation*}
		\calA_h(u-u_h,v_h)=0 \text{ for all } v_h\in V_h.
	\end{equation*}
	
	The error $u-u_h$ belongs to the space $V_{*h}:=V_{*}+V_h$.
	We assume that the norm $\vertiii{\cdot}$ can be extended to $V_{*h}$, since we want to measure the error in the same norm in which the discrete coercivity holds.
	We also need the boundedness in $V_{*h}\times V_h$, but, in general, is not achievable only with the norm  $\vertiii{\cdot}$ and a stronger norm $\vertiii{\cdot}_{*}$ on $V_{*h}$ is needed.
	\begin{Def}[Boundedness]
		Let $\vertiii{\cdot}$ and $\vertiii{\cdot}_{*}$ be two norms defined on $V_{*h}$ such that $\vertiii{v}\leq \vertiii{v}_{*}$ for all $v \in V_{*h}$.
		The discrete bilinear form $\calA_h$ is \textit{bounded} in $V_{*h}\times V_h$ if there exists $M>0$ such that
		\begin{equation*}
			\abs{\calA_h(v,w_h)}\leq M\vertiii{v}_{*}\vertiii{w_h} \quad \forall (v,w_h)\in V_{*h}\times V_h.
		\end{equation*}
	\end{Def}
	\begin{theorem}[Abstract error estimate]\label{errortheorem}
		Let $u$ solve \eqref{abstractproblem} and let $u_h $ solve \eqref{abstractproblemdg}.
		Assume that $u\in V_{*}\subseteq V$ and that the discrete bilinear form $\calA_h$ can be extended to $V_{*h}\times V_h$, where $V_{*h}=V_{*}+V_h$. Let $\vertiii{\cdot}$ and $\vertiii{\cdot}_{*}$ be two norms defined on $V_{*h}$ such that $\vertiii{v}\leq \vertiii{v}_{*}$ for all $v \in V_{*h}$.
		Assume consistency, discrete coercivity and boundedness. Then, the following error estimate holds true:
		\begin{equation*}
			\vertiii{u-u_h}\leq \left(1+\frac{M}{\alpha}\right) \inf_{v_h\in V_h}\vertiii{u-v_h}_{*}.
		\end{equation*}
	\end{theorem}
	\begin{proof}
		Let $v_h\in V_h$. Owing to the triangle inequality, we get
		$$
		\vertiii{u-u_h}\leq \vertiii{u-v_h}+\vertiii{v_h-u_h}.
		$$
		For the first term we simply have
		$
		\vertiii{u-v_h}\leq \vertiii{u-v_h}_{*}.
		$
		For the second term, 
		the discrete coercivity and the consistency imply that
		$$
		\alpha \vertiii{v_h-u_h}^2\leq  
		\calA_h(v_h-u_h,v_h-u_h)= 
		\calA_h(v_h-u,v_h-u_h),
		$$
		and the boundedness yields
		$$
		\vertiii{v_h-u_h}\leq  \frac{M}{\alpha} \vertiii{v_h-u}_{*}.
		$$
		Combining the previous bounds we arrive at 
		$$
		\vertiii{u-u_h}\leq \vertiii{u-v_h}_{*}+\frac{M}{\alpha}
		\vertiii{v_h-u}_{*},
		$$
		and we conclude the proof by taking the infimum over $v_h\in V_h$.
	\end{proof}
	\subsection{Function spaces}\label{functionspaces}
	In this section we recall some preliminary concepts for presenting in Section \ref{problema} the classical variational formulation of the model problem.
	
	We denote by $\mathbb{N}$ the set of the natural numbers $\{0, 1, 2, \dots\}$ including the zero.
	Let $ d \in\IN$, $d \ge 1 $ be the space dimension and let $\Omega\subseteq\mathbb{R}^d $ be an open, bounded, Lipschitz subset of $\mathbb{R}^d$.
	
	We introduce two important classes of function spaces, namely Lebesgue and Sobolev spaces. 
	\begin{Def}[Lebesgue space]
		Let $1\leq p\leq \infty$. We set
		$$
		\N{v}_{L^p(\Omega)}:=\left(\int_{\Omega}\abs{v}^p\right)^{\frac{1}{p}} \qquad 1\leq p < \infty,
		$$
		and 
		\begin{equation*}
			\N{v}_{L^{\infty}(\Omega)}:=\esssup_{x\in\Omega} \abs{v(x)} =\inf \{ C>0 \mid \abs{v(x)} \leq C   \text{ for a.e. } x\in \Omega\}.
		\end{equation*}
		We define the \textit{Lebesgue space}
		$$
		L^p(\Omega):=\left\{v:\Omega\to\mathbb{R} \text{ Lebesgue measurable } \mid \N{v}_{L^p(\Omega)}< \infty \right\}.
		$$
	\end{Def}
	We recall that $L^p(\Omega)$ is a vector space and that $\N{\cdot}_{L^p(\Omega)}$ is a norm  for all $1\leq p\leq \infty$  \cite[Theorem 4.7]{brezis2011functional}. Moreover, $\left(L^p(\Omega),\N{\cdot}_{L^p(\Omega)}\right)$ is a Banach space for all $1\leq p\leq \infty$ \cite[Theorem 4.8]{brezis2011functional}. 
	We denote by $C^{\infty}_c(\Omega)$  the space of infinitely differentiable functions with compact support in $\Omega$, i.e.,
	$C^{\infty}_c(\Omega):=\{v\in C^{\infty}(\Omega)\mid v(\bx)=0 \quad \forall  \bx\in \Omega\setminus K, \text{ where $K \subseteq \Omega$ is compact}\}$.
	We also recall that 
	$C^{\infty}_{c}(\Omega)$ is dense in $L^p(\Omega)$ for all $1\leq p < \infty$ \cite[Corollary 4.23]{brezis2011functional}.
	In the particular case $p = 2$, $ L^2(\Omega)$ is a real Hilbert space equipped with the scalar product
	$$
	(w,v)_{L^2(\Omega)}:= \int_{\Omega}w  v.
	$$
	
	We now recall two useful inequalities: the Cauchy--Schwarz inequality and the Young's inequality.
	\begin{prop}[Cauchy--Schwarz inequality]
		For all $ w,v \in L^2(\Omega)$, then $wv\in L^1(\Omega)$ and
		\begin{equation*}
			\abs{(w,v)_{L^2(\Omega)}}\leq \N{w}_{L^2(\Omega)}  \N{v}_{L^2(\Omega)} \quad \forall w,v\in L^2(\Omega).
		\end{equation*}
	\end{prop}
	
	\begin{prop}[Young’s inequality] For all $a,b\in\mathbb{R}$
		\begin{equation*}
			ab\leq \frac{a^2}{2}+\frac{b^2}{2}. 
		\end{equation*}
	\end{prop} 
	
	\begin{Def}[Sobolev space]
		Let $m\in\mathbb{N}$ and let $1\leq p\leq \infty$. 
		We define the \textit{Sobolev space}
		\begin{equation*}
			W^{m,p}(\Omega):=	\left\{ v\in L^p(\Omega) \mid
			\begin{aligned}
				&	\forall \bi, \abs{\bi}\leq m, \exists g_{\bi}\in L^p(\Omega)  \text{ such that }\\
				&	\int_{\Omega} u D^{\bi} \phi= (-1)^{\bi} \int_{\Omega}g_{\bi} \phi \quad \forall \phi\in C^{\infty}_c(\Omega)
			\end{aligned}
			\right\},
		\end{equation*}
		where we use the  notation $\bi=(i_1,i_2,\dots,i_d)\in \mathbb{N}^d$ for multi-indices with length $\abs{\bi}=\sum_{j=1}^d i_j$ and the notation for partial derivatives $$D^{\bi}=\frac{\partial^{\abs{\bi}}\phi}{\partial  x_1^{i_1} \dots \partial x_d^{i_d} }.$$
	\end{Def}
	For a function $v\in W^{m,p}(\Omega)$, the functions $g_{\bi}\in L^p(\Omega)$ are its distributional derivatives and we denote them with  $D^{\bi}v$.
	With the norm
	$$
	\N{v}_{W^{m,p}(\Omega)}:=\sum_{0\leq \abs{\bi}\leq m} \N{D^{\bi}v}_{L^p(\Omega)},
	$$
	the space $W^{m,p}(\Omega)$ is a Banach space.
	If $ p = 2$, the space $W^{m,2}(\Omega)$ is denoted by $H^m(\Omega)$ and is a Hilbert space with the scalar product
	$$
	(w,v)_{H^{m}(\Omega)}:=\sum_{0\leq \abs{\bi}\leq m} (D^{\bi}w,D^{\bi}v)_{L^{2}(\Omega)}.
	$$
	In particular, we will focus on the space 
	$$
	H^1(\Omega):=\left\{ v\in L^2(\Omega)\mid \nabla v \in [L^p(\Omega)]^d \right\},
	$$
	with the norm 
	$$\N{\cdot}_{H^1(\Omega)}^2=\N{\cdot}_{L^2(\Omega)}^2+\abs{\cdot}_{H^1(\Omega)}^2,$$
	where $\abs{\cdot}_{H^1(\Omega)}$ denotes the seminorm $\N{\nabla v}_{L^2(\Omega)}$.
	For simplicity of notation,  we have written the norm of $L^2(\Omega)$ instead of $[L^2(\Omega)]^d$, this is done for all vector quantities.
	
	A fundamental difference between $L^p(\Omega)$ and $W^{m,p}(\Omega)$ is that the functions in $L^p(\Omega)$ do not have a trace on the boundary $\partial \Omega$.
	Instead, the trace of a function $v$ of $W^{m,p}(\Omega)$ on $\partial \Omega$ is well-defined and we denote it with the restriction $v_{|_{\partial \Omega}}$.
	We recall the kernel of the trace operator in the case $p=2$:
	$$
	H^1_0(\Omega):=\left\{v\in H^1(\Omega) \mid v_{|_{\partial\Omega}} =0\right\}.
	$$
	For a sufficiently smooth subset $\Gamma$ of $\partial \Omega$, we define
	$$
	H^1_{\Gamma}(\Omega):=\left\{v\in H^1(\Omega) \mid v_{|_{\Gamma}} =0\right\}.
	$$
	In these spaces an important inequality is valid:
	\begin{prop}[Poincaré's inequality] There is a constant $C_p>0$ such that
		\begin{equation}\label{Poincare}
			\N{v}_{L^2(\Omega)}\leq C_p  \N{\nabla v}_{L^2(\Omega)} \quad \forall v \in H^1_0(\Omega).
		\end{equation}
		This inequality holds true also for all $v\in H^1_{\Gamma}(\Omega)$.
	\end{prop}
	
	We denote the gradient of  a function $v\in W^{m,p}(\Omega)$, with $m>0$, and the divergence of  $\bw\in [W^{m,p}(\Omega)]^d$ , with $m>0$, as
	$$
	\nabla v:=\left(\frac{\partial v}{\partial x_1},\dots,\frac{\partial v}{\partial x_d} \right), \qquad \text{div}(\bw):=\sum_{i=1}^d \frac{\partial \bw_i}{\partial x_i},
	$$
	respectively.
	We define the space $H(\text{div};\Omega)$, which is relevant in the analysis:
	\begin{Def}
		\begin{equation}\label{Hdivdef}
			H(\text{div};\Omega):=\{\bm{w} \in[ L^{2}(\Omega)] ^{d} \mid \text{div}(\bm{w}) \in L^{2}(\Omega) \}
		\end{equation}
	\end{Def}
	Moreover, we recall the divergence theorem:
	\begin{theorem}[Divergence theorem]
		For all $v\in H^1(\Omega)$, for all $\bm{w}\in [H^1(\Omega)]^d$
		$$
		\int_{\Omega}\mathrm{div}\left(\bm{w}\right)v+\int_{\Omega} \bm{w}\cdot\nabla v=\int_{\partial \Omega}\left(\bm{w}\cdot \bn\right)v
		$$
		where $\bn$ is the outward unit normal vector to $\partial \Omega$.
	\end{theorem}
	
	\section{Diffusion-advection-reaction problem}\label{problema}
	Let $d \in \mathbb{N}$, $d\ge1$ be the space dimension and let $\Omega$ be an open, bounded, Lipschitz subset of $\mathbb{R}^
	d$.
	Denote by $\Gamma:=\partial\Omega$ the boundary of the domain $\Omega$.
	
	We define the diffusion-advection-reaction operator $\calL$ applied to a function  $v:\Omega\to \mathbb{R}$ as
	\begin{equation*}
		\mathcal{L}v:=	\text{div}\left(-\bm{K}  \nabla v  +\bm{\beta}  v \right) +\sigma v,
	\end{equation*}
	with coefficients $\bk:\Omega\to \mathbb{R}^{d\times d}$, $\bbeta:\Omega\to \mathbb{R}^d$ and $\sigma:\Omega\to \mathbb{R}$.
	$\calL$ is a linear partial differential operator of the second order in which we can distinguish three terms: $\text{div}\left(-\bm{K}  \nabla v\right)$ is the diffusion term, $\text{div}\left(\bm{\beta}  v \right) $  is the advection term and  $\sigma v$  is the reaction term.
	
	Let $\Gamma_D$ and $\Gamma_N$ be sufficiently regular subsets of the boundary such that $\Gamma_D\neq\emptyset$, $\Gamma=\Gamma_D \cup \Gamma_N$ and  $\Gamma_D\cap\Gamma_N=\emptyset$. Dirichlet and Neumann boundary conditions are imposed on $\Gamma_D$ and $\Gamma_N$, respectively.
	Let $\bn(\bx)$ be the outward unit normal vector to the boundary at $\bx\in\Gamma$.
	
	Let $f\in L^2(\Omega)$, $g_D\in H^{\frac{1}{2}}(\Gamma_{D})$ and $g_N\in L^2(\Gamma_N)$,
	we consider the following boundary value problem for the diffusion-advection-reaction equation:
	\begin{alignat}{2}
		\text{div}(-\bm{K} \nabla u +\bm{\beta} u) +\sigma u &= f  && \quad \text{in  $\Omega$,}\label{eq}\\
		u&=g_{D} && \quad\text{on $ \Gamma_{D}$,} \label{Dirichlet}\\
		- \bm{K} \nabla u  \cdot \mathbf{n} &=g_{N} && \quad \text{on $ \Gamma_{N}$.}\label{Neumann}
	\end{alignat}
	We make the following assumptions on the data:
	\begin{equation}\label{assumptionbeta}
		\bk\in\left[L^{\infty}(\Omega)\right]^{d\times d},\quad \bbeta\in \left[W^{1,\infty}(\Omega)\right]^{d},\quad \sigma \in L^{\infty}(\Omega).
	\end{equation}
	In particular, this implies the regularity of $\bbeta$:
	$ 	\bbeta\in \left[H^1(\Omega)\right]^d$ and  $\bbeta\in H(\mathrm{div};\Omega).$
	We also assume that the following condition of ellipticity is satisfied.
	\begin{Def}[Ellipticity condition]
		The \textit{ellipticity condition} yields if there exists a constant $k_{min}>0$ such that
		\begin{equation}\label{ellipticity}
			\bm{\xi}^{T}\bk(\bx) \bm{\xi} \ge k_{min} \N{\bm{\xi}}^2\quad \forall \bm{\xi}\in \mathbb{R}^d, \forall \bx\in \Omega,
		\end{equation}
		where $ \N{\cdot}$ denotes the Euclidean norm in $\mathbb{R}^d$.
	\end{Def}
	Choosing $\bm{\xi}=(1,0,\dots,0)$, we deduce from the ellipticity condition \eqref{ellipticity}
	that 
	\begin{equation}\label{assalg}
		\bk_{11}(\bx)\ge k_{min}>0\quad  \text{for all  $\bx\in\Omega$}.
	\end{equation}
	We observe that, if the diffusion term vanishes, then the ellipticity condition is not satisfied, so this case is not included in our analysis.
	Moreover, we make the following assumption: if at least one between $\bbeta$ and $\sigma$ is not null, then there exists a constant $\sigma_0>0$ such that 
	\begin{equation}\label{assumptiondiv}
		\sigma(\bx)+\frac{1}{2}\text{div}\left(\bbeta(\bx)\right)\ge \sigma_{0}  \quad \text{ a.e. } \bx \in\Omega.
	\end{equation}
	Under these hypothesis is always present the diffusion term while the advection and reaction terms can both vanish,
	but if we have at least one of them, they need to satisfy the assumption \eqref{assumptiondiv}.
	
	We also distinguish between the inflow and outflow parts of the boundary $\Gamma$ defined as 
	\begin{equation}\label{inflowoutflow}
		\Gamma_{-}:=\{\bx\in\Gamma \mid \bm{\beta}(\bx) \cdot \mathbf{n}(\bx)<0 \}, \quad
		\Gamma_{+}:=\{\bx\in\Gamma \mid \bm{\beta}(\bx) \cdot \mathbf{n}(\bx)\geq0 \},
	\end{equation}
	respectively.
	We assume that, 
	when the advection term $\bbeta$ is present, $\bbeta\cdot \bn\ge 0$ on $\Gamma_N$ and that $\bbeta\cdot \bn < 0$ on $\Gamma_D$.
	We are assuming
	\begin{equation}\label{gammaDgammaN}
		\Gamma_D=\Gamma_{-},\qquad  \Gamma_N=\Gamma_{+}.
	\end{equation}
	This is done for simplicity but it it also physically reasonable, for example, if we want to model the movement of a substance knowing its concentration at the entrance but  not at the exit.
	
	We now consider the classical weak variational formulation of the model problem.
	Assume that  for all $ g_D \in H^{\frac{1}{2}}(\Gamma_D)$
	there exists an extension $\bar{g}_D \in H^{\frac{1}{2}}(\partial \Omega)$ such that $\bar{g_D}_{|_{\Gamma_D}}=g_D$ and that there is a constant $C_1>0$ such that 
	$\N{\bar{g}_D}_{ H^{\frac{1}{2}} (\partial \Omega)} \leq C_1 \N{g_D}_{ H^{\frac{1}{2} }(\Gamma_D)}$.
	Then, we use the lifting of $\bar{g}_D$ in $H^1(\Omega)$, i.e., a function $u_{g_D} \in H^1(\Omega)$ such that $u_{g_D} = \bar{g}_D$ on $\partial \Omega$ and that
	$\N{u_{g_D}}_{ H^1 ( \Omega)} \leq C_2 \N{\bar{g}_D}_{ H^{\frac{1}{2} }(\partial \Omega)}$ for some constant $C_2>0$.
	Hence, we have a function $u_{g_D}$ in $H^1(\Omega)$ such that $u_{g_D} = g_D$ on $\Gamma_D$ and that
	$\N{u_{g_D}}_{ H^1 ( \Omega)} \leq C \N{g_D}_{ H^{\frac{1}{2} }(\Gamma_D)}$ for some constant $C>0$.
	
	The classical weak variational formulation of the problem \eqref{eq}-\eqref{Neumann} is:
	\begin{equation}\label{variationalcontinuous}
		\begin{cases}
			\text{Find } u\in H^1(\Omega) \text{ such that}\\
			u=u_{g_D}+u_0, \quad u_0\in H^1_{\Gamma_D}(\Omega),\\
			a(u_0,v)=L(v)-a(u_{g_D},v) \text{ for all } v\in H^1_{\Gamma_D}(\Omega),
		\end{cases}
	\end{equation}
	with the bilinear form 
	$$
	a(u,v):=\int_{\Omega}\bm{K} \nabla u \cdot \nabla v +\int_{\Omega} 	\left(\bbeta \cdot \nabla u\right)v+\int_{\Omega}\left(\sigma +\text{div}(\bm{\beta})\right)u v$$
	and the linear functional 
	$$
	L(v):= \int_{\Omega} f v
	-\int_{\Gamma_N} g_N \cdot \bn v.
	$$
	
	We prove the equivalence between the model problem \eqref{eq}–\eqref{Neumann} and the variational formulation \eqref{variationalcontinuous} in the following proposition.
	\begin{prop}
		If $u$ is a weak solution of the problem \eqref{eq}–\eqref{Neumann}, then $u$ satisfies the variational problem \eqref{variationalcontinuous}. Conversely, if $u$ satisfies  \eqref{variationalcontinuous}, then $u$ solves the problem \eqref{eq}–\eqref{Neumann}.
	\end{prop}
	\begin{proof}
		First, we prove that if $u$ is solution of \eqref{eq}–\eqref{Neumann}, then it also solves  \eqref{variationalcontinuous}.
		Multiply \eqref{eq} by a test function $v\in H^{1}_{\Gamma_D}(\Omega)$ and integrate over the domain $\Omega$:
		$$
		\int_{\Omega}	\text{div}(-\bm{K} \nabla u) v +\int_{\Omega} \left(	\bbeta \cdot \nabla u\right)v+\int_{\Omega}\left(\sigma +\text{div}(\bm{\beta})\right)u v= \int_{\Omega} f v.
		$$
		Integrating by parts the first term, we get
		$$
		\int_{\Omega}\bm{K} \nabla u \cdot \nabla v -\int_{\partial \Omega } \bm{K} \nabla u \cdot \bn v  +\int_{\Omega} \left(	\bbeta \cdot \nabla u\right)v+\int_{\Omega}\left(\sigma +\text{div}(\bm{\beta})\right)u v= \int_{\Omega} f v.
		$$
		Owing to the fact that $v=0$ on $\Gamma_D$ and to the Neumann boundary condition \eqref{Neumann} we arrive at
		$$
		\int_{\Omega}\bm{K} \nabla u \cdot \nabla v +\int_{\Omega} \left(	\bbeta \cdot \nabla u\right)v+\int_{\Omega}\left(\sigma +\text{div}(\bm{\beta})\right)u v= \int_{\Omega} f v
		-\int_{\Gamma_N} g_N \cdot \bn v.
		$$
		We have reached $a(u,v)=L(v)$ and we obtain the thesis  separating $u_0$ and $u_{g_D}$.
		
		Conversely, let $u$ be a solution to  \eqref{variationalcontinuous}. Consider $\phi\in C^{\infty}_c(\Omega)$ as a test function, we obtain
		$$	\int_{\Omega}\left(	\text{div}(-\bm{K} \nabla u +\bm{\beta} u) +\sigma u \right)\phi 
		= \int_{\Omega} f \phi.$$
		This implies that $\calL u = f$ in $L^2(\Omega)$, since $C^{\infty}_c(\Omega)$ is dense in $L^2(\Omega)$.
		Hence, the equation \eqref{eq} holds a.e. in $\Omega$.
		The Dirichlet boundary condition \eqref{Dirichlet} follows from the fact that $u_0=0$ on $\Gamma_D$ since $u_0\in H^1_{\Gamma_D}(\Omega)$ and $u_{g_D}=g_D$ on $\Gamma_D$ for the lifting.
		Moreover, let $\phi\in H^1_{\Gamma_D}(\Omega)$, now that we have $\calL u=f$ a.e. in $\Omega$, we obtain
		$$	-\int_{\partial \Omega}\bk\nabla u\cdot \bn \phi 
		= -\int_{\Gamma_N} g_N \cdot \bn \phi.$$
		From this equality we recover the Neumann boundary condition \eqref{Neumann}.
	\end{proof}			
	
	We prove that the bilinear form $a$ is coercive and continuous on $V=H^1_{\Gamma_D}(\Omega)$ with respect to the $\N{\cdot}_{H^1(\Omega)}$-norm and that the linear functional  $L$ is continuous on $V=H^1_{\Gamma_D}(\Omega)$ with respect to the $\N{\cdot}_{H^1(\Omega)}$-norm.
	Hence, for the Lax--Milgram Theorem the problem  \eqref{variationalcontinuous} is well-posed.
	
	For all $w, v\in H^1_{\Gamma_D}(\Omega)$, the Cauchy--Schwarz inequality yields
	\begin{align*}
		\abs{a(w,v)}\leq& 	\N{\bk}_{L^{\infty}(\Omega)}   \N{\nabla w}_{L^2(\Omega)}     \N{\nabla v }_{L^2(\Omega)} +\N{\bbeta}_{L^{\infty}(\Omega)}  \N{\nabla w}_{L^2(\Omega)}   \N{ v }_{L^2(\Omega)} \\ &
		+\left(\N{\sigma}_{L^{\infty}(\Omega)} + \N{\text{div}(\bbeta)}_{L^{\infty}(\Omega)}  \right) \N{w}_{L^2(\Omega)}     \N{v }_{L^2(\Omega)} \\
		\leq &\left(\N{\bk}_{L^{\infty}(\Omega)} +\N{\bbeta}_{L^{\infty}(\Omega)}
		+\N{\sigma}_{L^{\infty}(\Omega)} + \N{\text{div}(\bbeta)}_{L^{\infty}(\Omega)}  \right) \N{w}_{H^1(\Omega)}     \N{v }_{H^1(\Omega)}.
	\end{align*}
	The bilinear form $a$ is continuous with the constant $M=	\N{\bk}_{L^{\infty}(\Omega)} +\N{\bbeta}_{L^{\infty}(\Omega)}
	+\N{\sigma}_{L^{\infty}(\Omega)} + \N{\text{div}(\bbeta)}_{L^{\infty}(\Omega)}  $.
	
	For the coercivity, we have to consider the quantity  $
	a(v,v)=\int_{\Omega}\bm{K} \nabla v\cdot \nabla v + \int_{\Omega} 	\left(\bbeta \cdot \nabla v \right)v+\int_{\Omega}\left(\sigma +\text{div}(\bm{\beta})\right)v^2$.
	For the first term $	\int_{\Omega}\bm{K} \nabla v\cdot \nabla v$,
	owing to the ellipticity condition \eqref{ellipticity} and to the Poincaré's inequality \eqref{Poincare}, we have
	\begin{align*}
		\int_{\Omega}\bm{K} \nabla v\cdot \nabla v
		&\ge k_{min}   \N{\nabla v}_{L^2(\Omega)}^2   
		= \frac{k_{min}}{2}   \N{\nabla v}_{L^2(\Omega)}^2  +\frac{k_{min}}{2}   \N{\nabla v}_{L^2(\Omega)}^2 \\&
		\ge \frac{k_{min}}{2}   \N{\nabla v}_{L^2(\Omega)}^2  +\frac{k_{min}}{2 C_p^2}   \N{ v}_{L^2(\Omega)}^2 \\&
		\ge\frac{k_{min}}{2} \min{\biggl\{1 ,\frac{1}{C_p^2}\biggl\}}  \N{ v}_{H^1(\Omega)}^2 .
	\end{align*}
	The second term can be rewritten using integration by parts as 
	$$
	\int_{\Omega} 	(\bbeta \cdot \nabla v)v=\int_{\Omega} 	\bbeta \cdot \left(\nabla \frac{v^2}{2}\right)=
	- \int_{\Omega} 	\text{div}(\bbeta) \frac{v^2}{2}+\int_{\partial \Omega}\bbeta \cdot \bn \frac{v^2}{2}.
	$$
	Using the hypothesis \eqref{assumptiondiv}, the fact that $v=0$ on $\Gamma_D$ and that $\bbeta\cdot \bn\ge0$ on $\Gamma_N$ \eqref{gammaDgammaN}:
	$$
	\int_{\partial \Omega}\bbeta \cdot \bn \frac{v^2}{2}+\int_{\Omega}\left(\sigma +\frac{\text{div}(\bm{\beta})}{2}\right)v^2\ge 0+ \sigma_0 \N{v}_{L^2(\Omega)}^2\ge0.
	$$
	The bilinear form $a$ is coercive  in $H^1_{\Gamma_D}(\Omega)$ with the constant $\alpha=\frac{k_{min}}{2} \min{\big\{1 ,\frac{1}{C_p^2}\big\}}$.
	
	Owing to Cauchy--Schwarz inequality and to the continuity of the Neumann trace we obtain the continuity of the linear functional $L$:
	$$
	\abs{ L(v)}\leq \N{f}_{L^2(\Omega)} \N{v}_{L^2(\Omega)}+ \N{g_N}_{L^2(\Gamma_N)}\N{v}_{L^2(\Gamma_N)}\leq \left(\N{f}_{L^2(\Omega)} + C\N{g_N}_{L^2(\Gamma_N)}\right) \N{v}_{H^1(\Omega)}.
	$$
	As a consequence of the inequalities that we have shown and Lax--Milgram Theorem, the variational problem \eqref{variationalcontinuous} admits a unique solution.
	
	\chapter{The discrete setting}\label{Chapter2}
	In this chapter we introduce the discrete setting for the discontinuous Galerkin methods.
	In Section \ref{sec:domain} we fix the notation, give some definitions and list the assumptions on the mesh of the domain $\Omega$.
	In Section \ref{DGtoools} we define the typical tools of DG methods such as jumps, averages and broken functional spaces together with important properties that will be  used later in the analysis.
	We follow \cite[Chapter 1]{di2011mathematical}.
	
	\section{Domain and mesh sequences}\label{sec:domain}
	We assume that the domain $\Omega$ is an open, bounded, Lipschitz polytope in $\mathbb{R}^d$, with $d \in \mathbb{N}$, $d\ge1$.
	
	In order to define  the polytopes, which are a generalization of 2D polygons and 3D polyhedrons, we proceed by induction on the dimension $d$.
	In one dimension a polytope is an open, bounded segment. In two dimensions is an open, bounded, connected and Lipschitz subset of $\mathbb{R}^2$ such that its boundary is a finite union of closures of one dimensional polytopes.
	In three dimensions a polytope is an open, bounded, connected and  Lipschitz subset of $\mathbb{R}^3$ such that its boundary is a finite union of closures of two dimensional polytopes.
	\begin{Def}[Polytope, facet]
		A $0$-dimensional polytope is a subset of $\mathbb{R}^d$ containing a single point.
		Let	$ n\in \mathbb{N}$, $1\leq n\leq d$.
		A \textit{$n$-dimensional polytope} of $\mathbb{R}^d$ is a relatively open, bounded, connected and Lipschitz subset of a $n$-dimensional affine subspace of $\mathbb{R}^d$ such that its relative boundary is a finite union of \textit{$(n-1)$-facets}, i.e., closures of $(n-1)$-dimensional polytopes.
	\end{Def}
	
	We discretize the domain $\Omega$ using a polytopic mesh.
	\begin{Def}[Polytopic mesh]
		A \textit{polytopic mesh} $\calT$ of the domain $\Omega$ is a finite family of disjoint  $d$-dimensional polytopes  $\calT=\{T\}_{T\in\calT}$ such that
		$$
		\bar{\Omega}=\bigcup_{T\in\calT}\bar{T}.
		$$
		Each $T\in\calT$ is called a \textit{mesh element}.
	\end{Def}
	The advantage of assuming that $\Omega $ is a polytope is that it can be exactly covered by a mesh consisting of polytopic elements.
	We work with a conforming mesh.
	\begin{Def}[Conforming mesh]
		$\calT$ is a \textit{conforming mesh} if, for all $T, T'\in\calT$, $T\neq T'$, the intersection $\partial T\cap\partial T'$ is either empty or a common $n$-dimensional facet with $n\leq d-1$.
	\end{Def}
	We now fix the notation and give some definitions about the mesh.
	\begin{Def}[Element diameter, meshsize]\label{meshsizEE}
		Let $\calT$ be a mesh of the domain $\Omega$. For all $T\in\calT$, the \textit{diameter} of the element $T$ is defined by 
		$$
		h_T := \sup_{\bx,\by\in T} \abs{\bx-\by},
		$$
		and the \textit{meshsize} $h$ is defined as  
		$$
		h:= \sup_{T\in\calT}h_T.
		$$
	\end{Def}
	In order to  analyse the convergence of the discontinuous Galerkin method as the meshsize $h$ goes to zero, we will denote by $\calT_h$ a mesh with meshsize $h$ and we will consider a mesh sequence $\calT_{\mathcal{H}}:=\{\calT_h\}_{h\in\mathcal{H}}$ where $\mathcal{H}$ is a countable subset of $\{h\in\mathbb{R}\mid h>0\}$ having $0$ as only accumulation point.
	
	\begin{Def}[Radius, barycentre, measure]
		For any element $ T\in\calT_h$, we denote by $\rho_T$ the \textit{radius} of the largest ball inscribed in $T$, by $\bx_T$ the \textit{barycentre} of the element $T$ and by $\abs{T}$ its $d$-dimensional \textit{measure}.
		$\partial T $ denotes the boundary of $T$ which is the union of $ (d - 1)$-dimensional facets
		of $T$ and $\abs{\partial T}$ denotes its $(d-1)$-dimensional measure.
	\end{Def}
	
	\begin{Def}[Element outward normal]
		Let $T\in\calT_h$, we define $\bn_T$ on  $\partial T$ as the unit \textit{outward normal} vector to the element $T$.
	\end{Def}
	
	We now fix the notation about the mesh facets of a subdivision $\calT_h$.
	\begin{Def}[Mesh facets]
		A \textit{mesh facet} is a facet of a polytopic mesh element $T\in\calT_h$, i.e., the closure of a $(d-1)$-dimensional polytope that form the boundary $\partial T$.
		We denote by $\calE_h$ the set of all facets of  $\calT_h$ and by $e\in\calE_h$ a mesh facet.
	\end{Def}
	We distinguish between \textit{interior facets} for which there exist two elements $T_1$ and $T_2$ such that $e=\partial T_1\cap \partial T_2$ and \textit{boundary facets} for which there exists an element $T$ such that $e\subseteq \partial T\cap \partial \Omega$.
	The sets of interior and boundary facets are denoted by $\calE_h^I$ and $\calE_h^B$, respectively.
	We assume that it is possible to collect the boundary facets where Dirichlet conditions are assigned in a set, denoted $\calE_h^D$, and the boundary facets where Neumann conditions are assigned in a set, denoted  $\calE_h^N$.
	
	\begin{Def}[Facets diameter, measure, unit normal]
		For all $e\in\calE_h$, we denote by	$h_e$ the \textit{diameter} of the facet $e$ and by $\abs{e}$ its $(d-1)$-dimensional \textit{measure}.
		We associate a \textit{unit normal} vector $\bn_e$ to each facet $e$. If $e\in \calE_h^B$ then $\bn_e$ coincides with the unit outward vector $\bn$ normal to $\partial \Omega$.
	\end{Def}
	The following analysis does not depend on the choice of $\bn_e.$
	
	\begin{Def}\label{Npartial}
		For each element $T\in\calT_h$ we define the set of all its facets as 
		$$
		\calF_T:=\{e\in\calE_h\mid e\subseteq\partial T\}.
		$$
		The maximum number of mesh facets composing the boundary of mesh elements is indicated by
		$$
		N_{\partial}:=\max_{T\in\calT_h}card(\calF_T),
		$$
		where $card$ denotes the cardinality of the set.
	\end{Def}
	
	We now define three properties that  a mesh sequence $\calT_{\mathcal{H}}$ can have. 
	
	\begin{Def}[Shape regularity]
		The mesh sequence $\calT_\mathcal{H}$ is \textit{shape-regular} if, for all $h\in\mathcal{H}$, there exists $C_{sr}>0$, independent of $h$, such that, for all $T\in\mathcal{T}_h$,
		\begin{equation}\label{shaperegularity}
			h_T\leq C_{sr} \rho_T.
		\end{equation}
	\end{Def}
	\begin{Def}[Graded mesh]
		The mesh sequence $\calT_{\mathcal{H}}$ is \textit{graded} if, for all $h\in\mathcal{H}$, there exists $C_{g}>0$, independent of $h$, such that, for all $ T\in\calT_h $ and for all $e\in\calF_T$,
		\begin{equation}\label{gradedmesh}
			C_{g}h_e\ge h_T.
		\end{equation}
	\end{Def}
	\begin{Def}[Star-shaped property]\label{starshaped}
		The \textit{star-shaped property} holds for a mesh sequence  $\calT_\mathcal{H}$ if, for all $h\in\mathcal{H}$, there exists $0< r\leq\frac{1}{2}$, independent of $h$, such that, each $ T\in\calT_h$  is star-shaped with respect to a ball centred at some $\bx\in T$ and with radius $r h_T$.
	\end{Def}
	
	We assume to work with mesh sequences that are shape-regular, graded and with the star-shaped property. 
	We will use these assumptions later in the analysis, together with the following one, which is a condition on the “chunkiness” of the mesh, similar to \cite[eq. 26]{imbert2023space}.
	
	We assume that, for all $h\in\mathcal{H}$, there exists a constant $\eta>0$, independent of $h$, such that, for all $T\in\calT_h$,
	\begin{equation}\label{cuboidalmesh}
		h_T\abs{\partial T}\leq \eta \abs{T}.
	\end{equation}
	\begin{rem}
		We prove the above inequality in the case $d=2$ for shape-regular meshes with triangular elements.
		Let $T\in\calT_h$, we observe that $\abs{\partial T}h_T\leq 3 h_T^2$, since $h_T$ is the diameter of $T$ and $\abs{e}\leq h_T$ for all $e\in\calF_T$.
		Moreover, for the shape-regularity assumption \eqref{shaperegularity} there exists $C_{sr}>0$, independent of $h$, such that, for all $T\in\mathcal{T}_h$,
		$h_T\leq C_{sr}\rho_T,$
		implying that $\abs{\partial T}h_T\leq 3 h_T^2\leq  3 C_{sr}^2 \rho_T^2$.
		Finally, we note that $\pi \rho_T^2\leq \abs{T}$, since the area of the triangle $T$ is greater than the area of the largest circle inscribed, so we obtain $\abs{\partial T}h_T\leq \eta \abs{T}$ with 
		$\eta=\frac{3}{\pi}C_{sr}^2$.
		
		We also observe that the inequality \eqref{cuboidalmesh} holds for shape-regular meshes with convex elements. This is a consequence of the fact that the measure of the boundary of a convex element is smaller than the measure of the boundary of any ball circumscribed, due to convexity.
	\end{rem}
	
	\section{Discontinuous Galerkin tools}\label{DGtoools}
	We now introduce the typical tools of a discontinuous Galerkin approximation such as jumps and averages of scalar and vector-valued functions across the facets of a mesh $\calT_h$. 
	
	Given $T\in\calT_h$, for all $\varphi\in H^1(T)$ the trace of $\varphi$ along the boundary $\partial T$ is well-defined and we denote it by $\varphi_{|_T}$.
	Let $e\in \calE_h^I$ be an interior facet shared by the elements $T_{1}$ and $T_{2}$.
	If we consider a function $\varphi$ in $H^1(T_1\cup T_2)$, then along $e$ there are two traces of $\varphi$: $\varphi_{|_{T_1}}$ and $\varphi_{|_{T_2}}$. We can add or subtract those values and we obtain an average and a jump for $\varphi$.
	Recall that the unit vector $\bn_{T_1}$ normal to $e$ points exterior to $T_1$ and  that the unit vector $\bn_{T_2}$ normal to $e$ points exterior to $T_2$, so $\bn_{T_1}=-\bn_{T_2}$.
	
	\begin{Def}[Averages and jumps of a scalar function]
		For all $e\in\calE_h^I$ with $e=\partial T_1 \cap \partial T_2$ and for all $\varphi\in H^1(T_1\cup T_2)$, we define the \textit{average} of $\varphi$ on $e$ as 
		\begin{equation*}
			\mvl{\varphi}:=\frac{\varphi_{|_{T_1}}+\varphi_{|_{T_2}}}2,
		\end{equation*}
		and the \textit{jump} of $\varphi$ on $e$ as 
		\begin{equation*}
			\jmp{\varphi}:= \varphi_{|_{T_1}}\bn_{T_1}+\varphi_{|_{T_2}} \bn_{T_2} .
		\end{equation*}
		For all $e\in\calE_h^B$ with $e\subseteq\partial T \cap \partial \Omega$ and for all $\varphi\in H^1(T)$, we define the \textit{average} and the \textit{jump} of $\varphi$ on  $e$ as 
		\begin{equation*}
			\mvl{\varphi} :=\varphi_{|_{T}}, \quad \jmp{\varphi}:=\varphi_{|_{T}}\bn_T.
		\end{equation*}
	\end{Def} 
	
	\begin{Def}[Averages and jumps of a vector-valued function]
		For all $e\in\calE_h^I$ with $e=\partial T_1 \cap \partial T_2$ and for all $\bw \in [H^1(T_1\cup T_2)]^d$, we define the \textit{average} of $\bw$ on $e$ as 
		\begin{equation*}
			\mvl{\bw}:=\frac{\bw_{|_{T_1}}+\bw_{|_{T_2}}}2,
		\end{equation*}
		and the \textit{jump} of $\bw$ on  $e$ as 
		\begin{equation*}
			\jmp{\bw}:= \bm{w}_{|_{T_1}}\cdot\bn_{T_1}+\bw_{|_{T_2}} \cdot\bn_{T_2}.
		\end{equation*}
		For all $e\in\calE_h^B$ with $e\subseteq\partial T \cap \partial \Omega$ and for all $\bw\in [H^1(T)]^d$, we define the \textit{average} and the \textit{jump} of $\bm{w}$ on $e$ as 
		\begin{equation*}
			\mvl{\bw} :=\bw_{|_{T}}, \quad \jmp{\bw}:=\bw_{|_{T}}\cdot \bn_T.
		\end{equation*}
	\end{Def} 
	
	Notice that the jump $\jmp{\varphi}$ of a scalar function $\varphi$ is a vector while the jump $\jmp{\bw}$ of a vector-valued function $\bw$ is a scalar. The advantage of these definitions is that they do not depend on the ordering assigned to the elements $T_i$. 
	
	For simplicity, on an interior facet $e=\partial T_1\cap \partial T_2$ we use the notation $\varphi_i=\varphi_{|_{T_i}}$, $\bm{w}_i=\bw_{|_{T_i}}$ and  $\bn_i=\bn_{T_i} $ for $ i\in \{1, 2\}$ and on a boundary facet $e=\partial T\cap \partial \Omega$ we use $\varphi=\varphi_{|_{T}}$, $\bm{w}^{}=\bw_{|_{T}}$ and  $\bn=\bn_{T}=\bn_e$.
	
	A classical tool in discontinuous Galerkin methods to deal with the advection term is the use of upwind numerical fluxes. We introduce them and some of their useful properties.
	
	For each element $T$ we define the inflow and outflow parts of $\partial T$ as
	$$	\partial T_{-}:=\{\bx\in\partial T \mid \bm{\beta}(\bx) \cdot \bn_T(\bx) < 0 \}, \quad 	\partial T_{+}:=\{\bx\in\partial T \mid \bm{\beta}(\bx) \cdot \bn_T(\bx) > 0 \},$$
	respectively,
	and the neutral part as
	$$
	\partial T_{0}:=\{\bx\in\partial T \mid \bm{\beta}(\bx) \cdot \bn_T(\bx) = 0 \}.$$
	For each $e\in\calE_h$, we make the following assumption:
	$\bbeta(\bx) \cdot \bn_e(\bx)$ has the same sign for all $\bx\in e$. Hence, for all $T\in\calT_h$ and for all $e\in\calF_T$ we have that $e\subseteq \partial T_{-}$ or $e\subseteq \partial T_{+}$ or $e\subseteq \partial T_{0}$.
	For all $T\in\calT_h$, instead of considering the arithmetic average, we can consider the upwind idea.
	This means that for all function $\bm{w}\in [H^{1}(\calT_h)]^d$ on $ \partial T_{-}$ we consider only the trace from outside, on $ \partial T_{+}$ the trace from inside and on $ \partial T_{0}$ we can choose any value directed as $\bm{w}$.  
	\begin{Def}[Upwind numerical flux]
		For all $e\in\calE_h^I$ with $e=\partial T_1 \cap \partial T_2$ and for all $\bw \in [H^1(T_1\cup T_2)]^d$, we define the 
		\textit{upwind numerical flux} of $\bw$ on $e$ as 
		\begin{equation}\label{upwind}
			\mvl{\bw}_{\text{upw}}: = \begin{cases}
				\bw_2 & \text{if $\bbeta \cdot \bn_1<0$,}\\
				\bw_1 & \text{if $\bbeta \cdot \bn_1>0$,}\\ 
				\mvl{\bw} & \text{if $\bbeta \cdot \bn_1=0$.}
			\end{cases}
		\end{equation}
	For all $e\in\calE_h^B$ with $e\subseteq\partial T \cap \partial \Omega$ and for all $\bw\in [H^1(T)]^d$, we define the \textit{upwind numerical flux} of $\bm{w}$ on $e$ as 
		$$	\mvl{\bw}_{\text{upw}}= \bw_{|_T}.$$
	\end{Def}
	We can define a more general average:
	\begin{Def}[Weighted averages]
		For all $e\in\calE_h^I$ with $e=\partial T_1 \cap \partial T_2$ and for all $\bm{w}\in [H^1(T_1\cup T_2)]^d$, we define the \textit{weighted average} of $\bm{w}$ on  $e$ as 
	\begin{equation*}
			\mvl{\bm{w}}_{\alpha}:=\alpha_1\bm{w}_{|_{T_1}}+\alpha_2\bm{w}_{|_{T_2}},
		\end{equation*}
		with $\alpha_1+\alpha_2=1$, $\alpha_1, \alpha_2\ge0$.
		
		For all $e\in\calE_h^B$ with $e\subseteq\partial T \cap \partial \Omega$ and for all $\bm{w}\in[ H^1(T)]^d$, we define the \textit{weighted average} of $\bm{w}$ on $e$ as 
		\begin{equation*}
			\mvl{\bm{w}}_{\alpha} :=\bm{w}_{|_{T}}.
		\end{equation*}
	\end{Def} 
	We note that the arithmetic average is obtained for $\alpha_1=\alpha_2=\frac{1}{2}$ while the upwind flux is obtained when $\alpha_i = \frac{(sign(\bbeta\cdot \bn_i)+1)}{2}$ for $ i = 1,2$.
	\begin{prop}
		For all $e\in\calE_h^I$ and for all $\varphi\in H^1(\calT_h)$,
		\begin{equation}\label{fluxesusefulid}
			\mvl{	\bbeta \varphi}_{\mathrm{upw}} \cdot \bn_e=\mvl{\bbeta \varphi} \cdot \bn_e +\frac{1}{2}\abs{\bbeta\cdot\bn_e} \jmp{\varphi}\cdot \bn_e.
		\end{equation}
	\end{prop}
	\begin{proof}
		Using the definitions of jumps and averages and the continuity of $\bbeta$ \eqref{assumptionbeta}, the right-hand side is equal to
		$$
		\frac{1}{2} (\varphi_1+  \varphi_2) \bbeta \cdot \bn_e +\frac{1}{2}\abs{\bbeta\cdot\bn_e} \left(\varphi_1\bn_1+  \varphi_2\bn_2\right) \cdot \bn_e.
		$$
		If $\bn_e=\bn_1$ we obtain
		$$
		\frac{1}{2} (\varphi_1+  \varphi_2) \bbeta \cdot \bn_e +\frac{1}{2}\abs{\bbeta\cdot\bn_e} \left(\varphi_1-  \varphi_2\right) .
		$$
		and coincides with $\mvl{	\bbeta \varphi}_{\text{upw}} \cdot \bn_e$, using the definition of the upwind flux.
	\end{proof}
	\begin{prop}
		For all $e\in\calE_h^I$ and for all $\varphi\in H^1(\calT_h)$,
		\begin{equation}\label{fluxesid}
			\frac{1}{2}\mvl{\bbeta} \cdot \jmp{\varphi^2}=\mvl{\bbeta \varphi} \cdot \jmp{\varphi}.
		\end{equation}
	\end{prop}
	\begin{proof}
		Using the definitions of jumps and averages and the continuity of $\bbeta$ \eqref{assumptionbeta} we can rewrite the assertion as
		\begin{equation*}
			\frac{1}{2} \bbeta\cdot (\varphi_1^2\bn_1+\varphi_{2}^2\bn_2)=	\frac{1}{2} (\bbeta \varphi_1+\bbeta \varphi_{2})\cdot (\varphi_{1}\bn_1+\varphi_{2}\bn_2).
		\end{equation*}
		Since $\bn_1=-\bn_2$ the two quantities coincide.
	\end{proof}
	
	We now introduce broken Sobolev spaces and  the broken version of the space $H(\text{div};\Omega)$ defined in \eqref{Hdivdef}.
	In Lemma \ref{Wsalti} and in Lemma \ref{Hdivsalti} we state their main properties, that we will widely use in the analysis.
	\begin{Def}[Broken Sobolev spaces]
		For  $ m\in\mathbb{N}$ and $1\leq p\leq \infty$ we define the \textit{broken Sobolev spaces}
		\begin{align*}
			W^{m,p}(\mathcal{T}_{h}):=&\{\varphi\in L^{p}(\Omega) \mid \varphi_{|_T}\in W^{m,p}(T) \quad \forall T \in \mathcal{T}_{h} \},\\
			H^{m}(\mathcal{T}_{h}):=&\{\varphi\in L^{2}(\Omega) \mid \varphi_{|_T} \in H^{m}(T) \quad \forall T \in \mathcal{T}_{h} \}.
		\end{align*}
	\end{Def}
	\begin{lemma}[Characterization of $W^{1,p}(\Omega)$]\label{Wsalti}
		Let $ 1\leq p\leq +\infty$ and let $\varphi$ be a function in  $W^{1,p}(\mathcal{T}_{h})$, then
		\begin{equation*}
			\varphi\in W^{1,p}(\Omega) \iff \jmp{\varphi}=0 \quad  \forall e\in \mathcal{E}_{h}^{I}.
		\end{equation*}
	\end{lemma}
	For the proof see \cite[Lemma 1.23]{di2011mathematical}.
	\begin{Def}[Broken version of $H(\text{div};\Omega)$]
		We define the \textit{broken version of  $H(\mathrm{div};\Omega)$} 
		\begin{equation*}
			H(\text{div};\mathcal{T}_{h}):=\{\bm{w}  \in[ L^{2}(\Omega)] ^{d} \mid 
			\bm{w}_{|_T} \in H(\text{div};T)  \quad \forall T \in \mathcal{T}_{h} \}
		\end{equation*}
	\end{Def}
	
	\begin{lemma}[Characterization of $H(\text{div};\Omega)$]\label{Hdivsalti}
		Let $\bm{w} $ be a function in $H(\mathrm{div};\mathcal{T}_{h}) \cap [W^{1,1}(\mathcal{T}_{h})]^{d}$, then
		\begin{equation*}
			\bm{w}  \in	H(\mathrm{div};\Omega) \iff \jmp{\bm{w}}=0 \quad  \forall e\in \mathcal{E}_{h}^{I}.
		\end{equation*}
	\end{lemma}
	For the proof see \cite[Lemma 1.24]{di2011mathematical}.
	
	We now prove the “DG magic formula”, which is fundamental for deriving and analysing the discontinuous Galerkin formulation.
	\begin{prop}[DG magic formula]
		For all $\varphi\in H^1(\calT_h)$ and for all $\bw\in \left[H^1(\calT_h)\right]^d$,
		\begin{equation}\label{DG_magic}
			\sum_{T\in \calT_{h}} \int_{\partial T} \bw \cdot \bn_T \varphi =
			\sum_{e\in \calE_{h}^{I}}	\int_{e} \left(\mvl{\bw} \cdot \jmp{\varphi} +  \jmp{\bw}  \mvl{\varphi}\right)
			+\int_{\partial\Omega}  \bw \cdot \bn \varphi.
		\end{equation}
	\end{prop}
	\begin{proof}
		Let $e\in\calE_h^I$ with $T_1$ and $T_2$  the elements sharing the facet $e$.
		Using the definitions of jumps and averages and the fact that for each internal facet $\bn_1=-\bn_2$, we obtain
		\begin{align*}
			\mvl{\bw} \cdot \jmp{\varphi} +  \jmp{\bw}\mvl{\varphi}&=
			\frac{1}{2}\left(\bw_1+\bw_2\right) \cdot \left(\varphi_1 \bn_1+\varphi_2\bn_2\right)+ \frac{1}{2}\left(\bw_1\cdot \bn_1+\bw_2\cdot\bn_2\right)  \left(\varphi_1 +\varphi_2\right)\\&=
			\bw_1\cdot \bn_1 \varphi_1+\bw_2\cdot \bn_2 \varphi_2.
		\end{align*}
		The assertion follows considering also the boundary facets.
	\end{proof}
	
	We also define the broken version of the space $C^{\infty}_c(\Omega)$, the space of infinitely differentiable functions with compact support in $\Omega$ defined in Section \ref{functionspaces}:
	$$C^{\infty}_c(\calT_h):=\{v \in L^{2}(\Omega) \mid 
	v_{|_T} \in C^{\infty}_c(T)  \quad \forall T \in \mathcal{T}_{h} \}.$$
	
	We now consider broken polynomial spaces.
	We use the notation $\mi=(i_{1},\ldots,i_{d})
	\in \IN^{d}$ for multi-indices and  define their length as  $|\mi|:=i_{1}+\cdots+i_{d}$.
	For $\bx\in \Omega$, we use the standard notation
	$
	\bx^\mi
	=x_1^{i_{1}}\cdots x_d^{i_{d}}.$
	\begin{Def}[Polynomial space]\label{polytot}
		Let $ p\in\mathbb{N}$. We define the \textit{space of polynomials of degree at most $p$ in $d$ variables} as
		$$
		\mathbb{P}^p_d(\Omega):=\biggl\{v:\Omega \to \mathbb{R}\mid  v(\bx)=\sum_{\abs{\bi}\leq p} \alpha_{\bi} \bx^{\bi} \text{ for some } (\alpha_{i})_{\abs{\bi}\leq p}\in\mathbb{R}^{S_{d,p}}\biggl\},
		$$
		where $ S_{d,p}$ denotes the dimension of the space 
		$$
		S_{d,p}:=\text{dim}\left(\mathbb{P}^p_d(\Omega)\right)=\binom{p+d}{d}=\frac{(p+d)!}{p! d!}.
		$$
	\end{Def}
	\begin{Def}[Broken polynomial space]\label{brokenpolyspace}
		Let $ p\in\mathbb{N}$. We define the \textit{broken space of polynomials of degree at most $p$ in $d$ variables} as
		$$
		\mathbb{P}_d^p(\calT_h)=\{v\in L^2(\Omega) \mid 
		v_{|_T} \in \mathbb{P}^p_d(T)  \quad \forall T \in \mathcal{T}_{h} \}.
		$$
	\end{Def}
	
	We recall the discrete inverse trace inequality that is a useful tool to analyse DG methods.
	\begin{lemma}[Discrete trace inequality]\label{Traceinverseineq}
		Let $\calT_{\mathcal{H}}$ be a shape-regular mesh sequence  with the star-shaped property. Then, exists $C_{tr}>0$ such that
		\begin{equation} \label{discretetraceinequality}
			\N{v}^2_{L^2(\partial T)}\leq C_{tr}h_T^{-1} \N{v}^2_{L^2(T)},
		\end{equation}
		for all $h\in\mathcal{H}$,  all $T\in\calT_h$, all $v\in \mathbb{P}^p_d(T)$.
	\end{lemma}
	\begin{proof}
		The assertion follows combining two results that hold under the assumptions.
		The estimate of \cite[Lemma 2]{moiola2018space}: for all $c>0$
		\begin{equation*}
			\N{v}^2_{L^2(\partial T)}\leq \frac{d+c}{rh_T} \N{v}^2_{L^2(T)}+ \frac{h_T}{rc} \N{\nabla v}^2_{L^2(T)},
		\end{equation*}
		where $r$ is defined in \ref{starshaped} and
		the bound of \cite[Lemma 4.5.3]{brenner2008mathematical}:
		exists $C_{inv}>0$ such that
		\begin{equation*}
			\N{\nabla v}_{L^2( T)}\leq C_{inv}h_T^{-1} \N{v}_{L^2(T)}.
		\end{equation*}
	\end{proof}
	We observe that the inequality above and the value of $C_{tr}$ are relative to polynomial spaces.  
	Since we will use as discrete space in the DG scheme the quasi-Trefftz space, which is a subspace of the full polynomial space as we will see in Chapter \ref{Chapter5}, this inequality can be applied.
	We observe that we can have inverse trace inequalities also for different spaces, since the only property of the polynomial space that is used is its finite dimension.
	
	\chapter{The DG variational formulation}\label{Chapter3}
	We study the discontinuous Galerkin method for approximating the diffusion-ad\-vec\-tion-reaction  problem \eqref{variationalcontinuous}.
	In this chapter, we derive the weak discontinuous Galerkin variational formulation of the model problem.
	We proceed considering two cases separately and then we combine them together to deal with the general case.
	First, we consider only the diffusion problem in Section \ref{s:Diffvar}. We derive a formulation that depends on two parameters: the penalty parameter $\gamma>0$ and the symmetrization parameter $\epsilon\in \{-1, 0, 1\}$.
	The choice of $\epsilon$ leads to three different formulations including the Symmetric Interior Penalty Galerkin (SIPG) one. 
	Then, we deal with the advection-reaction problem in Section \ref{s:Advreacvar} using the classical upwind numerical fluxes and, finally, in Section \ref{s:diffAdvreacvar} we consider  the diffusion-advection-reaction problem combining the previous two. 
	We remark that these two cases are considered separately for simplicity of the construction and of the following  discrete coercivity and boundedness analysis.
	Notice that, under the assumptions made in Section \ref{problema}, the diffusion equation is a particular case of the equation \eqref{eq}, since $\bbeta$ and $\sigma$ could vanish. However, this is not true for the problem with only advection and reaction terms, since a direct consequence of the ellipticity condition is that the diffusion matrix $\bk$ cannot be null.
	We mostly follow the analysis of \cite[Chapter 4]{di2011mathematical}, which considers the SIPG method to handle the diffusion term and the upwind DG method to handle the advection-reaction terms, and combine it with the analysis of \cite[Chapter 2]{riviere2008discontinuous}, which considers the three different  formulations  that arise from the choice of the symmetrization parameter $\epsilon$ for diffusion problems.
	
	Let $u$ be the weak solution of problem \eqref{eq}–\eqref{Neumann} and recall that the trial and test space is $V=H^1_{\Gamma_D}(\Omega)$.
	We want to approximate the problem \eqref{variationalcontinuous} using DG methods. 
	Consider as discrete space $V_h$ a finite-dimensional subspace of the broken Sobolev space $H^2(\calT_h)$.
	Hence $V_h \subseteq H^2(\calT_h)\not \subseteq V$.
	In the spirit of Theorem \ref{errortheorem},
	we assume that 
	$u\in V_{*} := H^1_{\Gamma_D}(\Omega) \cap H^2(\calT_h)$
	and we define $V_{*h} := V_* + V_h$.
	From now on we will use this notation for the above spaces unless otherwise specified.
	
	\section{Diffusion}\label{s:Diffvar}
	In this section we consider the following boundary value problem for the diffusion equation:
	\begin{alignat}{2}
		\text{div}(-\bm{K} \nabla u)&= f  && \quad \text{in  $\Omega$,}\label{eqdiff}\\
		u&=g_{D} &&\quad  \text{on $ \Gamma_{D}$,} \label{Dirichletdiff}\\
		- \bm{K} \nabla u  \cdot \mathbf{n} &=g_{N} &&\quad \text{on $ \Gamma_{N}$.}\label{Neumanndiff}
	\end{alignat}
	This is a particular case of the model problem \eqref{eq}-\eqref{Neumann} with $\bbeta=\bm{0}$ and $\sigma=0$. We work under the same assumptions made in the previous chapters. 
	We will derive the following discrete  bilinear form $a_h^{d}:V_{*h}\times V_h\to\mathbb{R}:$
	\begin{align*}
		a_h^{d}(w,v_h):=&
		\sum_{T\in\mathcal{T}_{h}}	\int_{T} \bm{K}  \nabla w\cdot \nabla  v_h -
		\sum_{e\in\mathcal{E}_{h}^{I}}	\int_{e}  \mvl{ \bm{K}   \nabla w} \cdot \jmp{v_h} +
		\epsilon \sum_{e\in\mathcal{E}_{h}^{I}}	\int_{e}  \jmp{w} \cdot  \mvl{ \bm{K}   \nabla v_h }\\
		&-\sum_{e\in\mathcal{E}_{h}^{D}}	\int_{e}  \bm{K}   \nabla w\cdot \mathbf{n} v_h+
		\epsilon 	\sum_{e\in\mathcal{E}_{h}^{D}}	\int_{e}  w \bm{K}   \nabla v_h \cdot \mathbf{n} \\
		&+ \sum_{e\in\mathcal{E}_{h}^{I}}\frac{\gamma}{h_e}	\int_{e}  \jmp{w}\cdot \jmp{v_h}+
		\sum_{e\in\mathcal{E}_{h}^{D}}\frac{\gamma}{h_e}	\int_{e} w v_h.
	\end{align*}
	We construct it in the spirit of Theorem \eqref{errortheorem}, so we want to achieve consistency,  discrete coercivity and boundedness.
	
	Let $u$ be the exact solution of the diffusion problem \eqref{eqdiff}-\eqref{Neumanndiff}.
	Let $v_h\in V_h$, we multiply \eqref{eqdiff} by the test function $v_h$ and integrate on one element $T$:
	\begin{equation*}
		\int_{T} \text{div}(-\bm{K} \nabla u)   v_h = \int_{T} f  v_h.
	\end{equation*}
	Integrating by parts
	and summing over all elements, we obtain
	\begin{equation*}
		\sum_{T\in\mathcal{T}_{h}}	\int_{T} \bm{K}  \nabla u \cdot \nabla  v_h -
		\sum_{T\in\mathcal{T}_{h}}	\int_{\partial T}   \bm{K}   \nabla u \cdot \mathbf{n}_{T} v_h
		= \int_{\Omega} f  v_h.
	\end{equation*}
	Owing to the DG magic formula \eqref{DG_magic} with $\bw=\bk\nabla u$ and $\varphi=v_h$ the second term can be rewritten as a sum over mesh facets in this way:
	\begin{equation*}
		\sum_{T\in \calT_{h}} \int_{\partial T}  \bm{K}   \nabla u\cdot \mathbf{n}_{T} v_h =
		\sum_{e\in \calE_{h}^{I}}	\int_{e} \left(\mvl{\bm{K}   \nabla u} \cdot \jmp{v_h} +  \jmp{\bm{K}   \nabla u}  \mvl{v_h}\right)
		+	\sum_{e\in \calE_{h}^{B}}\int_{e}  \bm{K}   \nabla u\cdot \bn v_h.
	\end{equation*}
	Note that  $\bm{K}   \nabla u\in H(\text{div};\Omega)$ since $u$ satisfies \eqref{eqdiff} and $f\in L^2(\Omega)$.
	Hence, for Proposition \ref{Hdivsalti} we know that the jump $ \jmp{\bm{K}   \nabla u}$ is zero on each interior facet, implying that $\sum_{e\in \calE_{h}^{I}}	\int_{e}  \jmp{\bm{K}   \nabla u} \mvl{v_h}=0$.
	Using the Neumann boundary condition \eqref{Neumanndiff}, we get
	\begin{equation}\label{addsym}
		\sum_{T\in\mathcal{T}_{h}}	\int_{T} \bm{K}  \nabla u \cdot \nabla  v_h -
		\sum_{e\in\mathcal{E}_{h}^{I}}	\int_{e}  \mvl{ \bm{K}   \nabla u } \cdot \jmp{v_h} -
		\sum_{e\in\mathcal{E}_{h}^{D}}	\int_{e}  \bm{K}   \nabla u  \cdot \mathbf{n} v_h
		= \int_{\Omega} f  v_h- \sum_{e\in\mathcal{E}_{h}^{N}}	\int_{e} g_{N} v_h.
	\end{equation}
	If the model problem is symmetric, like in this case, is good, in general, to preserve the symmetry property at the discrete level. 
	We can think of making the discrete bilinear form symmetric, this will lead to the Symmetric Interior Penalty Galerkin (SIPG) method.
	Focusing on the left-hand side of \eqref{addsym}, we note that the first term is already symmetric, so in order to achieve the symmetry we need to add two quantities:
	$
	-\sum_{e\in\mathcal{E}_{h}^{I}}	\int_{e} \jmp{u}\cdot \mvl{\bm{K}   \nabla v_h} $ and 
	$- \sum_{e\in\mathcal{E}_{h}^{D}}	\int_{e}  u \bm{K}   \nabla v_h  \cdot \mathbf{n} $. 
	We will consider not only the symmetric formulation but also two nonsymmetric ones.
	For this reason, we  introduce a parameter $\epsilon$ that can assume value $-1$,  $0$ or $1$ and generalize the previous step, adding to the left-hand side of \eqref{addsym} the two following terms
	\begin{equation*}
		+\epsilon \sum_{e\in\mathcal{E}_{h}^{I}}	\int_{e} \jmp{u}\cdot \mvl{\bm{K}   \nabla v_h }, \quad +\epsilon\sum_{e\in\mathcal{E}_{h}^{D}}	\int_{e}  u \bm{K}   \nabla v_h \cdot \mathbf{n}.
	\end{equation*}
	The choice of $\epsilon$ leads to three different bilinear forms and, hence, different variational formulations.
	For $\epsilon=-1$ the resulting method is called the Symmetric Interior Penalty Galerkin (SIPG) method, for $\epsilon=0$ we have the Incomplete Interior Penalty Galerkin (IIPG) method and for $\epsilon=1$ the Nonsymmetric Interior Penalty Galerkin (NIPG) method.
	We observe that the quantity $+\epsilon \sum_{e\in\mathcal{E}_{h}^{I}}	\int_{e} \jmp{u}\cdot \mvl{\bm{K}   \nabla v_h }$ can be added to the left-hand side of \eqref{addsym} without loosing the consistency property.
	This is due to the fact that, since $u\in H^{1}(\Omega)$, from Proposition \ref{Wsalti} we have $ \jmp{u}=0$ on each interior facet.
	The term $+\epsilon\sum_{e\in\mathcal{E}_{h}^{D}}	\int_{e}  u \bm{K}   \nabla v_h \cdot \mathbf{n}$, instead, is not consistent, so the same quantity is added to the right-hand side of \eqref{addsym} to restore the consistency and using the Dirichlet boundary condition \eqref{Dirichletdiff} on this term, we arrive at:
	\begin{align*}
		&\sum_{T\in\mathcal{T}_{h}}	\int_{T} \bm{K}  \nabla u \cdot \nabla  v_h -
		\sum_{e\in\mathcal{E}_{h}^{I}}	\int_{e}  \mvl{ \bm{K}   \nabla u } \cdot \jmp{v_h} 
		+	\epsilon \sum_{e\in\mathcal{E}_{h}^{I}}	\int_{e} \jmp{u}\cdot  \mvl{ \bm{K}   \nabla v_h }  \\ 
		&	-\sum_{e\in\mathcal{E}_{h}^{D}}	\int_{e}  \bm{K}   \nabla u  \cdot \mathbf{n} v_h+ 	\epsilon \sum_{e\in\mathcal{E}_{h}^{D}}	\int_{e} u \bm{K}   \nabla v_h  \cdot \mathbf{n} 
		\\
		&= \int_{\Omega} f  v - \sum_{e\in\mathcal{E}_{h}^{N}}	\int_{e} g_{N} v_h    +	\epsilon \sum_{e\in\mathcal{E}_{h}^{D}}	\int_{e} g_D  \bm{K}   \nabla v_h  \cdot \mathbf{n}.
	\end{align*}
	To achieve the discrete coercivity, we need to add a stabilization term which penalizes interior facets and boundary jumps.
	Let $\gamma>0$ be the penalty parameter, we add on the left-hand side the two following terms:
	\begin{equation}\label{pen}
		+\sum_{e\in\calE_{h}^{I}}\frac{\gamma}{h_e} 	\int_{e}  \jmp{u}\cdot \jmp{v}, \quad 	+\sum_{e\in\mathcal{E}_{h}^{D}}\frac{\gamma}{h_e} 	\int_{e} u v.
	\end{equation}
	As before, we note that the first term of \eqref{pen} is consistent, since $\jmp{u}=0$ on interior facets, while the other term is not. Hence, we add the second term of  \eqref{pen}  also to the right-hand side of the previous equation in order to maintain the consistency and then use Dirichlet condition \eqref{Dirichletdiff}.
	We obtain the following formulation:
	\begin{align*}
		&\sum_{T\in\mathcal{T}_{h}}	\int_{T} \bm{K}  \nabla u \cdot \nabla  v_h-
		\sum_{e\in\mathcal{E}_{h}^{I}}	\int_{e}  \mvl{ \bm{K}   \nabla u } \cdot \jmp{v_h} 
		+	\epsilon \sum_{e\in\mathcal{E}_{h}^{I}}	\int_{e} \jmp{u}\cdot  \mvl{ \bm{K}   \nabla v_h }  \\ 
		&	-\sum_{e\in\mathcal{E}_{h}^{D}}	\int_{e}  \bm{K}   \nabla u  \cdot \mathbf{n} v_h+ 	\epsilon \sum_{e\in\mathcal{E}_{h}^{D}}	\int_{e} u \bm{K}   \nabla v_h  \cdot \mathbf{n} 
		\\
		&
		+\sum_{e\in\calE_{h}^{I}}\frac{\gamma}{h_e} 	\int_{e}  \jmp{u}\cdot \jmp{v_h}	+\sum_{e\in\mathcal{E}_{h}^{D}}\frac{\gamma}{h_e} 	\int_{e} u v_h\\&
		= \int_{\Omega} f  v_h - \sum_{e\in\mathcal{E}_{h}^{N}}	\int_{e} g_{N} v_h    +	\epsilon \sum_{e\in\mathcal{E}_{h}^{D}}	\int_{e} g_D  \bm{K}   \nabla v_h  \cdot \mathbf{n} +\sum_{e\in\mathcal{E}_{h}^{D}}\frac{\gamma}{h_e} 	\int_{e} g_D v_h.
	\end{align*}
	
	\section{Advection-reaction}\label{s:Advreacvar}
	In this section we consider the following boundary value problem for the advection-reaction equation:
	\begin{alignat}{2}
		\text{div}(\bm{\beta} u) +\sigma u &= f  &&\quad \text{in  $\Omega$,}\label{eqtrasp}\\
		u&=g_{D} &&\quad\text{on $ \Gamma_{D}$.} \label{Dirichlettrasp}
	\end{alignat}
	All the data satisfy the same assumptions we made for the general case, even though this is not a special case of our model problem, since the diffusion matrix is null.
	In this case we will arrive at the following discrete  bilinear form $a_h^{ar}:V_{*h}\times V_h\to\mathbb{R}:$
	\begin{align*}
		a_h^{ar}(w,v_h):=&
		-\sum_{T\in\mathcal{T}_{h}}	\int_{T}  (\bm{\beta} w) \cdot  \nabla  v_h +\sum_{T\in\mathcal{T}_{h}}	\int_{T}  \sigma  w v_h\\
		&+\sum_{e\in\mathcal{E}_{h}^{I}}	\int_{e}  \mvl{\bm{\beta}  w} \cdot \jmp{v_h}+
		\frac{1}{2} \sum_{e\in\mathcal{E}_{h}^{I}}	\int_{e}  |\bm{\beta}\cdot \mathbf{n}_{e}|\jmp{w}\cdot\jmp{v_h}  +
		\sum_{e\in\mathcal{E}_{h}^{N}}	\int_{e} ( \bm{\beta}  w) \cdot \mathbf{n} v_h.
	\end{align*}
	As before, we derive it in the spirit of Theorem \ref{errortheorem}, so we want to achieve consistency,  discrete coercivity and boundedness.
	
	Let $u$ be the exact solution of the problem \eqref{eqtrasp}-\eqref{Dirichlettrasp}.
	Let $v_h\in V_h$, we multiply \eqref{eqtrasp} by the test function $v_h$ and integrate on one element $T$:
	\begin{equation*}
		\int_{T} \text{div}(\bm{\beta} u)  v_h +\int_{T}  \sigma u  v_h= \int_{T} f  v_h.
	\end{equation*}
	Integrating by parts the first term
	and summing over all elements, we obtain
	\begin{equation*}
		-	\sum_{T\in\mathcal{T}_{h}}	\int_{T}\left(\bbeta u\right)\cdot \nabla  v_h +
		\sum_{T\in\mathcal{T}_{h}}	\int_{\partial T}  \left(\bbeta u\right)\cdot \mathbf{n}_{T} v_h+
		\sum_{T\in\mathcal{T}_{h}}	\int_{T}  \sigma u  v_h= 	\int_{\Omega} f  v_h.
	\end{equation*}
	From the DG magic formula \eqref{DG_magic} with $\bw=\bbeta u$ and $\varphi=v_h$ we deduce that
	\begin{equation*}
		\sum_{T\in \calT_{h}} \int_{\partial T}   (\bm{\beta}  u)\cdot \mathbf{n}_{T} v_h=
		\sum_{e\in \calE_{h}^{I}}	\int_{e} \left(\mvl{ \bm{\beta}  u} \cdot \jmp{v_h} +  \jmp{\bm{\beta}  u} \mvl{v_h}\right)
		+\sum_{e\in \calE_{h}^B}\int_{e}  (\bm{\beta}  u)\cdot \bn v_h.
	\end{equation*}
	Note that  $\bbeta u \in H(\text{div};\Omega)$ since $u$ satisfies \eqref{eqtrasp} and $f\in L^2(\Omega)$ and $\sigma\in L^{\infty}(\Omega)$.
	From Proposition \ref{Hdivsalti} we know that the jump $ \jmp{\bbeta u}$ is zero on each interior facet, implying that $\sum_{e\in \calE_{h}^{I}}	\int_{e}  \jmp{\bbeta u}  \mvl{v_h}=0$.
	Using the Dirichlet boundary condition \eqref{Dirichlettrasp}, we get
	\begin{align*}
		&-\sum_{T\in\mathcal{T}_{h}}	\int_{T}  (\bm{\beta} u) \cdot  \nabla  v_h +
		\sum_{e\in\mathcal{E}_{h}^I}	\int_{e}  \mvl{\bm{\beta}  u} \cdot \jmp{v_h}	+\sum_{e\in \calE_{h}^N}\int_{e}  ( \bm{\beta}  u )\cdot \bn v_h+\sum_{T\in\mathcal{T}_{h}}	\int_{T}  \sigma u  v_h\\&= \int_{\Omega} f  v_h   - \sum_{e\in\mathcal{E}_{h}^{D}}	\int_{e}  ( \bm{\beta}  g_{D})\cdot \mathbf{n} v_h.
	\end{align*}
	On every interior facet $e$ we substitute the average $\mvl{\bbeta u}$ by the classical  upwind numerical flux $\mvl{\bbeta u}_{\text{upw}}$ defined in \eqref{upwind}. This is done in order to arrive at a formulation that is stable in a stronger norm than the $L^2$-norm, leading to a quasi-optimal error estimates.  
	We observe that we can do this without loosing the consistency.
	We notice that the average $\mvl{\bbeta u}$ appears on the second term multiplied by the jump $\jmp{v_h}$, hence only the normal component of $\bbeta u$ is relevant.
	For Proposition \ref{Hdivsalti}, since $\bbeta u\in H(\text{div},\Omega)$ for the regularity of $\bbeta $ \eqref{assumptionbeta}, the jump $\jmp{\bbeta u}=0$.
	This means that on each interior facet the two traces of $\bbeta u$ coincide, implying that $\mvl{\bbeta u}=\mvl{\bbeta u}_{\text{upw}}$.
	By doing this substitution and by using the identity \eqref{fluxesusefulid} valid on the interior facets, the second term becomes
	\begin{equation*}
		\sum_{e\in\calE_h^I}	\int_{e}  \mvl{\bm{\beta}  u}_{\text{upw}} \cdot \jmp{v_h}=
		\sum_{e\in\mathcal{E}_{h}^{I}}	\int_{e}  \mvl{\bm{\beta}  u} \cdot \jmp{v_h}+\frac{1}{2}\sum_{e\in\mathcal{E}_{h}^{I}}	\int_{e} \abs{\bbeta \cdot \bn_e} \jmp{u} \cdot \jmp{v_h}.
	\end{equation*}
	With this formula, we can also see the substitution of $\mvl{\bbeta u}$ with $\mvl{\bbeta u}_{\text{upw}}$ like the addition of a stabilization term containing $\jmp{u}$, which is zero since $u\in H^1(\Omega)$.
	This leads to:
	\begin{align*}
		&
		-\sum_{T\in\mathcal{T}_{h}}	\int_{T}  (\bm{\beta} u) \cdot  \nabla  v_h+
		\sum_{e\in\mathcal{E}_{h}^{I}}	\int_{e}  \mvl{\bm{\beta}  u} \cdot \jmp{v_h}+\frac{1}{2}\sum_{e\in\mathcal{E}_{h}^{I}}	\int_{e} \abs{\bbeta \cdot \bn_e} \jmp{u} \cdot \jmp{v_h}+
		\sum_{e\in\mathcal{E}_{h}^{N}}	\int_{e} ( \bm{\beta}  u) \cdot \bn v_h\\
		&+\sum_{T\in\mathcal{T}_{h}}	\int_{T}  \sigma u  v_h= \int_{\Omega} f  v_h  - \sum_{e\in\mathcal{E}_{h}^{D}}	\int_{e}  ( \bm{\beta}  g_{D})\cdot \bn v_h .
	\end{align*}
	
	\section{Diffusion-advection-reaction}\label{s:diffAdvreacvar}
	In this section we consider the diffusion-advection-reaction problem \eqref{eq}-\eqref{Neumann}.
	The DG method for the general problem combines the $a_h^{d} $ bilinear form defined in Section \ref{s:Diffvar} for the diffusion part and the $a_h^{ar} $ bilinear form defined in Section \ref{s:Advreacvar} for the advection-reaction part. 
	
	The discontinuous Galerkin variational formulation of the problem \eqref{eq}–\eqref{Neumann} is as follows:
	\begin{equation}\label{variational}
		\text{Find } u_h \in V_h \text{ such that }
		a_h^{\text{dar}}(u_h,v_h)=L_h(v_h) \quad \forall v_h \in V_h,
	\end{equation}
	with the DG bilinear form 
	$a_h^{\text{dar}}: V_{*h} \times V_h \to \mathbb{R} $:
	\begin{align*}
		a_h^{\text{dar}}(w,v_h):=&
		\sum_{T\in \calT_h}	\int_{T} \bk  \nabla w \cdot \nabla  v_h -
		\sum_{e\in\calE_h^I}	\int_{e}  \mvl{\bm{K}   \nabla w } \cdot \jmp{v_h}+
		\epsilon \sum_{e\in\calE_h^I}	\int_e  \jmp{w} \cdot \mvl{\bm{K}   \nabla v_h}  \\
		&-\sum_{e\in\mathcal{E}_{h}^{D}}	\int_{e}  \bm{K}   \nabla w  \cdot \mathbf{n} v_h+
		\epsilon 	\sum_{e\in\mathcal{E}_{h}^{D}}	\int_{e} w \bm{K}   \nabla v_h  \cdot \mathbf{n} \\
		&+\sum_{e\in\calE_{h}^{I}}\frac{\gamma}{h_e} 	\int_{e}  \jmp{w}\cdot \jmp{v_h}+
		\sum_{e\in\mathcal{E}_{h}^{D}}		\frac{\gamma}{h_e} \int_{e} w v_h\\
		&-\sum_{T\in\mathcal{T}_{h}}	\int_{T}  (\bbeta w) \cdot  \nabla  v_h + \sum_{T\in\mathcal{T}_{h}}	\int_{T}  \sigma w  v_h\\ 
		&+\sum_{e\in\mathcal{E}_{h}^{I}}	\int_{e}  \mvl{\bbeta  w} \cdot \jmp{v_h}+
		\frac{1}{2} \sum_{e\in\mathcal{E}_{h}^{I}}	\int_{e}  |\bm{\beta}\cdot \bn_{e}|\jmp{w}\cdot\jmp{v_h}+
		\sum_{e\in\mathcal{E}_{h}^{N}}	\int_{e} ( \bm{\beta}  w) \cdot \mathbf{n} v_h,
	\end{align*}
	and  the linear form $L_h:V_h\to\mathbb{R}:$
	\begin{equation*}
		L_h(v_h):=\int_{\Omega}  f v_h - \sum_{e\in\mathcal{E}_{h}^{N}} \int_{e}  g_{N} v_h + 	\epsilon \sum_{e\in\mathcal{E}_{h}^{D}}	\int_{e} g_D  \bm{K}   \nabla v_h  \cdot \mathbf{n}  + \sum_{e\in\mathcal{E}_{h}^{D}}	\frac{\gamma}{h_e}\int_{e}  g_{D} v_h- \sum_{e\in\mathcal{E}_{h}^{D}}	\int_{e}  ( \bm{\beta}  g_{D}) \cdot \mathbf{n} v_h.
	\end{equation*}
	The variational problem \eqref{variational} is derived collecting together the steps of the two previous sections, indeed
	\begin{equation*}
		a_h^{dar}(w,v_h)=a_h^{d}(w,v_h)+ a_h^{ar}(w,v_h) \quad \forall  (w,v_h)\in V_{*h}\times V_h.
	\end{equation*}
	The bilinear form $a_h^{dar}$ contains the parameter $\epsilon$ that may take the value $-1$, $0$ or $1$ and the penalty parameter $\gamma >0$ that penalizes the jump of the function values. 
	Note that the problem \eqref{variational} is independent of the choice of the normal $\bn_e$ on the internal facets, since its only occurrence in $a_h^{dar}$ is inside the absolute value.
	
	For the well-posedness of the problem \eqref{variational} and for the convergence analysis in the spirit of Theorem \ref{errortheorem}  we need to prove: discrete coercivity, consistency, and boundedness for the bilinear form $a_h^{dar}$. 
	The consistency is established in the following section, while the discrete coercivity and the boundedness in the next chapter.
	\section{Consistency} 
	Let $u$ be the weak solution of the problem \eqref{eq}–\eqref{Neumann}. 
	By construction we already have the consistency:
	\begin{equation*}
		a_h^{\text{dar}}(u,v_h)=L_h(v_h) \quad \forall v_h \in V_h.
	\end{equation*}
	
	We now show  a stronger property of the method, which also provides a direct proof for the consistency.
	We observe that the bilinear form $a_h^{dar}$ can be extended to $H^2(\calT_h) \times H^2(\calT_h)$. Consider the variational problem
	\begin{equation}\label{cccc}
		\text{Find } u \in H^2(\calT_h)\text{ such that }
		a_h^{\text{dar}}(u,v)=L_h(v) \quad \forall v \in H^2(\calT_h).
	\end{equation}
	The next proposition establishes the equivalence between the model problem \eqref{eq}–\eqref{Neumann} and the DG variational formulation \eqref{cccc}.
	We follow \cite[Proposition 2.9]{riviere2008discontinuous}, which considers diffusion-reaction problems, adapting the proof to the case of diffusion-advection-reaction problems.
	\begin{prop}
		Assume that the weak solution $u$ of problem \eqref{eq}–\eqref{Neumann} belongs to $H^2(\mathcal{T}_{h})$; then $u$ satisfies the variational problem \eqref{cccc}. Conversely, if $u\in H^1(\Omega) \cap H^2(\mathcal{T}_{h})$ satisfies  \eqref{cccc}, then $u$ is the solution of problem \eqref{eq}–\eqref{Neumann}.
	\end{prop}
	\begin{proof}
		First, we prove that if the solution $u$ of \eqref{eq}–\eqref{Neumann} belongs to $H^2(\calT_h)$, then it also solves  \eqref{cccc}.
		Let $v$ be a function in $H^2(\mathcal{T}_{h})$. We multiply \eqref{eq} by $v$ and integrate on one element $T$:
		\begin{equation*}
			\int_{T} \text{div}(-\bm{K} \nabla u +\bm{\beta} u)   v +\int_{T}  \sigma u  v= \int_{T} f  v.
		\end{equation*}
		We integrate by parts 
		and sum over all elements:
		\begin{equation*}
			-\sum_{T\in\mathcal{T}_{h}}	\int_{T} \left(-\bm{K}  \nabla u+\bbeta u\right) \cdot \nabla  v +
			\sum_{T\in\mathcal{T}_{h}}	\int_{\partial T}   \left(-\bm{K}  \nabla u+\bbeta u\right) \cdot \mathbf{n}_{T} v +
			\sum_{T\in\mathcal{T}_{h}}	\int_{T}  \sigma u  v= 	\sum_{T\in\mathcal{T}_{h}}\int_{T} f  v.
		\end{equation*}
		Using the DG magic formula \eqref{DG_magic}  with $\bw=\bk\nabla u+\bbeta u$ and $\varphi=v$ we can rewrite the second term as
		\begin{align*}
			\sum_{T\in\mathcal{T}_{h}}	\int_{\partial T}   \left(-\bm{K}  \nabla u+\bbeta u\right) \cdot \mathbf{n}_{T} v =&
			\sum_{e\in \calE_{h}^{I}}	\int_{e} (\mvl{-\bm{K}  \nabla u+\bbeta u} \cdot \jmp{v} +  \jmp{-\bm{K}  \nabla u+\bbeta u}  \mvl{v})\\&
			+\int_{\partial\Omega}    \left(-\bm{K}  \nabla u+\bbeta u\right)\cdot \bn v.
		\end{align*}
		Notice that $	-\bm{K} \nabla u +\bm{\beta} u $ belongs to  $H(\text{div};\Omega)$ since $u$ satisfy \eqref{eq} and $f$ and $\sigma $ are functions in $L^2(\Omega)$.
		For Proposition \ref{Hdivsalti}, this implies that the jump  $\jmp{-\bm{K}  \nabla u+\bbeta u} =0 $ on each interior facet, hence $\sum_{e\in \calE_{h}^{I}}	\int_{e}  \jmp{\bm{K}   \nabla u+\bbeta u} \cdot \mvl{v}=0$.
		Using the Dirichlet and Neumann boundary conditions \eqref{Dirichlet}-\eqref{Neumann}, we find
		\begin{equation}
			\begin{aligned}\label{intermed}
				&		-\sum_{T\in\mathcal{T}_{h}}	\int_{T} \left(-\bm{K}  \nabla u+\bbeta u\right) \cdot \nabla  v -
				\sum_{e\in\mathcal{E}_{h}^{I}}	\int_{e}  \mvl{ \bm{K}   \nabla u } \cdot \jmp{v} +
				\sum_{e\in\mathcal{E}_{h}\setminus \mathcal{E}_{h}^{D}}	\int_{e}  \mvl{\bm{\beta}  u} \cdot \jmp{v}	\\&
				-	\sum_{e\in\mathcal{E}_{h}^{D}}	\int_{e}  \bm{K}   \nabla u  \cdot \mathbf{n} v+\sum_{T\in\mathcal{T}_{h}}	\int_{T}  \sigma u  v= \int_{\Omega} f  v - \sum_{e\in\mathcal{E}_{h}^{N}}	\int_{e} g_{N} v   - \sum_{e\in\mathcal{E}_{h}^{D}}	\int_{e}  ( \bm{\beta}  g_{D})\cdot \mathbf{n} v .
			\end{aligned}
		\end{equation}
		As we have done  in Section \ref{s:Advreacvar} for the advection-reaction case, in the third term we substitute $\mvl{\bbeta u}$ with the upwind numerical flux $\mvl{\bbeta u}_{\text{upw}}$ and use the identity \eqref{fluxesusefulid}:
		\begin{equation*}
			\sum_{e\in\mathcal{E}_{h}\setminus \mathcal{E}_{h}^{D}}	\int_{e}  \mvl{\bm{\beta}  u}_{\text{upw}} \cdot \jmp{v}=
			\sum_{e\in\mathcal{E}_{h}^{I}}	\int_{e}  \mvl{\bm{\beta}  u} \cdot \jmp{v}+\frac{1}{2}\sum_{e\in\mathcal{E}_{h}^{I}}	\int_{e} \abs{\bbeta \cdot \bn_e} \jmp{u} \cdot \jmp{v}+
			\sum_{e\in\mathcal{E}_{h}^{N}}	\int_{e} ( \bm{\beta}  u) \cdot \bn v.
		\end{equation*}
		Since $u\in H^{1}(\Omega)$, by Proposition \ref{Wsalti} the jump $\jmp{u}=0$ on each interior facet.
		Proceeding as in Section \ref{s:Diffvar} for the diffusion case, we add on the left-hand side of \eqref{intermed} the quantity $
		+\epsilon \sum_{e\in\mathcal{E}_{h}^{I}}	\int_{e}  \jmp{u}\cdot \mvl{\bm{K}   \nabla v } $ without loosing consistency and we add $+	\epsilon \sum_{e\in\mathcal{E}_{h}^{D}}	\int_{e}  u \bm{K}   \nabla v  \cdot \mathbf{n} $ to both sides of \eqref{intermed} and then use the Dirichlet boundary condition \eqref{Dirichlet}.
		Combining the previous steps we have:
		\begin{align*}
			&	-\sum_{T\in\mathcal{T}_{h}}	\int_{T} \left(-\bm{K}  \nabla u+\bbeta u\right) \cdot \nabla  v -
			\sum_{e\in\mathcal{E}_{h}^{I}}	\int_{e}  \mvl{ \bm{K}   \nabla u } \cdot \jmp{v} 
			+	\epsilon \sum_{e\in\mathcal{E}_{h}^{I}}	\int_{e}  \jmp{u}\cdot  \mvl{ \bm{K}   \nabla v } \\ 
			&	-\sum_{e\in\mathcal{E}_{h}^{D}}	\int_{e}  \bm{K}   \nabla u  \cdot \mathbf{n} v+ 	\epsilon \sum_{e\in\mathcal{E}_{h}^{D}}	\int_{e} u \bm{K}   \nabla v  \cdot \mathbf{n}
			\\ &+
			\sum_{e\in\mathcal{E}_{h}^{I}}	\int_{e}  \mvl{\bm{\beta}  u} \cdot \jmp{v}+\frac{1}{2}\sum_{e\in\mathcal{E}_{h}^{I}}	\int_{e} \abs{\bbeta \cdot \bn_e} \jmp{u} \cdot \jmp{v}+
			\sum_{e\in\mathcal{E}_{h}^{N}}	\int_{e} ( \bm{\beta}  u) \cdot \bn v\\
			&+\sum_{T\in\mathcal{T}_{h}}	\int_{T}  \sigma u  v= \int_{\Omega} f  v - \sum_{e\in\mathcal{E}_{h}^{N}}	\int_{e} g_{N} v   - \sum_{e\in\mathcal{E}_{h}^{D}}	\int_{e}  ( \bm{\beta}  g_{D})\cdot \mathbf{n} v  +	\epsilon \sum_{e\in\mathcal{E}_{h}^{D}}	\int_{e}g_D  \bm{K}   \nabla v  \cdot \mathbf{n} .
		\end{align*}
		We add the stabilization terms, which penalize the interior and boundary jumps:
		the consistent term 
		$\frac{\gamma}{h_e} \sum_{e\in\calE_{h}^{I}}	\int_{e}  \jmp{u}\cdot \jmp{v}$ only on the left-hand side of the last equation and the quantity 
		$
		\frac{\gamma}{h_e} \sum_{e\in\mathcal{E}_{h}^{D}}	\int_{e} u v$
		to both sides and then use Dirichlet condition \eqref{Dirichlet}.
		We have proved that $u$ satisfy the variational equation \eqref{cccc}
		\begin{align*}
			&	-\sum_{T\in\mathcal{T}_{h}}	\int_{T} \left(-\bm{K}  \nabla u+\bbeta u\right) \cdot \nabla  v -
			\sum_{e\in\mathcal{E}_{h}^{I}}	\int_{e}  \mvl{ \bm{K}   \nabla u } \cdot \jmp{v} 
			+	\epsilon \sum_{e\in\mathcal{E}_{h}^{I}}	\int_{e}  \jmp{u}\cdot  \mvl{ \bm{K}   \nabla v } \\ 
			&	-\sum_{e\in\mathcal{E}_{h}^{D}}	\int_{e}  \bm{K}   \nabla u  \cdot \mathbf{n} v+ 	\epsilon \sum_{e\in\mathcal{E}_{h}^{D}}	\int_{e} u \bm{K}   \nabla v  \cdot \mathbf{n} 
			+\sum_{e\in\calE_{h}^{I}}\frac{\gamma}{h_e}	\int_{e}  \jmp{u}\cdot \jmp{v}+
			\sum_{e\in\mathcal{E}_{h}^{D}}	\frac{\gamma}{h_e}\int_{e} u v\\&+
			\sum_{e\in\mathcal{E}_{h}^{I}}	\int_{e}  \mvl{\bm{\beta}  u} \cdot \jmp{v}+\frac{1}{2}\sum_{e\in\mathcal{E}_{h}^{I}}	\int_{e} \abs{\bbeta \cdot \bn_e} \jmp{u} \cdot \jmp{v}+
			\sum_{e\in\mathcal{E}_{h}^{N}}	\int_{e} ( \bm{\beta}  u) \cdot \bn v+\sum_{T\in\mathcal{T}_{h}}	\int_{T}  \sigma u  v\\
			&= \int_{\Omega} f  v - \sum_{e\in\mathcal{E}_{h}^{N}}	\int_{e} g_{N} v   - \sum_{e\in\mathcal{E}_{h}^{D}}	\int_{e}  ( \bm{\beta}  g_{D})\cdot \mathbf{n} v  +	\epsilon \sum_{e\in\mathcal{E}_{h}^{D}}	\int_{e}g_D  \bm{K}   \nabla v  \cdot \mathbf{n}+
			\sum_{e\in\mathcal{E}_{h}^{D}}		\frac{\gamma}{h_e} \int_{e} g_D v.
		\end{align*}
		
		Conversely,  let $u\in H^1(\Omega) \cap H^2(\mathcal{T}_{h})$ be solution of \eqref{cccc} and 
		take $v\in C^{\infty}_c(\calT_h)$. Then \eqref{cccc} reduces to 
		\begin{equation*}
			-\sum_{T\in\mathcal{T}_{h}}	\int_{T} \left(-\bm{K}  \nabla u+\bbeta u\right) \cdot \nabla  v 
			+\sum_{T\in\mathcal{T}_{h}}	\int_{T}  \sigma u  v= 	\sum_{T\in\mathcal{T}_{h}}	\int_{T}  f  v.
		\end{equation*}
		This implies that, for all $T\in\calT_h$,  for the density of $C^{\infty}_c(T)$ in $L^2(T)$, we have 
		$
		\text{div}(-\bm{K}  \nabla u   + \bm{\beta} u)  + \sigma u =  f  
		$
		in $L^2(T)$ and, hence,
		a.e in $T$.
		
		Let $e\in \calE_h^I$ with $e=\partial T_1\cap\partial T_2$. Consider $v\in C^{\infty}_c(T_1\cup T_2\cup e)$ and consider its extension by zero over the rest of the domain. 
		We now know that $
		\text{div}(-\bm{K}  \nabla u   + \bm{\beta} u)  + \sigma u =  f  
		$
		in $T_1\cup T_2$.
		Multiplying this equation  by $v$ and using the integration by parts, we get 
		\begin{equation*}
			\int_{T_1 \cup T_2} \left(\bm{K}  \nabla u -\bbeta u\right)\cdot \nabla  v +
			\int_{e} \jmp{-\bm{K}  \nabla u+\bbeta u}   v +
			\int_{T_1 \cup T_2}  \sigma u  v= 	\int_{T_1 \cup T_2} f  v.
		\end{equation*}
		Since $\jmp{v}=0$ on all facets, $\nabla v=0$ on boundary facets and $\jmp{u}=0$ on interior facets, equation \eqref{cccc} becomes
		\begin{equation*}
			-\int_{T_1 \cup T_2} \left(-\bm{K}  \nabla u+\bbeta u\right) \cdot \nabla  v
			+	
			\int_{T_1 \cup T_2}  \sigma u  v= \int_{T_1 \cup T_2} f  v.
		\end{equation*}
		From the two previous equations, that we have derived, it follows that
		\begin{equation*}
			\int_{e} \jmp{-\bm{K}  \nabla u+\bbeta u}   v =0 \quad \forall v\in C^{\infty}_{c}(T_1 \cup T_2\cup e).
		\end{equation*}
		A consequence is that $	\jmp{-\bm{K}  \nabla u+\bbeta u}   =0$ in $L^2(e)$.
		This holds for all facets $e$, implying that 
		$\text{div}(-\bk \nabla u +\bbeta u) \in L^2(\Omega)$, and then we obtain
		\begin{equation}\label{eqdeducted}
			\text{div}(-\bk \nabla u +\bbeta u)+\sigma u=f\quad \text{in} \quad \Omega.
		\end{equation}
	
Next, we want to deduce the Dirichlet boundary condition \eqref{Dirichlet}. 
		First, consider the case when the symmetrization parameter $\epsilon\in\{-1,1\}$. Multiplying \eqref{eqdeducted} by a function $v\in H^{2}(\Omega)\cap H^1_0(\Omega)$ and integrating by parts, we have
		\begin{equation*}
			\sum_{T\in\calT_h}	\int_{T} \left(\bm{K}  \nabla u -\bbeta u\right)\cdot \nabla  v 
			+
			\sum_{T\in\calT_h}	\int_{T}  \sigma u  v= \sum_{T\in\calT_h}	\int_{T} f  v.
		\end{equation*}
		From  \eqref{cccc} with $v$ as test function, we get
		\begin{equation*}
			\sum_{T\in\mathcal{T}_{h}}	\int_{T} \left(\bm{K}  \nabla u-\bbeta u\right) \cdot \nabla  v  
			+ 	\epsilon \sum_{e\in\mathcal{E}_{h}^{D}}	\int_{e} u \bm{K}   \nabla v  \cdot \mathbf{n} +
			\sum_{T\in\mathcal{T}_{h}}	\int_{T}  \sigma u  v
			= \int_{\Omega} f  v  +	\epsilon \sum_{e\in\mathcal{E}_{h}^{D}}	\int_{e}g_D  \bm{K}   \nabla v  \cdot \mathbf{n}.
		\end{equation*}
		Comparing the two formulas above we arrive at
		\begin{equation*}
			\epsilon \sum_{e\in\mathcal{E}_{h}^{D}}	\int_{e} u \bm{K}   \nabla v  \cdot \mathbf{n} 
			= \epsilon \sum_{e\in\mathcal{E}_{h}^{D}}	\int_{e} g_D \bm{K}   \nabla v  \cdot \mathbf{n} 
		\end{equation*}
		This holds for all $v\in H^{2}(\Omega)\cap H^{1}_{0}(\Omega)$ and, since $\epsilon\neq 0$, implies that $u=g_D$ on $\Gamma_D$.
		Now consider the case $\epsilon=0$. Let $e\in\calE_h^D$ with $T\in\calT_h$ such that $e\subseteq \partial T$.
		Consider $v\in C^{\infty}_c(T\cup e\cup \Omega^c)$, where $\Omega^c$ denotes the complement of $\Omega$ in $\mathbb{R}^d$, and consider its extension by zero over the rest of the domain.
		Multiplying \eqref{eqdeducted} by $v$ and integrating by parts, we get
		\begin{equation*}
			\int_{T} \left(\bk  \nabla u-\bbeta u\right)\cdot \nabla  v  +\int_{e} \left(-\bk \nabla u + \bbeta u\right)\cdot \bn v+ 	\int_{T}  \sigma u  v 
			= \int_{T}  f v.
		\end{equation*}
		Using $v$ as test function in \eqref{cccc}, we obtain
		\begin{align*}
		&	\int_{T} \left(\bk  \nabla u-\bbeta u\right)\cdot \nabla  v  + 	\int_{T}  \sigma u  v -	\int_{e}  \bm{K}   \nabla u  \cdot \mathbf{n} v+
			\frac{\gamma}{h_e} 
			\int_{e} u v=\\& \int_{T}  f v -\int_{e}  ( \bm{\beta}  g_{D}) \cdot \mathbf{n} v+ 	\frac{\gamma}{h_e}\int_{e}  g_{D} v.
		\end{align*}
		From the two formulas above follows:
		\begin{equation*}
			\int_{e} \left(u-g_D\right)\left(\frac{\gamma}{h_e} -\bbeta \cdot \bn  \right)v =0.
		\end{equation*}
		Since $\bbeta \cdot \bn <0$ on $e\in \calE_h^D$  for the assumption \eqref{gammaDgammaN}  and since $\frac{\gamma}{h_e}>0$ we deduce that $ \left(\frac{\gamma}{h_e} -\bbeta \cdot \bn  \right)>0$ on $e$.
		Then, the equality above, which holds for all $v\in C^{\infty}_c(T\cup e\cup \Omega^c)$, implies that $u=g_D$ on $e$.
		This is true for all $e\in\calE_h^D$, hence $u=g_D$ on $\Gamma_D$.

	To recover the Neumann boundary condition \eqref{Neumann} we choose $v\in H^2(\Omega)$ such that $v=0$ on $\Gamma_D$.
		As done before, on one hand we multiply \eqref{eq} by $v$, integrate by parts and obtain
		\begin{equation*}
			\sum_{T\in\calT_h}	\int_{T} \left(\bm{K}  \nabla u -\bbeta u\right)\cdot \nabla  v +
			\sum_{e\in\mathcal{E}_{h}^{N}}\int_{e}    \left(-\bm{K}  \nabla u+\bbeta u\right)\cdot \bn v
			+
			\sum_{T\in\calT_h}	\int_{T}  \sigma u  v= \sum_{T\in\calT_h}	\int_{T} f  v.
		\end{equation*}
		On the other hand, from \eqref{cccc} and from the Dirichlet boundary condition, we find
		\begin{equation*}
			\sum_{T\in\mathcal{T}_{h}}	\int_{T} \left(\bm{K}  \nabla u-\bbeta u\right) \cdot \nabla  v 
			+
			\sum_{e\in\mathcal{E}_{h}^{N}}	\int_{e} ( \bm{\beta}  u) \cdot \bn v+\sum_{T\in\mathcal{T}_{h}}	\int_{T}  \sigma u  v= \int_{\Omega} f  v - \sum_{e\in\mathcal{E}_{h}^{N}}	\int_{e} g_{N} v.
		\end{equation*}
		From the comparison between the two equation above we infer that 
		\begin{equation*}
			-	\sum_{e\in\calE_h^N}	\int_{e}   \bm{K}   \nabla u \cdot \bn v 
			=- \sum_{e\in\mathcal{E}_{h}^{N}}	\int_{e} g_{N} v,
		\end{equation*}
		which implies \eqref{Neumann}, completing the proof.
	\end{proof}
	
	\chapter{Well-posedness and quasi-optimality of the DG method}\label{Chapter4}
	The aim of this chapter is to prove the well-posedness of the discrete DG problem \eqref{variational} and the quasi-optimal error estimates of the DG method. In particular, since we already established the consistency, we focus on the proofs of the discrete coercivity and of the boundedness.
	First, in Section \ref{normss} we collect the bilinear forms, the norms and some preliminary estimates.
	In Section \ref{coercivitydiscrete} we prove the discrete coercivity of $a_h^{dar}$, in Section \ref{s:boundedness} the boundedness of $a_h^{dar}$ and in Section \ref{s:erroranalysis} the error estimates.
	We follow mostly \cite[Chapter 4]{di2011mathematical} and \cite[Chapter 2]{riviere2008discontinuous}.
		\section{Bilinear forms, norms and preliminary estimates}\label{normss}
	Recall that the trial and test space is $V=H^1_{\Gamma_D}(\Omega)$, that the exact solution
	$u\in V_{*} :=H^1_{\Gamma_D}(\Omega) \cap H^2(\calT_h)$ and  $V_{*h} := V_* + V_h$.
	In Chapter \ref{Chapter3} we have considered as discrete space $V_h$ a finite-dimensional subspace of the broken Sobolev space $H^2(\calT_h)$. Now we consider $V_h\subseteq \mathbb{P}^p_d(\calT_h)$, i.e., a subspace of the broken space of degree-$p$ polynomials defined in \ref{brokenpolyspace}, since we will make use of the discrete inverse trace inequality \eqref{discretetraceinequality} for polynomial spaces.
	
	We recall the $a_h^{dar}$ bilinear form: for all $(v, w_h)\in V_{*h} \times V_h$, we have 
	\begin{equation*}
		a_h^{dar}(w,v_h)=a_h^{d}(w,v_h)+ a_h^{ar}(w,v_h),
	\end{equation*}
	where
	\begin{align*}
		a_h^{d}(w,v_h)=&
		\sum_{T\in \calT_h}	\int_{T} \bk  \nabla w \cdot \nabla  v_h -
		\sum_{e\in\calE_h^I}	\int_{e}  \mvl{\bm{K}   \nabla w } \cdot \jmp{v_h}+
		\epsilon \sum_{e\in\calE_h^I}	\int_e  \jmp{w} \cdot \mvl{\bm{K}   \nabla v_h}  \\
		&-\sum_{e\in\mathcal{E}_{h}^{D}}	\int_{e}  \bm{K}   \nabla w  \cdot \mathbf{n} v_h+
		\epsilon 	\sum_{e\in\mathcal{E}_{h}^{D}}	\int_{e} w \bm{K}   \nabla v_h  \cdot \mathbf{n} \\
		&+ \sum_{e\in\calE_{h}^{I}}\frac{\gamma}{h_e}	\int_{e}  \jmp{w}\cdot \jmp{v_h}+ \sum_{e\in\mathcal{E}_{h}^{D}}\frac{\gamma}{h_e}	\int_{e} w v_h,\\
		a_h^{ar}(w,v_h)=&
		-\sum_{T\in\mathcal{T}_{h}}	\int_{T}  (\bbeta w) \cdot  \nabla  v_h + \sum_{T\in\mathcal{T}_{h}}	\int_{T}  \sigma w  v_h\\ 
		&+\sum_{e\in\mathcal{E}_{h}^{I}}	\int_{e}  \mvl{\bbeta  w} \cdot \jmp{v_h}+
		\frac{1}{2} \sum_{e\in\mathcal{E}_{h}^{I}}	\int_{e}  |\bm{\beta}\cdot \bn_{e}|\jmp{w}\cdot\jmp{v_h}+
		\sum_{e\in\mathcal{E}_{h}^{N}}	\int_{e} ( \bm{\beta}  w) \cdot \mathbf{n} v_h.
	\end{align*}
	
	For the convergence analysis, in the spirit of Theorem \ref{errortheorem}, we need to define some norms on $V_{*h}$.
	For all $v\in V_{*h}$ we define the following norms.
	For the diffusion term we consider the energy norm
	\begin{equation*}
		\vertiii{v}^2_{d}:=\sum_{T\in \calT_h}\int_{T}\bk\nabla v\cdot \nabla v+\abs{v}^2_{J},
	\end{equation*}
	with the jump seminorm
	\begin{equation*}
		\abs{v}^2_{J}:=\sum_{e\in\calE_{h}^{I}}\frac{\gamma}{h_e}\int_{e} \jmp{v}^2+\sum_{e\in\calE_{h}^{D}}\frac{\gamma}{h_e}\int_{e}v^2,
	\end{equation*}
	and the following stronger norm for the boundedness
	\begin{equation*}
		\vertiii{v}^2_{d,*}:=	\vertiii{v}^2_{d}+\sum_{T\in \calT_h}h_T\N{\bk^{\frac{1}{2}}\nabla v\cdot \bn_T}^2_{L^2(\partial T)}.
	\end{equation*}
	For handling the advection-reaction term, we use
	\begin{equation*}
		\vertiii{v}^2_{ar}:=\N{v}^2_{L^{2}(\Omega)}+ \frac{1}{2}\sum_{e\in\calE_{h}}\int_{e}\abs{\bbeta\cdot\bn_e}\jmp{v}^2,
	\end{equation*}
	and the stronger norm for the boundedness
	\begin{equation*}
		\vertiii{v}^2_{ar,*}:=\vertiii{v}^2_{ar}+\sum_{T\in \calT_h}\N{ v}^2_{L^2(\partial T)}.
	\end{equation*}
	Combining the previous norms we define two norms, one stronger than the other, for the error analysis of the DG method for the diffusion-advection-reaction problem:
	\begin{align*}
		\vertiii{v}^2_{dar}:=&	\vertiii{v}^2_{d}+	\vertiii{v}^2_{ar},\\
		\vertiii{v}^2_{dar,*}:=&	\vertiii{v}^2_{d,*}+	\vertiii{v}^2_{ar,*}.
	\end{align*}
	In order to see that all of them are norms the only non trivial property to check is that for all $v \in V_{*h}$, if $\vertiii{v}=0$ then $v=0$, since the absolute homogeneity and the triangle inequality are trivial and if $v=0 $ then all of them are zero.
	
	In the next two lemmas, we estimate the quantity 
	$\sum_{e\in\calE_h^I\cup \calE_h^D} \int_{e}  \mvl{\bm{K}   \nabla v} \cdot \jmp{w_h}$ with two different bounds as we will require both of them in following analysis.
	\begin{lemma}
		For all $(v,w_h)\in V_{*h}\times  V_h$,
		\begin{equation}\label{consistencytermbound}
			\abs{\sum_{e\in\calE_h^I\cup \calE_h^D} \int_{e}  \mvl{\bm{K}   \nabla v} \cdot \jmp{w_h}} \leq
			\N{\bk}_{L^{\infty}(\Omega)}^{\frac{1}{2}}
			\left( 		\sum_{T\in\calT_h} \sum_{e\in\calF_T} \frac{h_e}{\gamma} 
			\N{(\bk^{\frac{1}{2}}\nabla v)_{|_T}\cdot \bn_e}_{L^2(e)}^2
			\right)^{\frac{1}{2}} \abs{w_h}_J.
		\end{equation}
	\end{lemma}
	\begin{proof}
		Let $(v,w_h)\in V_{*h}\times  V_h$ and let $\mathfrak{T}:=\sum_{e\in\calE_h^I\cup \calE_h^D} \int_{e}  \mvl{\bm{K}   \nabla v} \cdot \jmp{w_h}$.
		Using the Cauchy--Schwarz inequality, we get
		\begin{equation}\label{lemmaconsistency}
			\abs{\mathfrak{T}}\leq 
			\left(\sum_{e\in\calE_h^I\cup \calE_h^D}  \frac{h_e}{\gamma}
			\int_e \left(\mvl{\bk\nabla v}\cdot \bn_e\right)^{2}
			\right)^{\frac{1}{2}}
			\left(\sum_{e\in\calE_h^I\cup \calE_h^D}  \frac{\gamma}{h_e}\N{\jmp{w_h}}^2_{L^2(e)}\right)^{\frac{1}{2}},
		\end{equation}
		where the last term is the jump seminorm $\abs{w_h}_{J}$.
		In the other term we distinguish between interior facets and boundary ones:
		for all $e\in\calE_h^I$ with $e=\partial T_1 \cap \partial T_2$,
		the Young's inequality yields
		\begin{align*}
			\int_{e} \left( \mvl{\bm{K}   \nabla v} \cdot \bn_e\right)^2&
			=\int_{e}\left[\frac{1}{2} 
			\left(\bk_{|_{T_1}}^{\frac{1}{2}} (\bk^{\frac{1}{2}}\nabla v)_{|_{T_1}}+\bk_{|_{T_2}}^{\frac{1}{2}} (\bk^{\frac{1}{2}}\nabla v)_{|_{T_2}}\right)  \cdot \bn_e\right]^2\\
			&	\leq
			\frac{1}{2} 	\N{\bk}_{L^{\infty}(\Omega)} \left(\N{(\bk^{\frac{1}{2}}\nabla v)_{|_{T_1}} \cdot \bn_e}^2_{L^2(e)}+\N{(\bk^{\frac{1}{2}}\nabla v)_{|_{T_2}}\cdot \bn_e}^2_{L^2(e)}\right),
		\end{align*}
		and for all $e\in\calE_h^D$ with $e\subseteq\partial T\cap \Gamma_D$, we obtain
		\begin{equation*}
			\int_{e} \left( \mvl{\bm{K}   \nabla v} \cdot \bn_e\right)^2
			=\int_{e}\left[
			\left(\bk_{|_{T}}^{\frac{1}{2}} (\bk^{\frac{1}{2}}\nabla v)_{|_{T}}\right)  \cdot \bn_e\right]^2	\leq
			\N{\bk}_{L^{\infty}(\Omega)} \N{(\bk^{\frac{1}{2}}\nabla v)_{|_{T}} \cdot \bn_e}^2_{L^2(e)}.
		\end{equation*}
		Combining these two bounds, we arrive at
		\begin{align*}
			\abs{\mathfrak{T}}\leq & \N{\bk}_{L^{\infty}(\Omega)}^{\frac{1}{2}}
			\left[\sum_{e\in\calE_h^I}  \frac{h_e}{\gamma}\frac{1}{2} 
			\left(\N{(\bk^{\frac{1}{2}}\nabla v)_{|_{T^e_2}} \cdot \bn_e}^2_{L^2(e)}+\N{(\bk^{\frac{1}{2}}\nabla v)_{|_{T^e_2}}\cdot \bn_e}^2_{L^2(e)}\right)\right.\\
			&\left. +
			\sum_{e\in\calE_h^D}  \frac{h_e}{\gamma} 
			\N{(\bk^{\frac{1}{2}}\nabla v)_{|_T}\cdot\bn_e}_{L^2(e)}^2
			\right]^{\frac{1}{2}} \abs{w_h}_{J}.
		\end{align*}
		The thesis is derived by aggregating the facet contributions of each mesh element.
	\end{proof}
	\begin{lemma}
		For all $(v,w_h)\in V_{*h}\times  V_h$,
		\begin{equation}\label{consistencyterm}
			\abs{\sum_{e\in\calE_h^I\cup \calE_h^D} \int_{e}  \mvl{\bm{K}   \nabla v} \cdot \jmp{w_h}} \leq
			\N{\bk}_{L^{\infty}(\Omega)}
			\left( 		\sum_{T\in\calT_h} \sum_{e\in\calF_T} \frac{h_e}{\gamma} 
			\N{\nabla v_{|_T}\cdot \bn_e}_{L^2(e)}^2
			\right)^{\frac{1}{2}} \abs{w_h}_J.
		\end{equation}
	\end{lemma}
	\begin{proof}
		We proceed as in the proof of the previous lemma.
		The only difference is in the way we estimate the terms $\int_{e} \left( \mvl{\bm{K}   \nabla v} \cdot \bn_e\right)^2$:
		for all $e\in\calE_h^I$ with $e=\partial T_1 \cap \partial T_2$,
		the Young's inequality yields
		\begin{align*}
			\int_{e} \left( \mvl{\bm{K}   \nabla v} \cdot \bn_e\right)^2&
			=\int_{e}\left[\frac{1}{2} 
			\left(	(\bk\nabla v)_{|_{T_1}}+ (\bk\nabla v)_{|_{T_2}}\right)  \cdot \bn_e\right]^2\\
			&	\leq
			\frac{1}{2} 	\N{\bk}_{L^{\infty}(\Omega)}^2 \left(\N{(\nabla v)_{|_{T_1}} \cdot \bn_e}^2_{L^2(e)}+\N{(\nabla v)_{|_{T_2}}\cdot \bn_e}^2_{L^2(e)}\right),
		\end{align*}
		and for all $e\in\calE_h^D$ with $e\subseteq\partial T\cap \Gamma_D$, we obtain
		\begin{equation*}
			\int_{e} \left( \mvl{\bm{K}   \nabla v} \cdot \bn_e\right)^2
			=\int_{e}\left[
			(\bk\nabla v)_{|_{T}}  \cdot \bn_e\right]^2	\leq
			\N{\bk}_{L^{\infty}(\Omega)}^2 \N{(\nabla v)_{|_{T}} \cdot \bn_e}^2_{L^2(e)}.
		\end{equation*}
		From \eqref{lemmaconsistency}, using these two bounds and aggregating the facet contributions of each mesh element, we deduce the assertion.
	\end{proof}
	
	\section{Discrete coercivity}\label{coercivitydiscrete}
	First, in Lemma \ref{coerdiff} we establish the coercivity of the diffusion bilinear form $a_h^{d}$ on $V_h$ with respect to the $\N{\cdot}_{d}$-norm. Then, Lemma \ref{coeradvreac} states that the advection-reaction bilinear form $a_h^{ar}$ is coercive on $V_h$ with respect to the $\N{\cdot}_{ar}$-norm.
	Combining these two results in  in Lemma \ref{coercdiscreta}, we deduce the discrete coercivity of the diffusion-advection-reaction bilinear form $a^{dar}_h$ with respect to the  $\N{\cdot}_{dar}$-norm.
	As noted before, once we have the discrete coercivity, we also have the well-posedness of the discrete problem \eqref{variational}.
	
	The following lemma states that if the penalty parameter $\gamma$ is sufficiently large, then the diffusion bilinear form $a_h^{d}$ is coercive on $V_h$ with respect to the $\vertiii{\cdot}_d$-norm.
	\begin{lemma}[Discrete coercivity: diffusion]\label{coerdiff}
		Let
		\begin{equation}\label{gammaeps}
			\gamma_{\epsilon} := \begin{cases}
				0 & \text{if $\epsilon=1$,}\\
				\N{\bk}_{L^{\infty}(\Omega)}^2 \frac{N_{\partial}C_{tr}}{4 k_{min}} &\text{if $\epsilon=0$,} \\
				\N{\bk}_{L^{\infty}(\Omega)}^2 \frac{N_{\partial}C_{tr}}{k_{min}} &\text{if $\epsilon=-1$,}
			\end{cases}
		\end{equation}
		with $C_{tr}$ defined in \eqref{discretetraceinequality}, $N_{\partial}$ defined in \ref{Npartial} and $k_{min}$ defined in \eqref{ellipticity}.
		For all $\gamma > \gamma_\epsilon$, there exists $\alpha_d>0$, independent of $h$, such that
		\begin{equation*}
			a_h^{d}(v_h,v_h)\ge \alpha_d\vertiii{v_h}^2_{d} \quad \forall v_h\in V_h,
		\end{equation*}
		with 
		\begin{equation}\label{alphad}
			\alpha_d= \begin{cases}
				1 & \text{if $\epsilon=1$,}\\
				1+\frac{1}{2}(\epsilon-1)\N{\bk}_{L^{\infty}(\Omega)} 	 \left( \frac{N_{\partial}C_{tr}}{\gamma k_{min}}\right)^{\frac{1}{2}} &\text{if $\epsilon=0$ or $\epsilon=-1$.} 
			\end{cases}
		\end{equation}
	\end{lemma}
	\begin{proof}
		Let $v_h\in V_h$,
		\begin{equation*}
			\begin{split}
				a_h^{d}(v_h,v_h)=
				\sum_{T\in \calT_h}	\int_{T} \bk  \nabla v_h \cdot \nabla  v_h +(\epsilon-1)
				\sum_{e\in\calE_h^I \cup \calE_h^D}	\int_{e}  \mvl{\bm{K}   \nabla v_h } \cdot \jmp{v_h} 
				+\sum_{e\in\calE_{h}^{I}\cup \calE_h^D}\frac{\gamma}{h_e} 	\int_{e}   \jmp{v_h}^2.
			\end{split}
		\end{equation*}
		For $\epsilon=1$ we have $a_h^{d}(v_h,v_h)=\vertiii{v_h}^2_d$ for all $v_h\in V_h$. Hence, in this case the coercivity property is satisfied for all $\gamma>0$ with $\alpha_d=1$.
		
		Consider now the case where $ \epsilon=0$ or $ \epsilon=-1$.
		Owing to the fact that $h_e\leq h_T$ for all $e\in\calF_T$ and for all $T\in\calT_h$ and to the discrete trace inequality \eqref{discretetraceinequality}, we deduce from the bound \eqref{consistencyterm} that
		\begin{align*}
			\abs{\sum_{e\in\calE_h^I\cup \calE_h^D} \int_{e}  \mvl{\bm{K}   \nabla v_h} \cdot \jmp{v_h}} &\leq
			\N{\bk}_{L^{\infty}(\Omega)}
			\left( \sum_{T\in\calT_h}	\sum_{e\in\calF_T} \frac{h_T}{\gamma} 
			\N{\nabla v_{{h}_{|_T}}\cdot \bn_e}_{L^2(e)}^2
			\right)^{\frac{1}{2}} \abs{v_h}_J\\
			&	\leq\N{\bk}_{L^{\infty}(\Omega)}
			\left( 		\sum_{T\in\calT_h} \sum_{e\in\calF_T} \frac{h_T}{\gamma}  C_{tr} h_T^{-1}
			\N{\nabla v_{{h}_{|_T}}\cdot \bn_T}_{L^2(T)}^2
			\right)^{\frac{1}{2}} \abs{v_h}_J.
		\end{align*}
		Recalling that $N_{\partial}$ is the maximum number of facets composing the boundary of a mesh element, using the ellipticity condition \eqref{ellipticity} and then the Young's inequality, we obtain
		\begin{align*}
			\abs{\sum_{e\in\calE_h^I\cup \calE_h^D} \int_{e}  \mvl{\bm{K}   \nabla v_h} \cdot \jmp{v_h}} &\leq
			\N{\bk}_{L^{\infty}(\Omega)} 
			\left( 	 \frac{C_{tr}N_{\partial} }{\gamma}	\sum_{T\in\calT_h} 
			\N{\nabla v_{h_{|_T}}}_{L^2( T)}^2
			\right)^{\frac{1}{2}} \abs{v_h}_J\\&
			\leq
			\N{\bk}_{L^{\infty}(\Omega)}
			\left( \frac{N_{\partial}C_{tr}}{\gamma k_{min}}\right)^{\frac{1}{2}}
			\left( 		\sum_{T\in\calT_h} 
			\N{(\bk^{\frac{1}{2}}\nabla v_h)_{|_T}}_{L^2( T)}^2
			\right)^{\frac{1}{2}} \abs{v_h}_J\\
			&\leq	\N{\bk}_{L^{\infty}(\Omega)}  
			\left( \frac{N_{\partial}C_{tr}}{\gamma k_{min}}\right)^{\frac{1}{2}}
			\left(   \frac{1}{2}	
			\sum_{T\in\calT_h} 
			\N{(\bk^{\frac{1}{2}}\nabla v_h)_{|_T}}_{L^2( T)}^2
			+  \frac{1}{2}	 \abs{v_h}_J^2\right).
		\end{align*}
		Using this bound we achieve:
		\begin{align*}
			a_h^{d}(v_h,v_h)&\ge \vertiii{v_h}_d^2+\frac{1}{2}(\epsilon-1)  \N{\bk}_{L^{\infty}(\Omega)}
			\left( \frac{N_{\partial}C_{tr}}{\gamma k_{min}}\right)^{\frac{1}{2}}
			\left(	\sum_{T\in\calT_h} 
			\N{(\bk^{\frac{1}{2}}\nabla v_h)_{|_T}}_{L^2( T)}^2
			+ \abs{v_h}_J^2	\right)\\
			& = \left(1+\frac{1}{2}(\epsilon-1)\N{\bk}_{L^{\infty}(\Omega)} 	 \left( \frac{N_{\partial}C_{tr}}{\gamma k_{min}}\right)^{\frac{1}{2}}\right)
			\vertiii{v_h}_d^2.
		\end{align*}
		Choosing $\gamma$ large enough, $\gamma > \N{\bk}_{L^{\infty}(\Omega)}^2 \frac{N_{\partial}C_{tr}}{4 k_{min}}$ if $\epsilon = 0$ and $\gamma >  \N{\bk}_{L^{\infty}(\Omega)}^2 \frac{N_{\partial}C_{tr}}{k_{min}}$ if $\epsilon = -1$, we obtain the coercivity property with $\alpha_d=   1+\frac{1}{2}(\epsilon-1)\N{\bk}_{L^{\infty}(\Omega)} 	 \left( \frac{N_{\partial}C_{tr}}{\gamma k_{min}}\right)^{\frac{1}{2}}>0$.
	\end{proof}
	We notice that in this proof we could have used the bound \eqref{consistencytermbound} where $\bk$ is already inside the $L^2$-norm, instead of using \eqref{consistencyterm} and then pay $\frac{1}{k_{min}}$. However, this is done to maintain the same constant $C_{tr}$ when applying the discrete trace inequality \eqref{discretetraceinequality}, as $\bk^{\frac{1}{2}}\nabla v_h$ is not a polynomial.
	
	The next lemma establishes the coercivity property of the advection-reaction bilinear form $a_h^{ar}$ on $V_h$ with respect to the $\vertiii{\cdot}_{ar}$-norm.
	\begin{lemma}[Discrete coercivity: advection-reaction]\label{coeradvreac}
		There exists $\alpha_{ar}>0$, independent of $h$, such that
		\begin{equation*}
			a_h^{\text{ar}}(v_h,v_h)\ge \alpha_{ar}\vertiii{v_h}^2_{ar} \quad \forall v_h\in V_h,
		\end{equation*}
		where $\alpha_{ar}= \min\{1,\sigma_0\}$, with $\sigma_0$ defined in \eqref{assumptiondiv}.
	\end{lemma}
	\begin{proof}
		Let $v_h\in V_h$, 
		\begin{align*}
			a_h^{ar}(v_h,v_h)=
			&-\sum_{T\in\mathcal{T}_{h}}	\int_{T}  (\bbeta v_h) \cdot  \nabla  v_h + \sum_{T\in\mathcal{T}_{h}}	\int_{T}  \sigma v_h^2\\ 
			&+\sum_{e\in\mathcal{E}_{h}^{I}}	\int_{e}  \mvl{\bbeta  v_h} \cdot \jmp{v_h}+
			\frac{1}{2} \sum_{e\in\mathcal{E}_{h}^{I}}	\int_{e}  |\bm{\beta}\cdot \bn_{e}|\jmp{v_h}^2+
			\sum_{e\in\mathcal{E}_{h}^{N}}	\int_{e} ( \bm{\beta}  v_h) \cdot \mathbf{n}_{e} v_h.
		\end{align*}
		Integration by parts on the first term yields
		\begin{equation*}
			\sum_{T\in\calT_h}\int_{T}  (\bbeta v_h) \cdot  \nabla  v_h=\sum_{T\in\calT_h}\int_{T} \bbeta \cdot  \nabla  \left(\frac{v_h^2}{2}\right)
			=-\sum_{T\in\calT_h}\int_{T}  \text{div}(\bbeta) \frac{v_h^2}{2} +\sum_{T\in\calT_h}\int_{\partial T}  \bbeta \cdot \bn_T \frac{v_h^2}{2}.
		\end{equation*}
		Applying the DG magic formula \eqref{DG_magic} on the last term with $\bw=\bbeta$ and $\varphi=v_h^2$, we get
		\begin{equation*}
			\sum_{T\in\calT_h}\int_{\partial T}  \bbeta \cdot \bn_T \frac{v_h^2}{2}=
			\frac{1}{2}\sum_{e\in\calE_h^I}\int_{e}( \mvl{\bbeta} \cdot \jmp{v_h^2}+\jmp{\bbeta}\mvl{v_h^2})+\frac{1}{2}\sum_{e\in\calE_h^B}\int_{e}\bbeta \cdot\bn v_h^2.
		\end{equation*}
		We now use the formula  \eqref{fluxesid}:
		$
		\frac{1}{2}\mvl{\bbeta} \cdot \jmp{v_h^2}=\mvl{\bbeta v_h} \cdot \jmp{v_h},
		$
		which holds true on each interior facet and observe that $\jmp{\bbeta}=0$ on $e\in\calE_h^I$ for the regularity assumption \eqref{assumptionbeta}.
		Combining the previous steps,
		the bilinear form can be rewritten as follows:
		\begin{align*}
			a_h^{ar}(v_h,v_h)=&
			\sum_{T\in\mathcal{T}_{h}}	\int_{T}  \bigg(\sigma+\frac{\text{div}\bbeta}{2}\bigg) v_h^2 -\frac{1}{2}\sum_{e\in\calE_h^D}\int_{e}\bbeta \cdot\bn v_h^2+\frac{1}{2}\sum_{e\in\calE_h^N}\int_{e}\bbeta \cdot\bn v_h^2\\ 
			&+
			\frac{1}{2} \sum_{e\in\mathcal{E}_{h}^{I}}	\int_{e}  |\bm{\beta}\cdot \bn_{e}|\jmp{v_h}^2.
		\end{align*}
		Recalling that
		$\Gamma_D=\Gamma_{-}$ and $\Gamma_{N}=\Gamma_{+}$ for \eqref{gammaDgammaN} and using the assumption  \eqref{assumptiondiv}, we deduce that
		\begin{equation*}
			a_h^{\text{ar}}(v_h,v_h)\ge
			\sigma_0\N{v_h}_{L^{2}(\Omega)}^2+
			\frac{1}{2} \sum_{e\in\mathcal{E}_{h}}	\int_{e}  |\bm{\beta}\cdot \bn_{e}|\jmp{v_h}^2
			\ge \min\{1,\sigma_0\}\vertiii{v_h}_{ar}^2.
		\end{equation*}
		The coercivity constant is $\alpha_{ar}= \min\{1,\sigma_0\}$.
	\end{proof}
	
	By combining the two previous lower bounds we obtain that the diffusion-advection-reaction bilinear form  $a_h^{dar}$  is coercive on $V_h$ with respect to the $\vertiii{\cdot}_{dar}$-norm.
	
	\begin{lemma}[Discrete coercivity] \label{coercdiscreta}
		Let $\gamma_{\epsilon}$ be as in \eqref{gammaeps}.
		For all $\gamma > \gamma_{\epsilon}$, there exists $\alpha>0$, independent of $h$, such that
		\begin{equation*}
			a^{\text{dar}}(v_h,v_h)\ge \alpha\vertiii{v_h}^2_{dar} \quad \forall v_h\in V_h,
		\end{equation*}
		with $\alpha=\min\{1, \sigma_0, \alpha_d\}$, where  $\sigma_0$ is defined  in \eqref{assumptiondiv} and $\alpha_d$ in \eqref{alphad}.
	\end{lemma}
	\begin{proof}
		Let $v_h\in V_h$, from Lemma \ref{coerdiff} we have
		\begin{equation*}
			a_h^{d}(v_h,v_h)\ge \alpha_d\vertiii{v_h}^2_{d},
		\end{equation*}
		and from Lemma \ref{coeradvreac} we get 
		\begin{equation*}
			a_h^{ar}(v_h,v_h)\ge \min\{1,\sigma_0\}\vertiii{v_h}^2_{ar}.
		\end{equation*}
		Since $\vertiii{v_h}^2_{dar}= \vertiii{v_h}^2_{d}+ \vertiii{v_h}^2_{ar}$, we infer the thesis
		\begin{equation*}
			a_h^{dar}(v_h,v_h)=	a_h^{d}(v_h,v_h)+	a_h^{ar}(v_h,v_h)\ge \min\{1, \sigma_0, \alpha_d\} \vertiii{v_h}^2_{dar}.
		\end{equation*}
	\end{proof}
	
	\section{Boundedness}\label{s:boundedness}
	We consider separately the diffusion case from the advection-reaction one. 
	Lemma \ref{bounddiff}  states the boundedness of $a_h^d$ with respect to the norms $\N{\cdot}_{d,*}$ and $\N{\cdot}_{d}$ and in Lemma \ref{boundadvreac}  we establish the boundedness of $a_h^{ar}$ using the norms $\N{\cdot}_{ar,*}$ and $\N{\cdot}_{dar}$.
	Combining these two bounds in Lemma \ref{bounded}, we infer the boundedness of the diffusion-advection-reaction bilinear form $a^{\text{dar}}_h$ with respect to the  $\N{\cdot}_{dar,*}$-norm and to the $\N{\cdot}_{dar}$-norm.
	
	The next lemma states the boundedness of the diffusion bilinear form $a_h^{d}$ on $V_{*h}\times V_h$ using the norms $\N{\cdot}_{d,*}$ and $\N{\cdot}_{d}$. 
	\begin{lemma}[Boundedness: diffusion]\label{bounddiff}
		There exists $M_d>0$, independent of $h$, such that 
		\begin{equation*}
			a_h^{d}(v,w_h)\leq M_d\vertiii{v}_{d,*} \vertiii{w_h}_{d} \quad \forall (v,w_h)\in V_{*h}\times V_h,
		\end{equation*}
		with  $M_d=2+	\N{\bk}_{L^{\infty}(\Omega)}^{\frac{1}{2}}
		\gamma^{-\frac{1}{2}}+\N{\bk}_{L^{\infty}(\Omega)} \left(\frac{N_{\partial}C_{tr}}{\gamma k_{min}}\right)^{\frac{1}{2}}$.
	\end{lemma}
	\begin{proof}
		Let $(v,w_h)\in V_{*h}\times V_h$, we consider
		\begin{align*}
			a_h^{d}(v,w_h)=&
			\sum_{T\in\mathcal{T}_{h}}	\int_{T} \bm{K}  \nabla v \cdot \nabla  w_h -
			\sum_{e\in\mathcal{E}_{h}^{I}\cup\mathcal{E}_{h}^{D}}	\int_{e}  \mvl{ \bm{K}   \nabla v } \cdot \jmp{w_h}  \\
			&+
			\epsilon \sum_{e\in\mathcal{E}_{h}^{I}\cup\mathcal{E}_{h}^{D}}		\int_{e}  \jmp{v}\cdot  \mvl{ \bm{K}   \nabla w_h }+ \sum_{e\in\mathcal{E}_{h}^{I}\cup\mathcal{E}_{h}^{D}}		\frac{\gamma}{h_e}\int_{e}  \jmp{v}\cdot \jmp{w_h}\\
			=:&\mathfrak{T}_1+\mathfrak{T}_2+\mathfrak{T}_3+\mathfrak{T}_4.
		\end{align*}
		Owing to the Cauchy--Schwarz inequality we obtain
		\begin{align*}
			\abs{\mathfrak{T}_1 + \mathfrak{T}_4} &\leq  	\left(\sum_{T\in\mathcal{T}_{h}}	\int_{T} \bm{K}  \nabla v \cdot \nabla  v\right)^{\frac{1}{2} }
			\left(\sum_{T\in\mathcal{T}_{h}}	\int_{T} \bm{K}  \nabla w_h \cdot \nabla  w_h\right)^{\frac{1}{2}}+	\abs{v}_{J} 	\abs{w_h}_{J}
			\\&
			\leq	2\vertiii{v}_{d} 	\vertiii{w_h}_{d}.
		\end{align*}
		Since $h_e\leq h_T$ for all $e\in\calF_T$ and for all  $T\in\calT_h$, from to the bound \eqref{consistencytermbound} we get
		\begin{align*}
			\abs{\mathfrak{T}_2}&\leq
			\N{\bk}_{L^{\infty}(\Omega)}^{\frac{1}{2}}
			\left( 		\sum_{T\in\calT_h} \sum_{e\in\calF_T} \frac{h_T}{\gamma} 
			\N{(\bk^{\frac{1}{2}}\nabla v)_{|_T}\cdot \bn_e}_{L^2(e)}^2
			\right)^{\frac{1}{2}} \abs{w_h}_J\\
			&\leq
			\N{\bk}_{L^{\infty}(\Omega)}^{\frac{1}{2}}
			\gamma^{-\frac{1}{2}}	\vertiii{v}_{d,*}	\vertiii{w_h}_{d}.
		\end{align*}
		Finally, we proceed as in the proof of Lemma \ref{coerdiff} applying the discrete trace inequality \ref{Traceinverseineq} and the ellipticity condition \eqref{ellipticity} to the bound \eqref{consistencyterm}, achieving:
		\begin{align*}
			\abs{\mathfrak{T}_3}&\leq	\N{\bk}_{L^{\infty}(\Omega)}  \left(\frac{N_{\partial}C_{tr}}{\gamma k_{min}}\right)^{\frac{1}{2}}
			\left( 		\sum_{T\in\calT_h} 
			\N{(\bk^{\frac{1}{2}}\nabla w_h)_{|_T}}_{L^2( T)}^2
			\right)^{\frac{1}{2}} \abs{v}_J\\&\leq
			\N{\bk}_{L^{\infty}(\Omega)}\left(\frac{N_{\partial}C_{tr}}{\gamma k_{min}}\right)^{\frac{1}{2}}
			\vertiii{v}_{d}	\vertiii{w_h}_{d}.
		\end{align*}
		By combining all these bounds we infer the assertion with $M_d=2+	\N{\bk}_{L^{\infty}(\Omega)}^{\frac{1}{2}}
		\gamma^{-\frac{1}{2}}+\N{\bk}_{L^{\infty}(\Omega)} \left(\frac{N_{\partial}C_{tr}}{\gamma k_{min}}\right)^{\frac{1}{2}}$.
	\end{proof}
	We now prove that the advection-reaction bilinear form $a_h^{ar}$ is bounded on $V_{*h}\times V_h$ using the norms $\N{\cdot}_{ar,*}$ and $\N{\cdot}_{dar}$. 
	\begin{lemma}[Boundedness: advection-reaction]\label{boundadvreac}
		There exists $M_{ar}>0$, independent of $h$, such that 
		\begin{equation*}
			a_h^{ar}(v,w_h)\leq M_{ar}\vertiii{v}_{ar,*} \vertiii{w_h}_{dar}\quad \forall (v,w_h)\in V_{*h}\times V_h,
		\end{equation*}
		with $M_{ar}=\N{\bbeta}_{L^{\infty}(\Omega)} k_{min}^{-\frac{1}{2}}+	\N{\sigma}_{L^{\infty}(\Omega)}+2+	\N{\bbeta}_{L^{\infty}(\Omega)}^{\frac{1}{2}}$.
	\end{lemma}
	\begin{proof}
		Let $(v,w_h)\in V_{*h}\times V_h$, we consider
		\begin{align*}
			a_h^{ar}(v,w_h)=&
			-\sum_{T\in\mathcal{T}_{h}}	\int_{T}  (\bm{\beta} v) \cdot  \nabla  w_h +\sum_{T\in\mathcal{T}_{h}}	\int_{T}  \sigma  v w_h\\
			&+\sum_{e\in\mathcal{E}_{h}^{I}}	\int_{e}  \mvl{\bm{\beta}  v} \cdot \jmp{w_h}+
			\frac{1}{2} \sum_{e\in\mathcal{E}_{h}^{I}}	\int_{e}  |\bm{\beta}\cdot \mathbf{n}_{e}|\jmp{v}\cdot\jmp{w_h}  +
			\sum_{e\in\mathcal{E}_{h}^{N}}	\int_{e} ( \bm{\beta}  v) \cdot \mathbf{n}_{e} w_h\\
			=&:\mathfrak{T}_1+\mathfrak{T}_2+\mathfrak{T}_3+\mathfrak{T}_4+\mathfrak{T}_5.
		\end{align*} 
		The Cauchy--Schwarz inequality and the ellipticity condition \eqref{ellipticity} yield
		\begin{align*}
			\abs{\mathfrak{T}_1+\mathfrak{T}_2} &\leq\N{\bbeta}_{L^{\infty}(\Omega)}\N{v}_{L^2(\Omega)}
			\left(\sum_{T\in\calT_h}
			\N{\nabla w_h}_{L^2(T)}^2\right)^{\frac{1}{2}}+
			\N{\sigma}_{L^{\infty}(\Omega)} \N{v}_{L^2(\Omega)}\N{ w_h}_{L^2(\Omega)}\\
			&	\leq \N{\bbeta}_{L^{\infty}(\Omega)} k_{min}^{-\frac{1}{2}} \N{v}_{L^2(\Omega)}
			\left(\sum_{T\in\calT_h}
			\N{\bk^{\frac{1}{2}}\nabla w_h}_{L^2(T)}^2\right)^{\frac{1}{2}}+
			\N{\sigma}_{L^{\infty}(\Omega)} \N{v}_{L^2(\Omega)}\N{ w_h}_{L^2(\Omega)}\\
			&\leq  \N{\bbeta}_{L^{\infty}(\Omega)} k_{min}^{-\frac{1}{2}}\vertiii{v}_{ar}
			\vertiii{w_h}_{d}+
			\N{\sigma}_{L^{\infty}(\Omega)} \vertiii{v}_{ar}\vertiii{ w_h}_{ar}.
		\end{align*}
		Again owing to the Cauchy--Schwarz inequality:
		\begin{align*}
			\abs{\mathfrak{T}_4+\mathfrak{T}_5} &\leq 2  \left(\frac{1}{2} \sum_{e\in\calE_h^I \cup \calE_h^N}\int_{e} \abs{\bbeta\cdot\bn_e} \jmp{v}^2\right)^{\frac{1}{2}}\  \left(\frac{1}{2} \sum_{e\in\calE_h^I \cup \calE_h^N}\int_{e} \abs{\bbeta\cdot\bn_e} \jmp{w_h}^2\right)^{\frac{1}{2}}
			\\
			&
			\leq 2  \vertiii{v}_{ar}\vertiii{ w_h}_{ar}.
		\end{align*}
		Moreover, using the continuity of $\bbeta$ \eqref{assumptionbeta} and the Cauchy--Schwarz inequality, we infer
		\begin{align*}
			\abs{\mathfrak{T}_3} =\sum_{e\in\mathcal{E}_{h}^{I}}	\int_{e} \bbeta \mvl{v}\cdot \jmp{w_h}&\leq  \left(2\sum_{e\in\calE_h^I}\int_{e} \abs{\bbeta\cdot\bn_e}\mvl{v}^2\right)^{\frac{1}{2}}\  \left(\frac{1}{2} \sum_{e\in\calE_h^I}\int_{e} \abs{\bbeta\cdot\bn_e} \jmp{w_h}^2\right)^{\frac{1}{2}}\\
			&\leq  \left(   \N{\bbeta}_{L^{\infty}(\Omega)} \sum_{T\in\calT_h}\N{v}^2_{L^2(\partial T)} \right)^{\frac{1}{2}}\  \vertiii{ w_h}_{ar}\\
			&
			\leq 	\N{\bbeta}_{L^{\infty}(\Omega)}^{\frac{1}{2}}
			\vertiii{v}_{ar,*}\vertiii{ w_h}_{ar},        
		\end{align*}
		where in the second step we use the formula
		$2\mvl{v}^2=\frac{1}{2}(v_1+v_2)^2\leq v_1^2+v_2^2$ that follows from the Young's inequality.
		Collecting all the above bounds yields the thesis with $M_{ar}=\N{\bbeta}_{L^{\infty}(\Omega)} k_{min}^{-\frac{1}{2}}+	\N{\sigma}_{L^{\infty}(\Omega)}+2+	\N{\bbeta}_{L^{\infty}(\Omega)}^{\frac{1}{2}}$.
	\end{proof}
	The following lemma establishes the boundedness of the diffusion-advection-reaction bilinear form $a^{dar}_h$ on $V_{*h}\times V_h$ using the norms $\N{\cdot}_{dar,*}$ and  $\N{\cdot}_{dar}$.
	\begin{lemma}[Boundedness]\label{bounded}
		There exists $M>0$, independent of $h$, such that 
		\begin{equation*}
			a_h^{\text{dar}}(v,w_h)\leq M\vertiii{v}_{dar,*} \vertiii{w_h}_{dar}\quad \forall (v,w_h)\in V_{*h}\times V_h,
		\end{equation*}
		where $M=M_d+M_{ar}$, with $M_d$ as in Lemma \ref{bounddiff} and $M_{ar}$ as in Lemma \ref{boundadvreac}.
	\end{lemma}
	\begin{proof}
		Let $(v,w_h)\in V_{*h}\times V_h$, Lemma \ref{bounddiff} states
		$$
		a_h^{d}(v,w_h)\leq M_d\vertiii{v}_{d,*} \vertiii{w_h}_{d},
		$$
		and Lemma \ref{boundadvreac} yields
		$$
		a_h^{ar}(v,w_h)\leq M_{ar}\vertiii{v}_{ar,*} \vertiii{w_h}_{dar}.
		$$
		For the norm definitions, from these two bounds we derive the assertion with 
		$M=M_d+M_{ar}$.
	\end{proof}
	
	\section{Error analysis}\label{s:erroranalysis}
	The aim of this section is to infer a convergence rate in $h$ for the Galerkin error $u-u_h$ measured in the $\N{\cdot}_{dar}$-norm.	
	Since consistency, discrete coercivity and  boundedness are satisfied, we can apply Theorem \ref{errortheorem} to the DG method \eqref{variational} for any discrete space $V_h\subseteq \mathbb{P}^{p}_d(\calT_h)$.
	\begin{theorem}[Error estimate]\label{errorestimate}
		Let $u\in V_{*h}$ solve \eqref{variationalcontinuous} and let $u_h $ solve \eqref{variational} with the penalty parameter $\gamma$ as in Lemma \ref{coerdiff}. 
		Then, 
		\begin{equation}\label{darnormconverge}
			\vertiii{u-u_h}_{dar}\leq \left(1+\frac{M}{\alpha}\right) \inf_{v_h\in V_h}\vertiii{u-v_h}_{dar,*},
		\end{equation}
		with $\alpha$ as in Lemma \ref{coercdiscreta} and $M$ as in Lemma \ref{bounded}.
	\end{theorem}
	
	In order to deduce the $h$-convergence,
	given the quasi-optimality inequality \eqref{darnormconverge},
	we need to estimate the right-hand side
	$ \inf_{v_h\in V_h}\vertiii{u-v_h}_{dar,*}$.
	The estimation of this quantity depends on the approximation properties of the discrete space $V_h$.
	
	If we choose as discrete space $V_h$  in the DG scheme \eqref{variational} the broken space of degree-$p$ polynomials $\mathbb{P}^p_d(\calT_h)$
	we obtain the standard DG method.
	The Bramble-Hilbert Lemma states the 
	approximation properties that are achieved using the full polynomial space \cite[Lemma 4.3.8]{brenner2008mathematical}.
	\begin{lemma}[Bramble-Hilbert]
		Let $\calT_{\mathcal{H}}$ be a mesh sequence with the star-shaped property \ref{starshaped}.
		Then, for all $h \in \mathcal{H}$, all $ T \in \calT_h$ and for all polynomial degrees $p$, there exist a polynomial $v \in {P}^p_d(T)$ and $C_{opt}>0$, independent of $T$ and $h$,  such that, for all $ q \in \{0,\dots,p+1\}$ and all $u \in H^q(T)$, there holds
		\begin{equation}\label{BH}
			\abs{u-v}_{H^m(T)}\leq C_{opt} h_T^{q-m}\abs{v}_{H^q(T)} \quad \forall m\in\{0,\dots,q\}.
		\end{equation}
	\end{lemma}
	
	The next theorem establishes the convergence rate in $h$ for the error $u-u_h$ in the $\N{\cdot}_{dar}$-norm.
	From the quasi-optimality inequality \eqref{darnormconverge} we deduce the convergence rate for the DG method, using the the error bound  \eqref{BH}.
	We recall that 
	$ V_{*} :=H^1_{\Gamma_D}(\Omega) \cap H^2(\calT_h)$.
	
	\begin{theorem}[Convergence rate]\label{finalerror}
		Let $ p \in \mathbb{N} $ and let $u\in V_{*}\cap H^{p+1}(\calT_h)$ solve \eqref{eq}-\eqref{Neumann}.
		Let $u_h $ solve \eqref{variational} with discrete space $V_h=\mathbb{P}_d^p(\calT_h)$ and the penalty parameter $\gamma$ as in Lemma \ref{coerdiff}.
		Then,  the following convergence rate holds
		\begin{equation}
			\begin{aligned}\label{fullpoly}
				\vertiii{u-u_h}_{dar}\leq & \left(1+\frac{M}{\alpha}\right)C_{opt} \Biggl(
				\sum_{T\in \calT_h}	
				\left[\left(1+C_{tr}\right)\N{\bk}_{L^{\infty}(T)} 	+2 C_g \gamma  C_{tr}
				\right.	 \\&\left.
				+C_{tr} 
				\left(\N{\bbeta}_{L^{\infty}(T)} 
				+1\right)h_T
				+h_T^2	\right]  h_T^{2p}\abs{u}^2_{H^{p+1}(T)}\Biggr)^{\frac{1}{2}}.
			\end{aligned}
		\end{equation}
		with $\alpha$ as in Lemma \ref{coercdiscreta} and $M$ as in Lemma \ref{bounded}.
	\end{theorem}
	\begin{proof}
		Since we have the quasi-optimality inequality \eqref{darnormconverge}, we now focus on estimating the quantity
		$	\inf_{v_h\in V_h}\vertiii{u-v_h}_{dar,*}$.
		We recall the definition of the $\N{\cdot}_{dar,*}$-norm, for $v\in V_{*h}$,
		\begin{align*}
			\vertiii{v}^2_{dar,*}=&\sum_{T\in \calT_h}\int_{T}\bk\nabla v\cdot \nabla v+\sum_{e\in\calE_{h}^{I}\cup \calE_h^D}\frac{\gamma}{h_e}\int_{e} \jmp{v}^2+ \frac{1}{2}\sum_{e\in\calE_{h}}\int_{e}\abs{\bbeta\cdot\bn_e}\jmp{v}^2\\&+
			\N{v}^2_{L^{2}(\Omega)}+\sum_{T\in \calT_h}h_T\N{\bk^{\frac{1}{2}}\nabla v\cdot \bn_T}^2_{L^2(\partial T)}+	\sum_{T\in \calT_h}\N{ v}^2_{L^2(\partial T)}.
		\end{align*}
		We recall the assumption that the mesh is graded \eqref{gradedmesh} and 
		observe that, from the Young's inequality, $\jmp{v}^2=(v_1-v_2)^2\leq 2(v_{1}^2+v_2^2)$ on an internal facet $e=\partial T_1\cap \partial T_2$
		and $\jmp{v}^2=v^2$ on a boundary facet $e$. Using these two facts in the second and third terms of the norm, we rearrange the sums over mesh elements as sums over elements and obtain the bound:
		\begin{align*}
			\vertiii{v}^2_{dar,*}\leq&
			\sum_{T\in \calT_h}	\biggl(\N{\bk}_{L^{\infty}(T)} \N{\nabla v}^2_{L^2(T)} +2 C_g\frac{\gamma}{h_T} \N{v}^2_{L^2(\partial T)}
			+\N{\bbeta}_{L^{\infty}(T)}\N{v}^2_{L^2(\partial T)}\\&
			+\N{v}^2_{L^2(T)}
			+\N{\bk}_{L^{\infty}(T)}h_T\N{\nabla v}^2_{L^2(\partial T)}
			+\N{v}^2_{L^2(\partial T)}\biggr).
		\end{align*}
		Owing to the discrete trace inequality \eqref{discretetraceinequality}, we get
		\begin{align*}
			\vertiii{v}^2_{dar,*}\leq&
			\sum_{T\in \calT_h}	\biggl(\N{\bk}_{L^{\infty}(T)} \N{\nabla v}^2_{L^2(T)} +
			2 C_g \frac{\gamma}{h_T} C_{tr}h_T^{-1}\N{v}^2_{L^2( T)}\\&
			+\N{\bbeta}_{L^{\infty}(T)}C_{tr}h_T^{-1}\N{v}^2_{L^2( T)}
			+\N{v}^2_{L^2(T)}\\&
			+\N{\bk}_{L^{\infty}(T)}h_T C_{tr}h_T^{-1}\N{\nabla v}^2_{L^2( T)}
			+ C_{tr}h_T^{-1}\N{v}^2_{L^2( T)}\biggr).
		\end{align*}
		Considering the quantity of our interest we have 
		\begin{align*}
			\inf_{v_h\in V_h}\vertiii{u-v_h}^2_{dar,*}\leq &
			\sum_{T\in \calT_h}		\inf_{v_h\in V_h}
			\biggl[\left(1+C_{tr}\right)\N{\bk}_{L^{\infty}(T)} 	\N{\nabla (u-v_h)}^2_{L^2(T)} \\&
			+\left(2 C_g \gamma h_T^{-2} C_{tr}
			+\N{\bbeta}_{L^{\infty}(T)} C_{tr} h_T^{-1}
			+1
			+C_{tr} h_T^{-1}\right)\N{u-v_h}^2_{L^2(T)}\biggr].
		\end{align*}
		The optimal polynomial approximation properties of the Bramble-Hilbert Lemma \eqref{BH} yield
		\begin{align*}
			\inf_{v_h\in V_h}\vertiii{u-v_h}^2_{dar,*}\leq &		\sum_{T\in \calT_h}	
			\biggl[\left(1+C_{tr}\right)\N{\bk}_{L^{\infty}(T)} 	
			C_{opt}^2 h_T^{2p}\abs{u}^2_{H^{p+1}(T)}\\&
			+
			\biggl(2 C_g \gamma h_T^{-2} C_{tr}
			+\N{\bbeta}_{L^{\infty}(T)} C_{tr} h_T^{-1}
			+1
			+C_{tr} h_T^{-1}\biggr)	\times\\&	C_{opt}^2 h_T^{2(p+1)}\abs{u}^2_{H^{p+1}(T)}\biggr].
			\\=&
			\sum_{T\in \calT_h}	
			\biggl[\left(1+C_{tr}\right)\N{\bk}_{L^{\infty}(T)} 	+2 C_g \gamma  C_{tr}
			\\&
			+C_{tr} 
			\left(\N{\bbeta}_{L^{\infty}(T)} 
			+1\right)h_T
			+h_T^2	\biggr]C_{opt}^2h_T^{2p}\abs{u}^2_{H^{p+1}(T)}.
		\end{align*}
		By combining this bound with the inequality \eqref{darnormconverge} we infer the assertion.
	\end{proof}
	
	\chapter{Quasi-Trefftz space}\label{Chapter5}
	In this chapter we define the quasi-Trefftz space, prove its high-order approximation properties and describe an algorithm for the construction of the basis functions.
	First, in Section \ref{s:DefNot} we fix some notation and give the definitions. In Section \ref{s:ApproxT} we prove approximation bounds for the local quasi-Trefftz space and then in Section \ref{s:ADR-QT}
	we prove the $h$-convergence of the quasi-Trefftz DG method for the diffusion-advection-reaction problem.
	In Section \ref{s:BasisT} we define a family of basis functions and  we explicitly construct them in Section \ref{s:Algorithm}.
	
	\section{Definitions and notation}\label{s:DefNot}
	We call multi-indices the vectors $\mi=(i_{1},\ldots,i_{d})
	\in \IN^{d}$.
	We define their length $|\mi|:=i_{1}+\cdots+i_{d}$ and we establish a partial order
	$\mi\leq \mj$ if $i_k\leq j_k$ for all $k\in\{1,\dots,d\}$.
	The factorial and the binomial coefficient are defined, respectively, as
	$$\mi!:=i_{1}!\cdots i_{d}!	\quad\text{and}\quad
	\binom\mi\mj:=\frac{\mi!}{\mj!(\mi-\mj)!}
	=\binom{i_{1}}{j_{1}}\cdots
	\binom{i_{d}}{j_{d}}.$$
	We use standard multi-index notation for partial derivatives and monomials:
	$$D^\mi f := 
	\partial_{x_1}^{i_{1}}\cdots\partial_{x_d}^{i_{d}}f, \quad \text{and} \quad
	\bx^\mi
	=x_1^{i_{1}}\cdots x_d^{i_{d}}.$$
	We recall the following formula
	\begin{equation}\label{formulaMoiola}
		\sum_{\bi\in\mathbb{N}^d, \abs{\bi}=k} \frac{1}{\bi!}=\frac{d^k}{k!} \quad\forall  k\in\mathbb{N},
	\end{equation}
	obtained from the multinomial theorem
	$$
	(x_1+\dots+x_d)^k=\sum_{\bi\in\mathbb{N}^d, \abs{\bi}=k} \frac{k!}{\bi !}\bx^{\bi } \quad \forall \bx\in\mathbb{R}^d, k\in\mathbb{N},$$
	by choosing
	$\bx=(1,\dots,1)$ \cite[p.~198]{moiola2011trefftz}.
	We use the canonical basis of $\IR^d$, namely $\{\be_k\in\IR^d,1\leq k\leq  d, (\be_k)_l=\delta_{kl}\}$.
	We recall the Leibniz product rule for multi-indices:
	\begin{equation}\label{def:Leibniz}
		D^\mi(f g)
		=\sum_{\substack{\mj\in\IN^{d},\;\mj\le\mi}}
		\binom\mi\mj D^\mj fD^{\mi-\mj} g.
	\end{equation}
	Let $p\in\mathbb{N}$ , let $\Upsilon\subset \mathbb{R}^d$ be an open, bounded, Lipschitz set and let $\bz\in\Upsilon$.
	\begin{Def}[Taylor polynomial]
		Given $v\in C^{p}(\Upsilon)$, the \textit{Taylor polynomial} of order $p+1$ of $v$, centred at $\bz\in \Upsilon$, is defined as
		$$	\tayt^{p+1}_{\bz}[v](\bx):=
		\sum_{|\mj|\le p}\frac1{\mj!}	D^\mj v(\bz)
		(\bx-\bz)^{\mj}
		.$$
	\end{Def}
	Notice that $\tayt^{p+1}_{\bz}[v]$ is a polynomial of degree at most $p$.
	We highlight the following property: for every multi-index  $\mi$
	\begin{equation}\label{propTay}
		\begin{split}
			D^\mi \tayt^{p+1}_{\bz}[v](\bx)&=
			\sum_{\substack{\abs{\bj}\leq p\\ \bj\ge \bi}}\frac{1}{\mj!}D^{\mj } v(\bz)\frac{\mj!}{(\mj-\mi)!}
			(\bx-\bz)^{\mj-\mi}\\
			&=\sum_{\abs{\bm{\gamma}}\leq p-\abs{\bi}}\frac{1}{\bm{\gamma}!}D^{\bm{\gamma}+\bi } v(\bz)
			(\bx-\bz)^{\bm{\gamma}}= \tayt^{p+1-\abs{\bi}}_{\bz}[D^{\bi}v](\bx).
		\end{split}
	\end{equation}
	From the first step, evaluated at $\bz$, it follows that
	\begin{equation}\label{Tayprop}
		D^\mi \tayt^{p+1}_{\bz}[v](\bz)=
		\left\{
		\begin{array}{ll}
			D^\mi v(\bz) &\text{ if }|\mi|\le p,\\
			0 &\text{ if }|\mi|>p.
		\end{array}
		\right.
	\end{equation}
	\noindent
	We also recall the Lagrange's form of the Taylor remainder \cite[Cor.~3.19]{callahan2010advanced}: if $v\in C^{p+1}(\Upsilon)$ and the segment $S$ with extremes $\bz$ and $\bx$ is  contained in $\Upsilon$, then exists $ \bx_*\in S$ such that 
	\begin{equation}\label{eq:TaylorRem}
		v(\bx)-\tayt^{p+1}_{\bz}[v](\bx)=
		\sum_{|\mj|=p+1}\frac1{\mj!}D^\mj v(\bx_*)(\bx-\bz)^{\mj}.
	\end{equation}
	If $\Upsilon$ is a mesh element $T\in \calT_h$, which is the case of our interest, we denote  the centre of the Taylor expansion as $\bx_T\in T$ and the Taylor polynomial as $\tayt^{p+1}_T[v](\bx)=\tayt^{p+1}_{\bx_T}[v](\bx)$.
	
	For each $T\in \calT_h$, for $m\in\IN$ we use the standard $C^m$ norms and seminorms denoted by
	\begin{equation*}
		\N{v}_{C^0(T)}:=\sup_{\bx\in T}|v(\bx)|,\qquad
		\abs{v}_{C^m(T)}:=\max_{|\mi|=m}\N{D^\mi v}_{C^0(T)}.
	\end{equation*}
	We define the global space
	$C^{m}(\calT_h):=\{v \in L^{2}(\Omega) \mid 
	v_{|_T} \in C^{m}(T)  \quad \forall T \in \mathcal{T}_{h} \}$
	and the seminorms
	\begin{equation*}
		\abs{v}_{C^m\Th}:=\max_{T\in\calT_h} \abs{v_{|_T}}_{C^m(T)}.
	\end{equation*}
	
	We now define the quasi-Trefftz space and prove its optimal approximation properties for a general linear PDE with varying coefficients. This is possible since no special property of the diffusion-advection-reaction equation will be used in the proof of Theorem \ref{prop:Approx}, which provides local best-approximation bounds.
	Then, we will focus on the homogeneous diffusion-advection-reaction equation of our model problem.
	
	We consider a linear partial differential operator of order $m\in\mathbb{N}$ denoted as
	\begin{equation}\label{linearop}
		\calM:=\sum_{|\bm{j}|\leq m} \alpha_{\bm{j}} D^{\bm{j}},
	\end{equation}
	with coefficients $\alpha_{\bj}:\Omega\to \mathbb{R}$ for $\abs{\bj}\leq m$.
	We introduce local Trefftz and quasi-Trefftz spaces.
	\begin{Def}[Local Trefftz space]
		Let $p\in \mathbb{N}$, let $T\in\calT_h$ and assume that, for all $\abs{\bj}\leq m$, the coefficients $\alpha_{\bj}$ of $\calM$ are constant inside the mesh element $T$.
		We define the \textit{polynomial Trefftz space} 
		for the homogeneous equation $\calM u=0$ 
		on the mesh element $T$ as
		$$
		\mathbb{T}^{p}(T):=\{v\in \mathbb{P}_d^{p}(T) \mid \calM v=0 \text{ in } T\}.
		$$
	\end{Def}
	This is the space of degree-$p$ polynomials such that are exact solutions of $\calM u=0$ in $T$.
	We observe that, by definition, $\mathbb{T}^p(T)$ is a subspace of the full polynomial space  $\mathbb{P}_d^p(T)$.
	
	\begin{rem}
		When all the terms in the differential operator $\calM$ are derivatives of the same order, such as, for example, in the Laplace's equation and in the wave equation, the polynomial Trefftz space  $\mathbb{T}^{p}(T)$ offers the same orders of $h$-convergence
		as the full polynomial space $\mathbb{P}^{p}_d(T)$ \cite[Lemma 1]{moiola2018space}.
		On the other hand, if the differential operator $\calM$ includes derivatives of different orders, in general,
		the convergence rates for the Trefftz subspace of degree-$p$ polynomials are not the same as those for full polynomial space of the same degree \cite{gomez2023polynomial}.
		For example, in \cite{gomez2023polynomial} for the linear time-dependent Schr\"{o}dinger equation  the same accuracy as for full polynomials of  degree at most $p$ is reached using Trefftz polynomials of degree $2p$.
		Moreover, when the derivatives are of different orders, the polynomial Trefftz space defined above can also be trivial, such as in the case of the Helmholtz’s equation, which does not admit polynomial solutions.
		The quasi-Trefftz space, instead, is always rich enough to give the same approximation bounds as the full polynomial space of the same degree, as we will see below in Theorem \ref{prop:Approx}.
		This suggests that also for problems with piecewise-constant coefficients  could be a good idea to use quasi-Trefftz methods.
	\end{rem}
	\begin{Def}[Local quasi-Trefftz space]
		Let $p\in\mathbb{N}$, let $T\in\calT_h$ and let $\bx_T\in T$. Assume that the coefficients $\alpha_\bj\in C^{\max\{p-m,0\}}(T)$ for all $\abs{\bj}\leq m$.
		We define the \textit{quasi-Trefftz space} for the homogeneous equation $\calM u=0$ on the mesh element $T$ as
		\begin{equation}\label{QTlocal}
			\QT^{p}(T):=\{ v\in \mathbb{P}^{p}_d(T) \mid D^{\bm{i}}\calM v (\mathbf{x}_{T})=0, \forall \bm{i}\in \mathbb{N}^{d}, |\bm{i}|\leq p-m\}.
		\end{equation}
	\end{Def}
	\noindent
	This is the space of degree-$p$ polynomials such that the Taylor polynomial of order $p-m+1$ of their image by the differential operator $\calM$ vanishes at $\bx_T$.
	Notice that, by definition, $\QT^p(T)$ is a subspace of the full polynomial space  $\mathbb{P}_d^p(T)$ and that, for $p<m$, $\QT^p(T)$ coincides with $\mathbb{P}_d^p(T)$.
	
	We notice that, in general,  $\QT^p(T)\not\subseteq\QT^{p+1}(T)$.
	In order to show this, we can consider, for example, as linear partial differential operator $\calM$ the second-order diffusion-advection-reaction operator $\calL$ with $\bk(\bx)=Id$, $\bbeta(\bx)=(1,\cdots,1)^T$ and  $\sigma(\bx)=\frac{2}{x_1^2+1}$.
	If we consider the point $\bx_T=\bm{0}$ and the function $v(\bx)=x_1^2+1\in  \mathbb{P}^2_d(T)$,  then, $\calL v(\bx)=-2+2x_1+2$, so $\calL v(\bx_T)=0$.
	Hence $v\in \QT^2(T)$ , but $\partial_{x_1}\calL v(\bx)=2$, implying that $ v\in\QT^2(T)\setminus \QT^3(T)$.
	This means that, in general, for increasing polynomial degree $p$ the quasi-Trefftz spaces are not nested.
	
For $\abs{\bi}\leq p-m$, using the definition of $\calM$ \eqref{linearop} and the Leibniz formula \eqref{def:Leibniz}, we obtain, for any  $v\in C^{\abs{\bi}+m}(T)$,
	\begin{equation}\label{derimage}
		\begin{aligned}
			D^\mi \calM v(\bx_T)&=\sum_{\abs{\bj}\leq m} D^\bi \left(\alpha_{\bj} (\bx_T)D^\bj v(\bx_T)\right)\\&
			=\sum_{\abs{\bj}\leq m}\sum_{\br\leq \bi} \binom{\bi}{\br} D^\br \alpha_{\bj} (\bx_T)D^{\bi-\br+\bj}v(\bx_T).
		\end{aligned}
	\end{equation}
	From \eqref{derimage}, the space $\QT^p(T)$ is well-defined if, for all $\abs{\bj}\leq m$, $\alpha_{\bj}\in C^{\max\{p-m,0\}}$ in a neighbourhood of $\bx_T$. 
	
	The global quasi-Trefftz space on the whole mesh is defined as follows.
	\begin{Def}[Global quasi-Trefftz space]
		For $p\in\mathbb{N}$, we define the \textit{global quasi-Trefftz space} for the homogeneous equation $\calM u=0$ as
		\begin{equation}\label{globalQT}
			\QT^{p}(\mathcal{T}_{h}):=\{v \in L^{2}(\Omega) \mid 
			v_{|_T} \in \QT^{p}(T)  \quad \forall T \in \mathcal{T}_{h} \}.
		\end{equation}
	\end{Def}
	\noindent
	We want to choose this space as discrete space $V_h$ in the DG method \eqref{variational}.
	
	\section{Approximation properties of the quasi-Trefftz space}\label{s:ApproxT}
	The following theorem states the approximation properties of the local quasi-Trefftz spaces. 
	We prove that Taylor polynomials of smooth solutions of the homogeneous linear PDE considered are quasi-Trefftz functions. This implies that we can approximate smooth solutions  in $\QT^p(T)$ with optimal convergence rates with respect to the mesh size, meaning that
	the orders of $h$-convergence for quasi-Trefftz spaces are the same as those for full polynomial spaces.
	
	\begin{theorem}\label{prop:Approx}
		Let $ p \in \mathbb{N} $ and $T\in\calT_h$. Assume that $\alpha_{\bm{j}}\in C^{\max\{p-m,0\}}(T)$ for all $\abs{\bj}\leq m$ and  that $u\in C^{p+1}(T)$ satisfies $\calM u=0$ in $T$.
		
		Then the Taylor polynomial $\mathsf{T}^{p+1}_T[u]\in \QT^p(T)$.
		
		Moreover, if $T$ is star-shaped with respect to $\mathbf{x}_T$, then, for all $q\in\mathbb N$ such that $q\le p$,
		\begin{equation}\label{eq:Approx}
			\inf_{v\in\QT^p(T)}\abs{u-v}_{C^q(T)}
			\le \frac{d^{p+1-q}}{(p+1-q)!} h_T^{p+1-q} \abs{u}_{C^{p+1}(T)}.
		\end{equation}
	\end{theorem}
	\begin{proof}
		First we prove that $\mathsf{T}^{p+1}_T[u]\in \QT^p(T)$.
		By definition, we already know that  $\mathsf{T}^{p+1}_T[u]\in \mathbb{P}^p_d(T)$.
		For all $|\mi|\le p-m$, we have
		\begin{align*}
			D^\mi \calM \mathsf{T}^{p+1}_T[u](\bx_T)&=\sum_{\abs{\bj}\leq m}\sum_{\br\leq \bi} \binom{\bi}{\br} D^\br \alpha_{\bj}(\bx_T) D^{\bi-\br+\bj} \mathsf{T}^{p+1}_T[u](\bx_T)\\
			&=\sum_{\abs{\bj}\leq m}\sum_{\br\leq \bi} \binom{\bi}{\br} D^\br \alpha_{\bj}(\bx_T) D^{\bi-\br+\bj} u(\bx_T)\\&=
			D^\mi\calM u (\bx_T).
		\end{align*}
		The first step is the identity \eqref{derimage}  with $v=\mathsf{T}^{p+1}_T[u]\in C^{\abs{\bi}+m}(T)$, while the second step follows from the property \eqref{Tayprop} with partial derivatives of order at most equal to $\abs{\bi}+m\leq p$, and the third one is \eqref{derimage} again  with $v=u\in C^{\abs{\bi}+m}(T)$.
		Since $u$ is solution of $\calM u=0$ in $T$, the last term $D^\mi\calM u (\bx_T)$ vanishes, implying that $D^\mi \calM \mathsf{T}^{p+1}_T[u](\mathbf{x}_{T})=0$ for all $|\mi|\le p-m$.
		This shows that the Taylor polynomial $\mathsf{T}^{p+1}_T[u]$ belongs to the quasi-Trefftz space $\QT^p(T)$ defined in \eqref{QTlocal}.
		
		In order to prove the inequality \eqref{eq:Approx}, fix $q\in\mathbb N$ such that $q\le p$.
		Using the fact that $\mathsf{T}^{p+1}_T[u]\in \QT^p(T)$, the best approximation error is estimated by the error of the Taylor polynomial
		$$\inf_{v\in\QT^p(T)}\abs{u-v}_{C^q(T)}\le\abs{u-\mathsf{T}^{p+1}_T[u]}_{C^q(T)}.$$
		Using the $\abs{ \cdot }_{C^q}$-norm definition, the linearity of the derivatives and the identity $D^\mi \tayt^{p+1}_T[u]=\tayt^{p+1-|\mi|}_T[D^\mi u]$
		for $|\mi|=q\le p$ from \eqref{propTay}, we obtain
		\begin{align*}
			\abs{u-\mathsf{T}^{p+1}_T[u]}_{C^q(T)}
			&=\max_{\mi\in\IN^{d},\; |\mi|=q}\N{D^\mi (u-\mathsf{T}^{p+1}_T[u])}_{C^0(T)}
			\\
			&=\max_{\mi\in\IN^{d},\; |\mi|=q}
			\N{D^\mi u-\mathsf{T}^{p+1-q}_T[D^\mi u]}_{C^0(T)}.
		\end{align*}
		Since $u\in C^{p+1}(T)$ and $T$ is star-shaped with respect to $\mathbf{x}_T$, we can consider the Lagrange's form of the Taylor remainder \eqref{eq:TaylorRem} and estimate the last term as follows:
		\begin{align*}
			\max_{\mi\in\IN^{d},\; |\mi|=q}\N{D^\mi u-\mathsf{T}^{p+1-q}_T[D^\mi u]}_{C^0(T)}
			&	\le \max_{\mi\in\IN^{d},\; |\mi|=q}
			\sum_{\substack{|\mj|=p+1-q}}
			\frac1{\mj!}		\sup_{\bx\in T}\abs{	D^{\mi+\mj}u(\bx_*)(\bx-\bx_T)^\mj}.
			\\&
			\le \frac{d^{p+1-q}}{(p+1-q)!} h_{T}^{p+1-q} \abs{u}_{C^{p+1}(T)},
		\end{align*}
		where we use the formula
		$\sum_{|\mj|=p+1-q}\frac1{\mj!}= \frac{d^{p+1-q}}{(p+1-q)!}$, obtained from  \eqref{formulaMoiola} with $k=p+1-q$.
		From the chain of inequalities, the thesis follows.
	\end{proof}
	
	\section{Convergence analysis of the quasi-Trefftz DG method}\label{s:ADR-QT}
	\subsection{Quasi-Trefftz space for the homogenous diffusion-advection-re\-ac\-tion equation}
	We now focus on the diffusion-advection-reaction operator $\calL$ with variable coefficients $\bk:\Omega\to \mathbb{R}^{d\times d}$, $\bbeta:\Omega\to \mathbb{R}^d$ and $\sigma:\Omega\to \mathbb{R}$:
	\begin{equation*}
		\mathcal{L}v=	\text{div}\left(-\bm{K}  \nabla v  +\bm{\beta}  v \right) +\sigma v.
	\end{equation*}
	Since $\calL$ is a linear partial differential operator of the second order, the quasi-Trefftz space \eqref{QTlocal} for the homogeneous diffusion-advection-reaction equation on a mesh element $T\in\mathcal{T}_{h}$ is
	\begin{equation}\label{localdar}
		\QT^{p}(T)=\{ v\in \mathbb{P}^{p}_d(T) \mid D^{\bm{i}} (\mathcal{L}v) (\mathbf{x}_{T})=0, \forall \bm{i}\in \mathbb{N}^{d}, |\bm{i}|\leq p-2\}, \quad  p\in \mathbb{N}.
	\end{equation}
	As seen before, this space is well-defined if the PDE coefficients are sufficiently smooth. 
	The operator $\calL$ can be written in the form \eqref{linearop} with $m=2$, for which we have shown that the coefficients $\alpha_\bj$ need to belong to $C^{\max\{p-2,0\}}(T)$. 
	In order to deduce the regularity assumptions of $\bk$, $\bbeta$ and $\sigma$, under which the quasi-Trefftz space is well-defined, we expand the operator $\calL$, using linearity and Leibniz formula \eqref{def:Leibniz}:
	\begin{align*}
		\mathcal{L}v &= \sum_{j=1}^{d} D^{\bm{e}_{j}}( -\bm{K}  \nabla v + \bm{\beta}  v )_{j}+\sigma v\\
		&=  
		\sum_{j=1}^{d} \biggl(-\sum_{m=1}^{d}  D^{\bm{e}_{j}}\big(\bm{K}_{jm}  D^{\bm{e}_{m}} v\big) +	
		D^{\bm{e}_{j}}  (\bm{\beta}_{j}  v ) \biggl)+\sigma v\\
		&=\sum_{j=1}^{d}  \sum_{\substack{\bm{\ell}\in\IN^{d},\;\bm{\ell}\le\bm{e}_{j}}}	\binom{\bm{e}_{j}}{\bm{\ell}}
		\biggl(-\sum_{m=1}^{d}
		D^{\bm{\ell}}\bm{K}_{jm}  D^{\bm{e}_{j}-\bm{\ell} +\bm{e}_{m}} v+	
		D^{\bm{\ell}}  \bm{\beta}_{j} 	 D^{\bm{e}_{j}-\bm{\ell}} v\biggl)  +\sigma v\\
		&=\sum_{j=1}^{d}  
		\biggl[-\sum_{m=1}^{d}\bigg(
		\bm{K}_{jm}  D^{\bm{e}_{j} +\bm{e}_{m}} v+D^{\bm{e}_j}	\bm{K}_{jm}  D^{\bm{e}_{m}} v\bigg)+
		\bm{\beta}_{j} 	 D^{\bm{e}_{j}} v+D^{\bm{e}_j} \bbeta_j v\biggl]  +\sigma v.
	\end{align*}
	It follows that $ D^{\bm{\ell}}\bm{K}_{jm}$, $D^{\bm{\ell}}  \bm{\beta}_{j}$ and $ \sigma $ need to belong to $C^{\max\{p-2,0\} } (T)$ for all $\bm\ell\leq\bm{e}_{j}$ and for all $ j,m=1,\dots,d$.
	Hence, the quasi-Trefftz space $\QT^p(T)$ for the diffusion-advection-reaction equation is well-defined if 
	$\bk_{jm}\in C^{\max\{p-1,1\} } (T)$, $\bbeta_j \in C^{\max\{p-1,1\}}(T)$ and $\sigma\in C^{\max\{p-2,0\}}(T)$ for all $ j,m=1,\dots,d$.
	
	Using linearity and Leibniz formula \eqref{def:Leibniz}, we also derive the following expression of the partial derivatives of $\calL v$ which will be useful for the construction of the quasi-Trefftz basis functions:
	\begin{align}\label{derop}
		D^{\bm{i}}	\mathcal{L}v =&	\nonumber	-\sum_{j=1}^{d} \sum_{m=1}^{d}  D^{\bm{i}+\bm{e}_{j}} (\bm{K}_{jm}  D^{\bm{e}_{m}} v) +	
		\sum_{j=1}^{d} D^{\bi+\bm{e}_{j}}  (\bm{\beta}_{j}  v ) +D^{\bi} \sigma v\\
		=&
		-\sum_{j=1}^{d} \sum_{m=1}^{d} \sum_{\bm{\ell}\leq \bm{i}+\bm{e}_{j}} \binom{\bm{i}+\bm{e}_{j}}{\bm{\ell}} D^{\bm{\ell}} \bm{K}_{jm}  D^{\bm{i}+\bm{e}_{j}-\bm{\ell}+\bm{e}_{m}} v\\&+
		\sum_{j=1}^{d}\sum_{\bm{\ell}\leq \bm{i}+\bm{e}_{j}} \binom{\bm{i}+\bm{e}_{j}}{\bm{\ell}}  D^{\bm{\ell}} \bm{\beta}_{j}  D^{\bm{i}+\bm{e}_{j}-\bm{\ell}} v + \sum_{\bm{\ell}\leq \bm{i}} \binom{\bm{i}}{\bm{\ell}} D^{\bm{\ell}}\sigma D^{\bm{i}-\bm{\ell}} v.\nonumber
	\end{align}
	\subsection{\textit{h}-convergence of the quasi-Trefftz DG method}
	We consider only homogenous diffusion-advection-reaction equation, i.e, with no source term $f$, since the quasi-Trefftz space defined in \eqref{localdar} is suited for this kind of problems.
	Choosing the global quasi-Trefftz space of degree $p$ defined in \eqref{globalQT} as discrete space $V_h$ in the DG scheme \eqref{variational} with $f=0$, we obtain the quasi-Trefftz DG method.
	
	The next theorem establishes the convergence rate in $h$ for the error $u-u_h$ in the $\N{\cdot}_{dar}$-norm.
	From the quasi-optimality inequality \eqref{darnormconverge} we deduce the convergence rate for the quasi-Trefftz DG method, using the the error bound  \eqref{eq:Approx}.
	We recall that 
	$ V_{*} :=H^1_{\Gamma_D}(\Omega) \cap H^2(\calT_h)$.
	\begin{theorem}[Quasi-Trefftz DG convergence rate]\label{finaleerror}
		Let $ p \in \mathbb{N} $ and let $u\in V_{*}\cap C^{p+1}(\calT_h)$ solve \eqref{eq}-\eqref{Neumann} with 
		$\bk_{jm}\in C^{\max\{p-1,1\} } (\calT_h)$, $\bbeta_j \in C^{\max\{p-1,1\}}(\calT_h)$ and $\sigma\in C^{\max\{p-2,0\}}(\calT_h)$ for all $ j,m=1,\dots,d$ and $f=0$.
		Let $u_h $ solve \eqref{variational} with discrete space $V_h=\QT^p(\calT_h)$ and the penalty parameter $\gamma$ as in Lemma \ref{coerdiff}.
		Assume that each mesh element $T$  is star-shaped with respect to $\mathbf{x}_T$.  Then,  the following convergence rate holds
		\begin{equation}
			\begin{split}
				\vertiii{u-u_h}_{dar}\leq& \left(1+\frac{M}{\alpha}\right)
				\frac{d^{p}}{p!}
				\left[	\sum_{T\in \calT_h}	\abs{T}
				\biggl(\left(1+\eta\right)\N{\bk}_{L^{\infty}(T)} 
				\right. \\&\left.
				+	 \frac{d^{2}}{(p+1)^2}	\left(2	C_g \gamma \eta 
				+\N{\bbeta}_{L^{\infty}(T)} \eta h_T
				+h_T^2
				+\eta h_T\right)\biggr)
				h_T^{2p}\abs{u}^2_{C^{p+1}(T)}	\right]^{\frac{1}{2}},
			\end{split}
		\end{equation}	
		with $\alpha$ as in Lemma \ref{coercdiscreta} and $M$ as in Lemma \ref{bounded}.
	\end{theorem}
	\begin{proof}     
		Since we have the quasi-optimality inequality \eqref{darnormconverge}, we now focus on estimating the quantity
		$	\inf_{v_h\in \QT^p(\calT_h)}\vertiii{u-v_h}_{dar,*}$.
		We recall, for $v\in V_{*h}=V_{*}+\QT^p(\calT_h)$, the definition of the $\N{\cdot}_{dar,*}$-norm defined in Section \ref{normss}
		\begin{align*}
			\vertiii{v}^2_{dar,*}=&\sum_{T\in \calT_h}\int_{T}\bk\nabla v\cdot \nabla v+\sum_{e\in\calE_{h}^{I}\cup \calE_h^D}\frac{\gamma}{h_e}\int_{e} \jmp{v}^2+ \frac{1}{2}\sum_{e\in\calE_{h}}\int_{e}\abs{\bbeta\cdot\bn_e}\jmp{v}^2\\&+
			\N{v}^2_{L^{2}(\Omega)}+\sum_{T\in \calT_h}h_T\N{\bk^{\frac{1}{2}}\nabla v\cdot \bn_T}^2_{L^2(\partial T)}+	\sum_{T\in \calT_h}\N{ v}^2_{L^2(\partial T)}.
		\end{align*}
		We recall the assumption \eqref{gradedmesh}: the mesh $\calT_h$ is graded, i.e., there exists a constant $C_{g}>0$   such that, for all $ T\in\calT_h $ and all $e\in\calF_T$, 
		$$
		C_{g}h_e\ge h_T.
		$$
		Observe that, from the Young's inequality, $\jmp{v}^2=(v_1-v_2)^2\leq 2(v_{1}^2+v_2^2)$ on an internal facet $e=\partial T_1\cap \partial T_2$
		and $\jmp{v}^2=v^2$ on a boundary facet $e$.
		Using these two facts in the second and third terms of the norm, we rearrange the sums over parts of the mesh skeleton as sums over elements and obtain the bound:
		\begin{align*}
			\vertiii{v}^2_{dar,*}\leq&
			\sum_{T\in \calT_h}	\biggl(\N{\bk}_{L^{\infty}(T)} \N{\nabla v}^2_{L^2(T)} +2 C_g\frac{\gamma}{h_T} \N{v}^2_{L^2(\partial T)}
			+\N{\bbeta}_{L^{\infty}(T)}\N{v}^2_{L^2(\partial T)}\\&
			+\N{v}^2_{L^2(T)}
			+\N{\bk}_{L^{\infty}(T)}h_T\N{\nabla v}^2_{L^2(\partial T)}
			+\N{v}^2_{L^2(\partial T)}\biggr).
		\end{align*}
		Next, we use the definition of the $\N{\cdot}_{C^{m}}$-norms and obtain
		\begin{align*}
			\vertiii{v}^2_{dar,*}\leq&
			\sum_{T\in \calT_h}	\biggl(\N{\bk}_{L^{\infty}(T)}\abs{T} \N{\nabla v}^2_{C^0(T)} +
			2C_g	\frac{\gamma}{h_T} \abs{\partial T}\N{v}^2_{C^0( T)}\\&
			+\N{\bbeta}_{L^{\infty}(T)} \abs{\partial T}\N{v}^2_{C^0(T)}
			+\abs{T}\N{v}^2_{C^0(T)}\\&
			+\abs{\partial T}\N{v}^2_{C^0(T)}
			+\N{\bk}_{L^{\infty}(T)}\abs{\partial T}h_T\N{\nabla v}^2_{C^0( T)}\biggr).
		\end{align*}
		Considering the quantity of our interest we have 
		\begin{align*}
			\inf_{v_h\in \QT^p(\calT_h)}\vertiii{u-v_h}^2_{dar,*}\leq&
			\sum_{T\in \calT_h}		\inf_{v_h\in \QT^p(T)}
			\biggl[\left(\abs{T}+\abs{\partial T}h_T\right)\N{\bk}_{L^{\infty}(T)} \N{\nabla (u-v_h)}^2_{C^0(T)}\\&
			+\left(2 C_g	\frac{\gamma}{h_T} \abs{\partial T}
			+\N{\bbeta}_{L^{\infty}(T)} \abs{\partial T}
			+\abs{T}
			+\abs{\partial T}\right)\N{u-v_h}^2_{C^0(T)}\biggr].
		\end{align*}
		Using the Quasi-Trefftz approximation properties \eqref{eq:Approx} we get 
		\begin{align*}
			\inf_{v_h\in \QT^p(\calT_h)}\vertiii{u-v_h}^2_{dar,*}\leq&
			\sum_{T\in \calT_h}	
			\left[\left(\abs{T}+\abs{\partial T}h_T\right)\N{\bk}_{L^{\infty}(T)} 
			\frac{d^{2p}}{(p!)^2} h_T^{2p}\abs{u}^2_{C^{p+1}(T)} \right. \\&\left.
			+\left(2 C_g	\frac{\gamma}{h_T} \abs{\partial T}
			+\N{\bbeta}_{L^{\infty}(T)} \abs{\partial T}
			+\abs{T}
			+\abs{\partial T}\right)\times \right. \\&\left.\frac{d^{2(p+1)}}{((p+1)!)^2} h_T^{2(p+1)}\abs{u}^2_{C^{p+1}(T)} \right].
		\end{align*}
		We now use the assumption \eqref{cuboidalmesh}: exists $\eta>0$ such that
		$
		h_T \abs{\partial T} \leq \eta \abs{T}$ for all  $T\in\calT_h
		$,
		and we infer that
		\begin{align*}
			\inf_{v_h\in \QT^p(\calT_h)}\vertiii{u-v_h}^2_{dar,*}\leq& 	\frac{d^{2p}}{(p!)^2}
			\sum_{T\in \calT_h}	
			\biggl[	\left(1+\eta\right)\abs{T}\N{\bk}_{L^{\infty}(T)} 
			+	 \frac{d^{2}}{(p+1)^2}\times  \\&
				\left(2	C_g \frac{\gamma}{h_T} \eta h_T^{-1}+\N{\bbeta}_{L^{\infty}(T)} \eta h_T^{-1}	+1 +\eta h_T^{-1}
		\right)\abs{T}
			h_T^{2}\biggr] h_T^{2p}\abs{u}^2_{C^{p+1}(T)}.
		\end{align*}
		Combining this bound  with the inequality \eqref{darnormconverge} yields the assertion.
	\end{proof}
	We remark that we use the $C^m$-norms in the convergence analysis of the quasi-Trefftz DG method, instead of the classical Sobolev norms $H^m$. We observe that in Theorem \ref{prop:Approx} and in Theorem \ref{finaleerror} the solution $u$ needs to belong to $C^{p+1}(\calT_h)$, which is a stronger regularity assumption than the one required in Theorem \ref{finalerror}, i.e., $u\in H^{p+1}(\calT_h)$.
	For the Trefftz space the analysis has been extended to the case of less regular solutions using the fact that “averaged Taylor polynomials” \cite[p. 421]{moiola2018space} of solutions are Trefftz functions. However, we cannot use this argument since, in general, they are not quasi-Trefftz functions and to our knowledge a study using Sobolev norms for the quasi-Trefftz space is still missing \cite[Remark 4.7]{imbert2023space}.
	Apart from this difference, we have optimal convergence rates in the $\N{\cdot}_{dar}$-norm.
	\section{Basis functions}\label{s:BasisT}
	In this section, we define, for any $p\in\mathbb{N} $ and $T\in\calT_h$, a family of bases for the quasi-Trefftz space $\QT^p(T)$ for the homogeneous diffusion-advection-reaction equation and derive an explicit formula to compute them.
	
	Since each function $v\in\QT^p(T)$ is a degree-$p$ polynomial, $v$ can be expressed as a linear combination of scaled monomials centred in $\bx_T\in T$:
	\begin{equation}\label{lincombmon}
		v(\bx)=\sum_{\mk\in\IN^{d},|\mk|\leq p} a_\mk \left(\frac{\mathbf{x}-\mathbf{x}_{T}}{h_{T}}\right)^{\mk},
	\end{equation}
	where the coefficients $\{a_\mk\}_{\mk\in\IN^{d},|\mk|\le p}$  are such that 
	$$
	a_\mk=\frac{h_T^{\abs{\bm{k}}}}{\mk!}D^\mk v(\bx_T).
	$$
	This is obtained by considering the derivatives of $v$
	$$
	D^{\bi}v(\bx)=\sum_{\substack{\mk\in\IN^d, \abs{\mk}\leq p\\ \mk\ge \bi}} a_\mk \left(\frac{\mathbf{x}-\mathbf{x}_{T}}{h_{T}}\right)^{\mk-\bi} \frac{\mk !}{(\mk-\bi)!} \frac{1}{h_T^{\abs{\bi}}},
	$$
	and evaluating them at $\bx_T$.
	The construction of $v\in\QT^p(T)$  as in \eqref{lincombmon} is equivalent to finding the set of coefficients $\{a_\mk\}_{\mk\in\IN^{d},|\mk|\le p}$ such that  $v$ satisfies $	D^{\bm{i}}	\mathcal{L}v(\bx_T)=0$ for all $\abs{\bi}\leq p-2$.
	Evaluating \eqref{derop} in $\bx_T$ and using the relation $D^{\bm{k}} v(\bx_T)=\frac{a_{\bm{k}}\bm{k}!}{h_T^{\abs{\bm{k}}}}$, we obtain:
	\begin{align*}
		D^{\bm{i}}	\mathcal{L} v(\bx_T)=&
		-\sum_{j=1}^{d} \sum_{m=1}^{d} \sum_{\bm{\ell}\leq \bm{i}+\bm{e}_{j}} \binom{\bm{i}+\bm{e}_{j}}{\bm{\ell}} D^{\bm{\ell}} \bm{K}_{jm} (\bx_T)
		\frac{a_{\bm{i}+\bm{e}_{j}-\bm{\ell}+\bm{e}_{m}}(\bm{i}+\bm{e}_{j}-\bm{\ell}+\bm{e}_{m})!}{h_T^{\abs{\bi-\bm{\ell}}+2}}\\
		&+\sum_{j=1}^{d}\sum_{\bm{\ell}\leq \bm{i}+\bm{e}_{j}} \binom{\bm{i}+\bm{e}_{j}}{\bm{\ell}}  D^{\bm{\ell}} \bm{\beta}_{j} (\bx_T)
		\frac{a_{\bm{i}+\bm{e}_{j}-\bm{\ell}}(\bm{i}+\bm{e}_{j}-\bm{\ell})!}{h_T^{\abs{\bi-\bm{\ell}}+1}} \\&
		+ \sum_{\bm{\ell}\leq \bm{i}} \binom{\bm{i}}{\bm{\ell}} D^{\bm{\ell}}\sigma (\bx_T) \frac{a_{\bm{i}-\bm{\ell}}(\bm{i}-\bm{\ell}	)!}{h_T^{\abs{\bi-\bm{\ell}}}}.
	\end{align*}
	Using the definition of the binomial coefficients, this can be rewritten as
	
	\begin{align}\label{coeff}
		D^{\bm{i}}	\mathcal{L} v (\bx_T)=&	\nonumber
		-\sum_{j=1}^{d} \sum_{m=1}^{d} \sum_{\bm{\ell}\leq \bm{i}+\bm{e}_{j}} \frac{(\bm{i}+\bm{e}_{j})!}{\bm{\ell}!} D^{\bm{\ell}} \bm{K}_{jm} (\bx_T)
		\frac{a_{\bm{i}+\bm{e}_{j}-\bm{\ell}+\bm{e}_{m}}(i_m+(\bm{e}_{j})_m-\ell_m+1)}{h_T^{\abs{\bi-\bm{\ell}}+2}}\\ 
		&+\sum_{j=1}^{d}\sum_{\bm{\ell}\leq \bm{i}+\bm{e}_{j}} \frac{(\bm{i}+\bm{e}_{j})!}{\bm{\ell}!}  D^{\bm{\ell}} \bm{\beta}_{j} (\bx_T)
		\frac{a_{\bm{i}+\bm{e}_{j}-\bm{\ell}}}{h_T^{\abs{\bi-\bm{\ell}}+1}}\\&
		+ \sum_{\bm{\ell}\leq \bm{i}} \frac{\bi !}{\bm{\ell}!} D^{\bm{\ell}}\sigma (\bx_T) \frac{a_{\bm{i}-\bm{\ell}}}{h_T^{\abs{\bi-\bm{\ell}}}}.\nonumber
	\end{align}
	We recall the assumption \eqref{assalg}:
	\begin{equation*}
		\bk_{11}(\bx)\ge k_{min}>0\quad  \text{for all} \quad \bx\in\Omega.
	\end{equation*}
	From the conditions $D^{\bm{i}}	\mathcal{L}v(\bx_T)=0$ for all $\abs{\bi}\leq p-2$, 
	we can compute the element with $\bm{\ell}=\bm{0}$, $j=1$ and $m=1$ in the first sum of \eqref{coeff}, dividing by its non-zero coefficient $\bk_{11}(\bx_T)>0$:
	\begin{align}\label{final}
		a_{\bi+2\be_1}=&\nonumber\frac{{h_T^{\abs{\bi}+2}}}{\bk_{11}(\bx_{T})(\bi+2\be_{1})!}\times\\&\nonumber
		\Biggl(
		-\sum_{\substack{j,m=1, \dots, d\\ \bm{\ell}\leq \bm{i}+\bm{e}_{j}\\(j,m,\bm{\ell})\neq (1,1,\bm{0})}}
		\frac{(\bm{i}+\bm{e}_{j})!}{\bm{\ell}!} D^{\bm{\ell}} \bm{K}_{jm} (\bx_T)
		\frac{a_{\bm{i}+\bm{e}_{j}-\bm{\ell}+\bm{e}_{m}}(i_m+(\bm{e}_{j})_m-\ell_m+1)}{h_T^{\abs{\bi-\bm{\ell}}+2}}\\
		&+\sum_{\substack{j=1, \dots, d\\ \bm{\ell}\leq \bm{i}+\bm{e}_{j}}}
		\frac{(\bm{i}+\bm{e}_{j})!}{\bm{\ell}!}  D^{\bm{\ell}} \bm{\beta}_{j} (\bx_T)
		\frac{a_{\bm{i}+\bm{e}_{j}-\bm{\ell}}}{h_T^{\abs{\bi-\bm{\ell}}+1}}\\&
		+ \sum_{\bm{\ell}\leq \bm{i}} \frac{\bi !}{\bm{\ell}!} D^{\bm{\ell}}\sigma (\bx_T) \frac{a_{\bm{i}-\bm{\ell}}}{h_T^{\abs{\bi-\bm{\ell}}}}\Biggr).\nonumber
	\end{align}
	We have found this recurrence relation \eqref{final} between the coefficients $a_{\mk}$, which will be used to compute the 
	coefficients iteratively.
	We will describe an algorithm in Section \ref{s:Algorithm}  that at every steps relies only on values already computed and is initialized by the coefficients $a_{\mk}$ with $k_1\in\{0,1\}$.
	It is clear from \eqref{final} that the recursive formula is based only on the partial derivatives at $\bx_T$ of the PDE coefficients $\bk$, $\bbeta$ and $\sigma$.
	We note that the equations \eqref{final} could also be written as a square linear system with $a_{\mk}$ with $k_1\neq 0,1$ as unknowns and $a_{\mk}$ with $k_1 \in \{0,1\}$ as right-hand side. However, we will not follow this path but a recursive implementation.
	
	We use the relation \eqref{final} to construct a basis for the local quasi-Trefftz space.
	We recall the dimension of the space of polynomials of maximum degree $p$ in $d$ variables defined in \ref{polytot}: $$S_{d,p}=\text{dim}(\mathbb{P}^p(\mathbb{R}^d))=\binom{p+d}{d}.$$
	To initialize the recursion, we choose two $2$ polynomial bases:
	\begin{equation*}
		\left\{\widehat f_s\right\}_{s=1,\dots,S_{d-1,p}}
		\text{ basis for }\mathbb P^p(\mathbb R^{d-1}) \quad \text{ and } \quad
		\left\{\widetilde f_s\right\}_{s=1,\dots,S_{d-1,p-1}}
		\text{ basis for }\mathbb P^{p-1}(\mathbb R^{d-1}).
	\end{equation*}
	Their total cardinality is 
	
	\begin{equation*}
		N_{d,p}:=S_{d-1,p}+S_{d-1,p-1}=\binom{p+d-1}{d-1}+\binom{p+d-2}{d-1}=\frac{(p+d-2)!(2p+d-1)}{(d-1)!p!}.
	\end{equation*}
	We define the following set of $N_{d,p}$ elements of $\QT^p(T)$:
	\begin{align*}
		&	\mathcal{B}_T^p:=	\left\{
		\begin{array}{ll}
			\refb\in\QT^p(T) :\\
			\refb((\bx_T)_1,\cdot)=\widehat f_J \;\tand \;\partial_{x_1} \refb((\bx_T)_1,\cdot)=0 \tfor J\leq S_{d-1,p},
			\\ 
			\refb((\bx_T)_1,\cdot)=0\; \tand \;\partial_{x_1} \refb((\bx_T)_1,\cdot)=\widetilde f_{J- S_{d-1,p}}
			\tfor J\ge S_{d-1,p}+1.
		\end{array}
		\right\}_{J=1,\dots,N_{p,d},}
	\end{align*}
	where $\refb((\bx_T)_1,\cdot)$ denotes the restriction of $\refb$ to the hyperplane $\{x_1=(\bx_T)_1\}$.
	Equivalently,
	the elements $b_J$ are degree-$p$ polynomials that satisfy the following system of equations
	\begin{align}\label{eq:systembasis}
		\begin{cases}
			D^\mi  \calL \refb(\bx_T)=0 
			&\quad  J=1,\ldots,N_{d,p},\; \mi \in\IN^{d},\; |\mi|\le p-m,\\
			\refb((\bx_T)_1,\cdot)=\widehat f_J \tand  \partial_{x_1} \refb((\bx_T)_1,\cdot)=0
			&\quad J=1,\ldots S_{d-1,p},\\ 
			\refb((\bx_T)_1,\cdot)=0 \tand  \partial_{x_1} \refb((\bx_T)_1,\cdot)=\widetilde f_{J- S_{d-1,p}}
			&\quad J= S_{d-1,p} +1,\ldots,N_{d,p}.
		\end{cases}
	\end{align}
	We  observe that the second and third set of conditions in \eqref{eq:systembasis} determine the coefficients $a_{\mk}$ with $k_{1} \in \{0,1\}$  of the monomial expansion \eqref{lincombmon} of an element $\refb$. 
	In the next proposition we show that the elements $b_J$ are well-defined and that they constitute a basis of $\QT^p(T)$.
	\begin{prop}\label{prop:Basis}
		For any $p \in \mathbb{N} $ and $T \in \calT_h $, the set $\mathcal{B}^p_T$ constitutes a basis for the space $\QT^p(T)$.
	\end{prop}
	\begin{proof}
		We need to verify that the elements $\refb$ are well defined and that  $\mathcal{B}_T^p$ is a set of linearly independent generators.
		The algorithms described in Section \ref{s:Algorithm} show that  the coefficients $a_{\mk}$ of an element $\refb$ are uniquely determined once the coefficients $a_{\mk}$ with $k_1\in\{0,1\}$ are assigned.
		This implies that the $N_{d,p}$ polynomials $\refb$ are well and uniquely defined, since the second and third set of conditions in \eqref{eq:systembasis} fix the coefficients $a_{0,k_2,\dots,k_d}$ and  $a_{1,k_2,\dots,k_d}$.
		We also notice that, since \eqref{final} holds for any $v\in \QT^p(T)$, we can apply the algorithms described in Section \ref{s:Algorithm} to construct $v$.
		From this we deduce that each $v\in\QT^p(T)$ is uniquely determined by the coefficients $a_{0,k_2,\dots,k_d}$ and $a_{1,k_2,\dots,k_d}$ and hence by its restriction $v((\bx_T)_1, \cdot)$ and the restriction of its derivative $\partial_{x_1}v((\bx_T)_1, \cdot)$.
		
		For any $v\in\QT^p(T)$, since its restriction $ v\left((\bx_T)_1,\cdot\right)$ is a polynomial of degree $p$ and the restriction of its derivative $ \partial_{x_1} v((\bx_T)_1,\cdot)$ is a polynomial of degree $p-1$, there exist some coefficients $\{\lambda_J\}_{J=1}^{N_{p,d}}\subseteq \mathbb{R}$ such that
		\begin{align*}
			v(\left(\bx_T)_1,\cdot\right)&=\sum_{J=1}^{S_{d-1,p}}\lambda_J \widehat f_J=\sum_{J=1}^{S_{d-1,p}}\lambda_J \refb((\bx_T)_1,\cdot),\\
			\partial_{x_1} v((\bx_T)_1,\cdot)&=\sum_{J=S_{d-1,p}+1}^{N_{d,p}}\lambda_J \widetilde f_{J- S_{d-1,p}}=\sum_{J=S_{d-1,p}+1}^{N_{d,p}}\lambda_J \partial_{x_1}\refb((\bx_T)_1,\cdot).
		\end{align*} 
		Hence, $v=\sum_{J=1}^{N_{d,p}}\lambda_J\refb$ is a linear combination of $\refb$.
		
		The fact that the  polynomials $\{\refb \}_{J=1,\dots,N_{d,p}}$ are linearly independent is due to their restrictions to $x_1=(\bx_T)_1$. Assume that $\sum_{J=1}^{N_{d,p}}c_J\refb=0$ for some coefficients $\{c_J\}_{J=1}^{N_{d,p}}\subseteq \mathbb{R}$.
		Then, restricting to $x_1=(\bx_T)_1$, we obtain
		\begin{align*}\sum_{J=1}^{S_{d-1,p}}c_J \refb((\bx_T)_1,\cdot)=&\sum_{J=1}^{S_{d-1,p}}c_J\widehat f_J=0,\\ \sum_{J=S_{d-1,p}+1}^{N_{d,p}}c_J \partial_{x_1}\refb((\bx_T)_1,\cdot)=&\sum_{J=S_{d-1,p}+1}^{N_{d,p}}c_J\widetilde  f_{J- S_{d-1,p}}=0.
		\end{align*}
		This implies that $c_J=0$ for $J=1,\dots, N_{d,p}$, since  $\left\{\widehat f_s\right\}_{s=1,\dots,S_{d-1,p}}$ and $\left\{\widetilde f_s\right\}_{s=1,\dots,S_{d-1,p-1}}$ are linearly independent.
	\end{proof}
	A consequence of Proposition \ref{prop:Basis} is that the conditions in the definition of $\QT^p(T)$ are linearly independent:
	\begin{equation*}
		\dim\big(\IP^p_d(T)\big)-\#\{\mi \in\IN^{d}\mid |\mi |\leq p-2\}=\begin{pmatrix}p+d\\d\end{pmatrix}-\begin{pmatrix}p+d-2\\ d\end{pmatrix} 
		=N_{d,p}=\dim\big(\QT^p(T)\big).
	\end{equation*}
	In particular, we have
	\begin{equation*}
		\dim\big(\QT^p(T)\big)=N_{d,p}
		=\begin{cases}
			2 & d=1\\ 
			2p+1 & d=2\\ 
			(p+1)^2 & d=3
		\end{cases}
		\quad=\calO_{p\to\infty}(p^{d-1}).
	\end{equation*}
	Notice that, in the one dimensional case, if we vary the polynomial degree $p$, then the dimension of the quasi-Trefftz space remains the same but the space changes.
	
	By comparing the dimension of $\QT^p(T)$ with the dimension of $\mathbb{P}^p_d(T)$, we obtain 
	\begin{align*}
		\dim\big(\QT^p(T)\big)=&\frac{(p+d-2)!(2p+d-1)}{(d-1)!p!}\\
		=&\calO_{p\to\infty}(p^{d-1})\ll \dim(\IP^p(T))=\binom{p+d}{d}=\calO_{p\to\infty}(p^{d}).
	\end{align*}
	We highlight that for large polynomial degrees $p$, the dimension of the quasi-Trefftz space is much smaller than the dimension of the full polynomial space of the same degree.
	
	We have shown in Theorem \ref{prop:Approx} that smooth solutions of the diffusion-advection-reaction equation with variable coefficients are approximated in $\QT^p(T)$ and in $\IP_d^p(T)$ with the same convergence rates with respect to the meshsize $h$.
	This is the main advantage of Trefftz and quasi-Trefftz schemes: offering the same level of accuracy for much fewer degrees of freedom.
	
	\section{The construction of the basis functions}\label{s:Algorithm}
	The aim of this section is the construction of  the basis functions of $\QT^p(T)$. 
	We describe in detail the iterative algorithm for the case $ d=1$, $d=2$ and for the general $d$-dimensional case.
	Every step of the iterations is based on the recursive formula \eqref{final}, hence depends on the Taylor coefficients at $\bx_T$ of the PDE coefficients $\bk$, $\bbeta$ and $\sigma$.
	We also stress that each step uses values already computed.
	
	\subsection{Algorithm for the case  \texorpdfstring{$d=1$}{d=1}}\label{sub1}
	Consider the case $d=1$.
	As noticed before, in this particular case the dimension of the quasi-Trefftz space $ \QT^p(T)$ is equal to $N_{1,p}=2$, which does not depend on $p$, even though the space does.
	We recall that, in the one dimensional case, the space of the analytical solutions of a homogeneous linear differential equation of the second-order with continuous coefficients is a two-dimensional space; this is reflected by the quasi-Trefftz space.
	Since we are in the scalar case, we denote
	by $\text{x}$ the only variable, by $\text{x}_T$ the point $\bx_T\in T$ and by $i$ the multi-indices $\bi$.
	The two quasi-Trefftz basis functions $b_1$ and $b_2$ are degree-$p$ polynomials in one variable such that satisfy 
	\begin{align*}
		b_1(\text{x}_T)=&\widehat f \tand  \partial_{x} b_1(\text{x}_T)=0,\\
		b_2(\text{x}_T)=&0 \tand  \partial_{x} b_2(\text{x}_T)=\widetilde f,
	\end{align*}
	where $
	\widehat f$ and $ \widetilde f$ are  two non zero constants that have been chosen.
	For $J=1,2$ we have
	$$
	b_J(\text{x})=\sum_{k\in\IN,k\leq p} a_k \left(\frac{\text{x}-\text{x}_{T}}{h_{T}}\right)^{k},
	$$
	and we want to compute the coefficients $a_k$.
	The coefficients $a_{0}$ and $a_{1}$ are determined by the choice of $\widehat f$ and $\widetilde f$.
	Formula \eqref{final} allows to compute $a_{i+2}$ from the coefficients $a_{j}$ with $j\leq i +1$.
	We can perform a single loop to compute  $a_{i+2}$ for increasing values of $i$.
	This procedure is described in Algorithm \ref{algo:Basis1D}.
	
	\vspace{0.3cm}
	\begin{algorithm}[H]
		{\sc Algorithm}\\
		\SetAlgoLined
		Data: $p$, $\text{x}_T$, $h_T$, $D^{\ell}\bk(\text{x}_T)$, $D^{\ell}\bbeta(\text{x}_T)$ for $\ell \leq p-1$, $D^{\ell}\sigma(\text{x}_T)$ for $\ell\leq p-2$.\\
		Fix coefficients $a_{0}$, $a_{1}$, choosing polynomial bases $\widehat f$, $\widetilde f$.\\
		For each $J=1,2$, we construct $\refb$ as follows:\\
		\For{$ i =0$ to $p-2$}{ 
			\begin{equation*}
				\begin{split}
					a_{i+2}=&\frac{{h_T^{i+2}}}{\bk(\text{x}_{T})(i+2)!}
					\left(
					-\sum_{ 0<\ell \leq i+1}
					\frac{(i+1)!}{\ell !} D^{\ell} \bm{K}(\text{x}_T)
					\frac{a_{i+2-\ell}(i+2-\ell)}{h_T^{i+2-\ell}}\right.\\
					&\left.+\sum_{0\leq\ell \leq i+1}
					\frac{(i+1)!}{\ell !}  D^{\ell} \bm{\beta} (\text{x}_T)
					\frac{a_{i+1-\ell}}{h_T^{i+1-\ell}}\right.\\&\left.
					+ \sum_{0\leq\ell\leq i} \frac{i !}{\ell !} D^{\ell}\sigma (\text{x}_T) \frac{a_{i-\ell}}{h_T^{i-\ell}}\right),
				\end{split}
			\end{equation*}
		}
		$\displaystyle\refb(x)
		=\sum_{k\in\IN, k\leq p} a_{ k} \bigg(\frac{x-\text{x}_T}{h_T}\bigg)^{ k}.$
		\vspace{2mm}
		\caption{The algorithm for the construction of $\refb$ in the case $d=1$.}
		\label{algo:Basis1D}
	\end{algorithm}
	\vspace{0.3cm}
	
	In order to show that if we vary the polynomial degree $p$ the quasi-Trefftz space changes, even though the dimension remains the same, we consider the following example: a homogeneous diffusion-advection-reaction equation 
	in $T=(0,2)$ with coefficients
	\begin{equation}\label{example1Dbasis}
		\bk(x)=x^4+1, \qquad \bbeta(x)= \text{sin}(x)+2, \qquad \sigma(x)=1.
	\end{equation}
	In Figure \ref{basisfun1D} we plot the two quasi-Trefftz basis functions $b_1$ and $b_2$ for polynomial degree $p=1,2,3,4$ using Algorithm \ref{algo:Basis1D}.
	\begin{figure}[h!]
		\centering
		\begin{subfigure}[b]{0.45\textwidth}
			\centering
			\includegraphics[width=\textwidth]{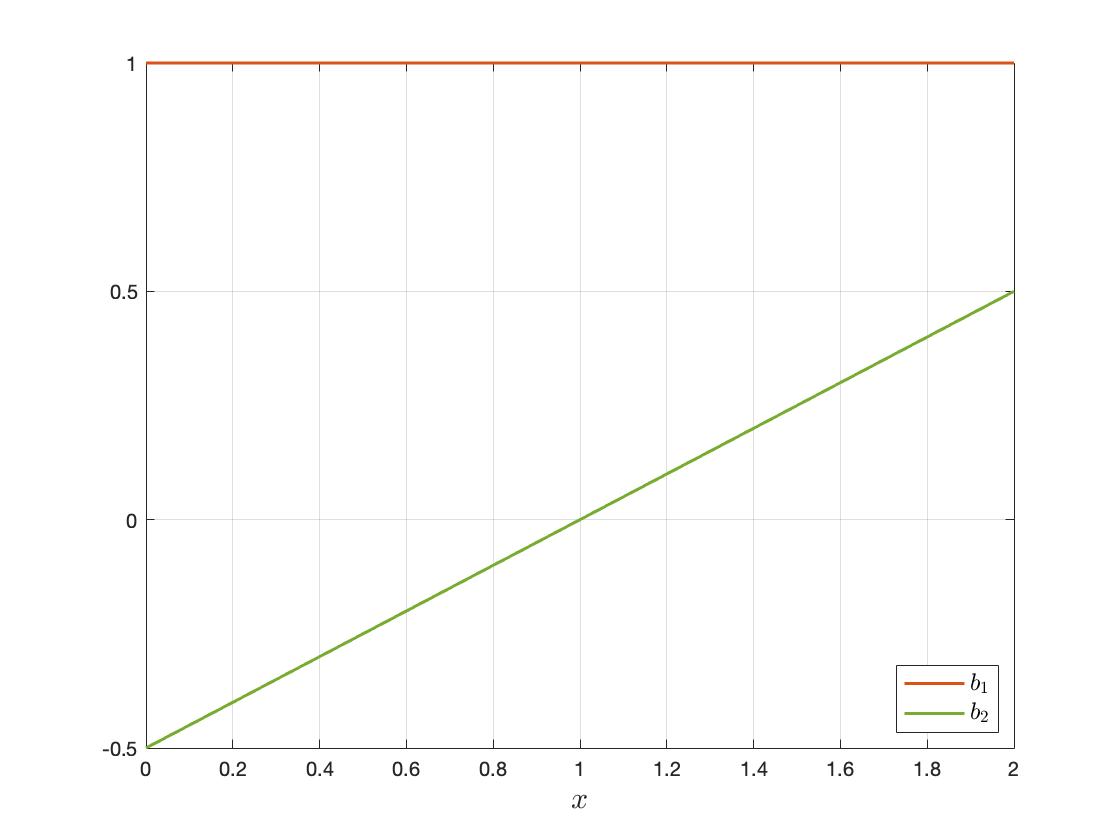}
			\caption{$p=1$}
		\end{subfigure}
		\begin{subfigure}[b]{0.45\textwidth}
			\centering
			\includegraphics[width=\textwidth]{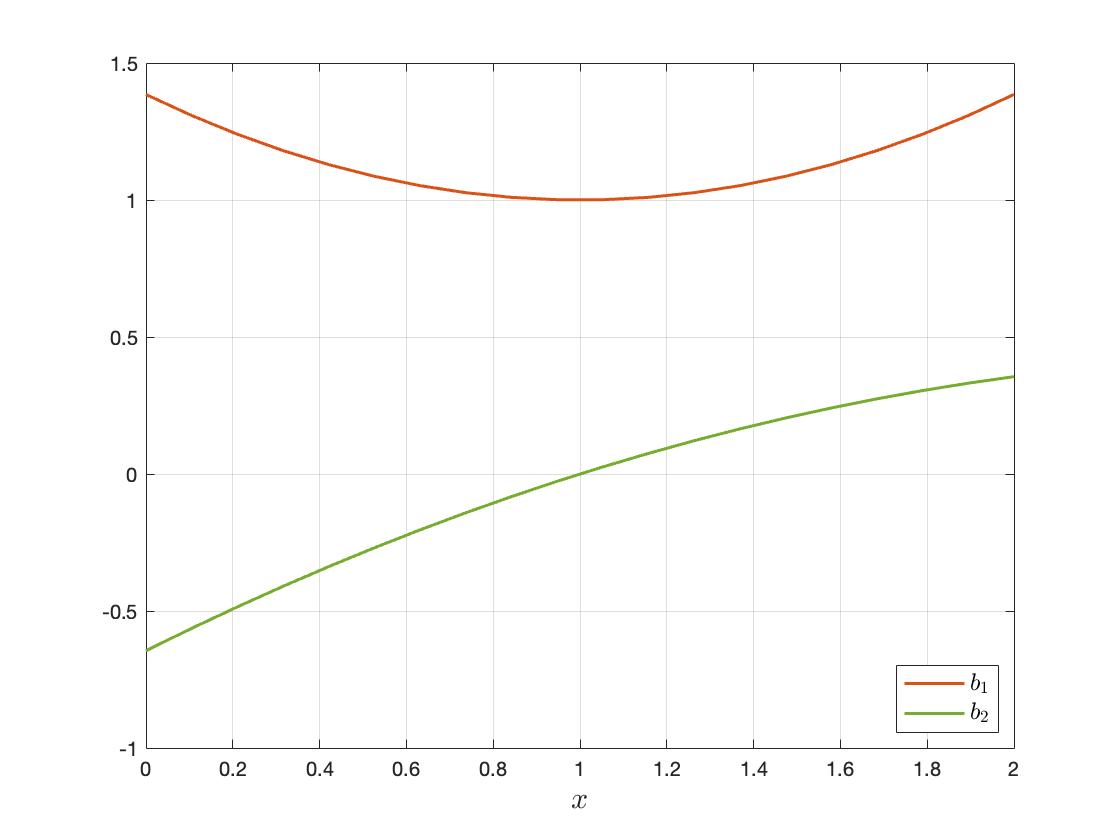}
			\caption{$p=2$}
		\end{subfigure}
		\vfill
		\begin{subfigure}[b]{0.45\textwidth}
			\centering
			\includegraphics[width=\textwidth]{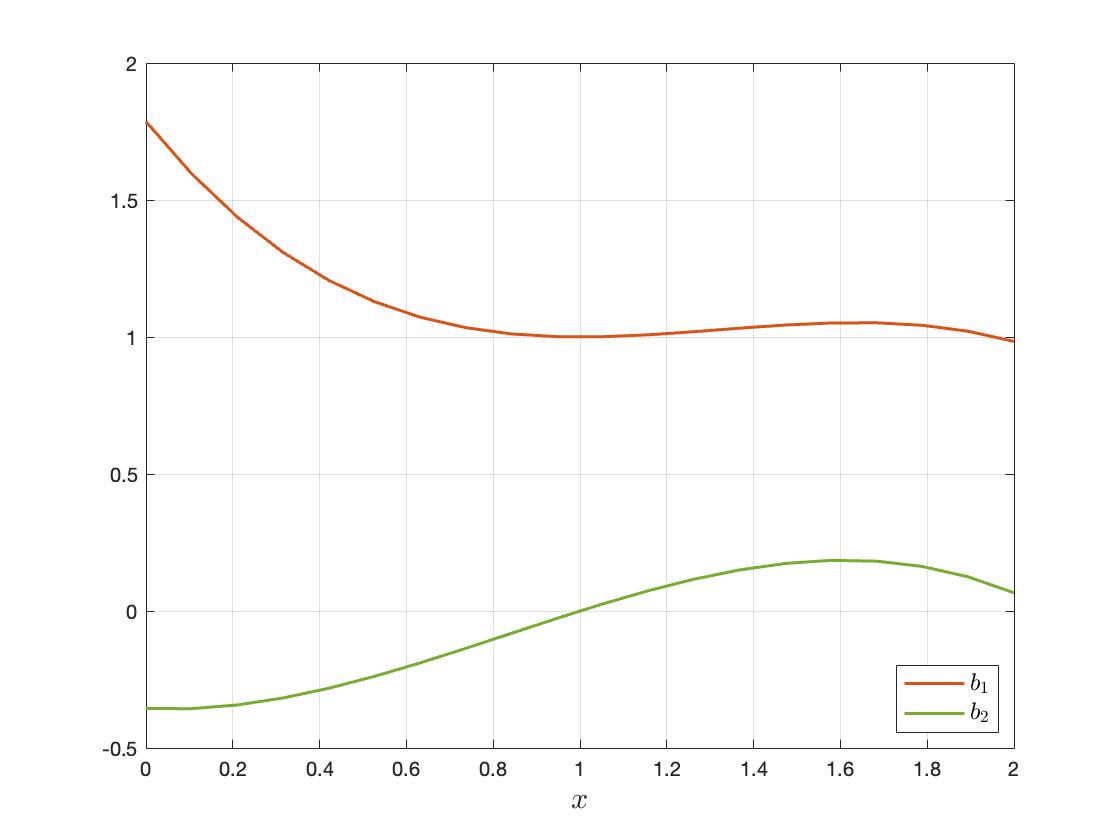}
			\caption{$p=3$}
		\end{subfigure}
		\begin{subfigure}[b]{0.45\textwidth}
			\centering
			\includegraphics[width=\textwidth]{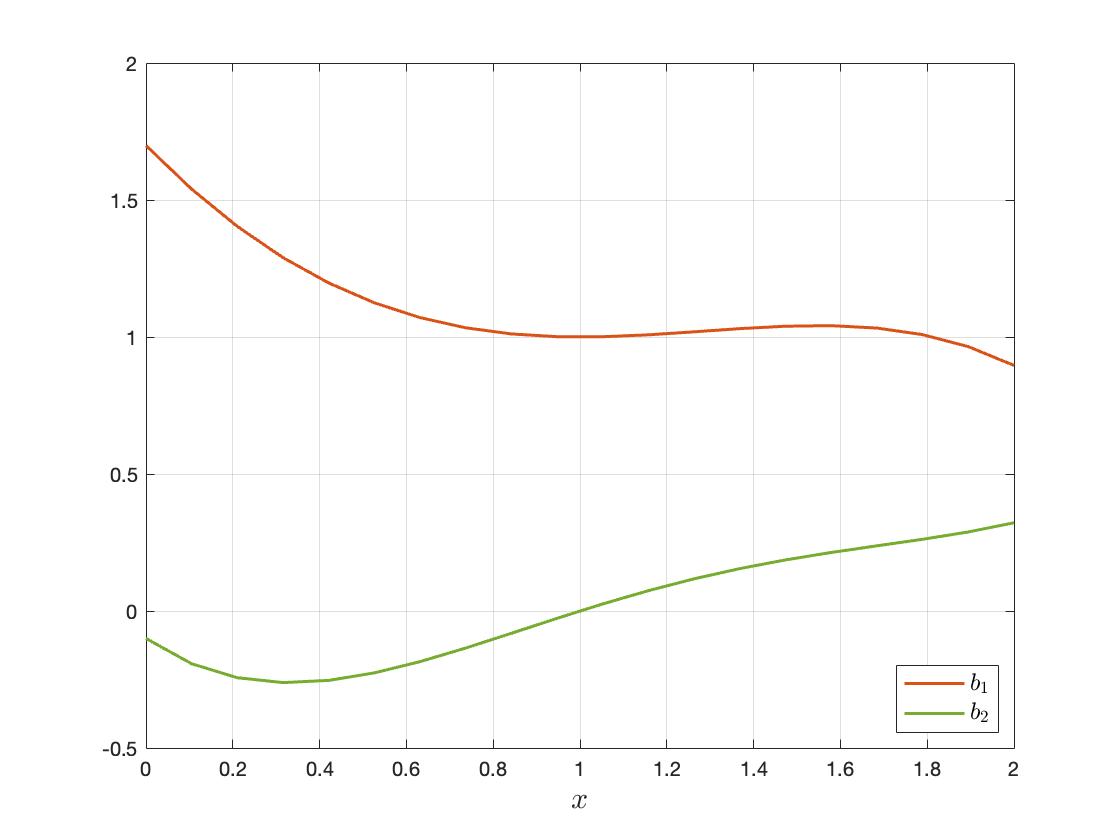}
			\caption{$p=4$}
		\end{subfigure}
		\caption{The two basis functions $b_1$ and $b_2$ of the Quasi-Trefftz space $\QT^p((0,2))$ for polynomial degree $p=1,2,3,4$ for the example \eqref{example1Dbasis}.
		}
		\label{basisfun1D}
	\end{figure}
	\subsection{Algorithm for the case \texorpdfstring{$d=2$}{d=2}}
	We now focus on the two dimensional case and describe the iterative algorithm to compute the  $N_{2,p}=2p+1$ quasi-Trefftz basis functions.
	For $J=1,\dots,2p+1$, we need to compute the coefficients of the monomial expansion of the basis function $\refb$, i.e,  $a_{ k _1, k _2}$ with $ k _1, k _2\in\IN$, $ k _1+ k _2\le p$.
	In Figure \ref{fig:IndexTriangle}, the coefficients $a_{k_1,k_2}$ are represented by the dots in the plane of indices $(k_1,k_2)\in\mathbb N^2$. Under the constraint that $k_1+ k_2\le p$, these dots form a triangular shape in the plane.
	To initialize the algorithm we have to choose two polynomial bases:
	$$
	\left\{\widehat f_s\right\}_{s=1,\dots,p+1}
	\text{ basis for }\mathbb P^p(\mathbb R) \quad \text{ and } \quad
	\left\{\widetilde f_s\right\}_{s=1,\dots,p}
	\text{ basis for }\mathbb P^{p-1}(\mathbb R),
	$$
	which we call “Cauchy data”. 
	This choice fixes the coefficients $a_{0, k _2}$ with  $ 0\leq k _2\le p$ and $a_{1, k _2}$ with  $0\leq  k _2\le p-1$, represented by the shaded yellow area in the figure.
	Once these are known, all the other coefficients are uniquely defined and can be computed using the relation \eqref{final}.
	From this recurrence relation we can compute $a_{ i_1+2, i _2}$ from the coefficients $a_{j_1, j_2}$ with $j_1\leq i _1+2$, $j_2\leq i_2+2$ and $j_1+j_2\leq i_1+i_2+2$, different from $a_{ i_1+2, i _2}$ itself.
	We have to apply \eqref{final} following a precise ordering of the multi-indices $(i_1,i_2)$, since at every step of the iterations we need to use values already computed.
	This suggests to proceed “diagonally”, computing the values $a_{ k _1, k _2}$ for $ k _1+ k _2=r$ increasingly from $r=2$ to $r=p$.
	On each of these diagonals we compute the values of $a_{ k _1,r- k _2}$ for increasing $k _1$.
	This is done performing a double loop: the external loop moves away from the origin across diagonals ($\nearrow$) and the inner loop moves from the $k _2$ axis to the $k _1$ axis along diagonals ($\searrow$). 
	The ordering of the coefficients $a_{ i_1+2, i _2}$ computed is shown by the arrows in Figure \ref{fig:IndexTriangle2}. 
	This procedure is described in Algorithm \ref{algo:Basis2D}.
	\vspace{-4cm}
	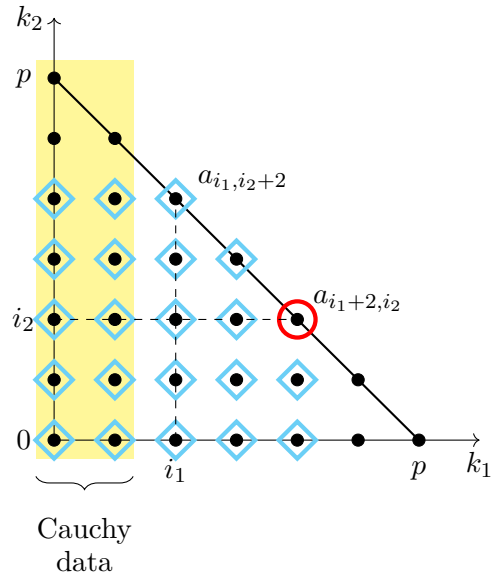
\begin{figure}[h!]
		\caption{A graphical representation of the coefficients  of a quasi-Trefftz basis function $\refb$ in the two dimensional case	(here $p=6$).
			The black dots $\bullet$ in the $(k_1,k_2)$ plane represent the coefficients $a_{k _1, k _2}$ of the function $\refb$.
			The coefficients  with $ k _1\in\{0,1\}$, which correspond to the nodes in the shaded yellow area, are determined from the “Cauchy data”, the second and third sets of equations in \eqref{eq:systembasis}, which we assign.
			The shape of the stencil formed by the nodes surrounded by the blue diamond $\MyDiamond[draw=cyan!50, ultra thick]$ represents the first equation in \eqref{eq:systembasis} for $(i_1,i_2)=(2,2)$ which, given $a_{j_1, j_2}$ with $j_1\leq i _1+2$, $j_2\leq i_2+2$ and $j_1+j_2\leq i_1+i_2+2$ (except for $a_{i_1+2,i_2}$), allows to compute the coefficient $a_{i_1+2,i_2}$ associated with the node surrounded by the red circle $\tikzcircle{}$.
			Each coefficient in the non-shaded region is computed  with formula \eqref{final} in this way, using its corresponding stencil of blue diamonds.}
		\label{fig:IndexTriangle}
		\begin{tikzpicture}[scale=.8]
			
			\fill[yellow!50!white](-.3,-.3)--(-.3,6.3)--(1.3,6.3)--(1.3,-0.3)--(-.3,-.3);
			
			\draw[->](0,0)--(7,0); 
			\draw[->](0,0)--(0,7); 
			\draw[thick](0,6)--(6,0);
			\draw(-.5,6)node{$p$};
			\draw(-.5,0)node{$0$};

			\node[diamond,draw, color=cyan!50,ultra thick] (d) at (0,0) {};
			\draw[fill](0,0)circle(.1);
			
			\node[diamond,draw, color=cyan!50,ultra thick] (d) at (1,0) {};
			\draw[fill](1,0)circle(.1);
			
			\node[diamond,draw, color=cyan!50,ultra thick] (d) at (2,0) {};
			\draw[fill](2,0)circle(.1);
			
			\node[diamond,draw, color=cyan!50,ultra thick] (d) at (3,0) {};
			\draw[fill](3,0)circle(.1);
			
			\node[diamond,draw, color=cyan!50,ultra thick] (d) at (4,0) {};
			\draw[fill](4,0)circle(.1);
			\draw[fill](5,0)circle(.1);
			\draw[fill](6,0)circle(.1);
			\node[diamond,draw, color=cyan!50,ultra thick] (d) at (0,1) {};
			
			\draw[fill](0,1)circle(.1);
			\node[diamond,draw, color=cyan!50,ultra thick] (d) at (1,1) {};
			
			\draw[fill](1,1)circle(.1);
			\node[diamond,draw, color=cyan!50,ultra thick] (d) at (2,1) {};
			
			\draw[fill](2,1)circle(.1);
			
			\node[diamond,draw, color=cyan!50,ultra thick] (d) at (3,1) {};
			\draw[fill](3,1)circle(.1);
			
			\node[diamond,draw, color=cyan!50,ultra thick] (d) at (4,1) {};
			\draw[fill](4,1)circle(.1);
			\draw[fill](5,1)circle(.1);
			
			\node[diamond,draw, color=cyan!50,ultra thick] (d) at (0,2) {};
			\draw[fill](0,2)circle(.1);
			
			\node[diamond,draw, color=cyan!50,ultra thick] (d) at (1,2) {};
			\draw[fill](1,2)circle(.1);
			
			\node[diamond,draw, color=cyan!50,ultra thick] (d) at (2,2) {};
			\draw[fill](2,2)circle(.1);
			\node[diamond,draw, color=cyan!50,ultra thick] (d) at (3,2) {};
			
			\draw[fill](3,2)circle(.1);
			
			\draw[fill](4,2)circle(.1);
			\draw[ultra thick,red](4,2)circle(.3);
			
			\draw(5,2.3)node{$a_{ i _1+2, i _2}$};
			
			\node[diamond,draw, color=cyan!50,ultra thick] (d) at (0,3) {};
			\draw[fill](0,3)circle(.1);
			\node[diamond,draw, color=cyan!50,ultra thick] (d) at (1,3) {};
			
			\draw[fill](1,3)circle(.1);
			
			\node[diamond,draw, color=cyan!50,ultra thick] (d) at (2,3) {};
			\draw[fill](2,3)circle(.1);
			\node[diamond,draw, color=cyan!50,ultra thick] (d) at (3,3) {};
			
			\draw[fill](3,3)circle(.1);
			\node[diamond,draw, color=cyan!50,ultra thick] (d) at (0,4) {};
			
			\draw[fill](0,4)circle(.1);
			\node[diamond,draw, color=cyan!50,ultra thick] (d) at (1,4) {};
			
			\draw[fill](1,4)circle(.1);
			\node[diamond,draw, color=cyan!50,ultra thick] (d) at (2,4) {};
			
			\draw[fill](2,4)circle(.1);
			\draw(3.1,4.3)node{$a_{ i _1, i _2+2}$};
			\draw[fill](0,5)circle(.1);
			\draw[fill](1,5)circle(.1);
			\draw[fill](0,6)circle(.1);
			
			\draw(7,-.4)node{$ k _1$};
			\draw(-.4,7)node{$ k _2$};
			\draw[decorate, decoration={brace,amplitude=5}] (1.3,-.6) -- (-0.3,-.6);
			\draw(0.5,.-1.5)node{Cauchy};
			\draw(0.5,.-2)node{data};
			\draw[dashed](2,0)--(2,4); \draw(2,-.5)node{$ i _1$};
			\draw[dashed](0,2)--(4,2); \draw(-.5,2)node{$ i _2$};
			\draw(6,-0.5)node{$p$};
		\end{tikzpicture}
	\end{figure}
	\pagebreak
	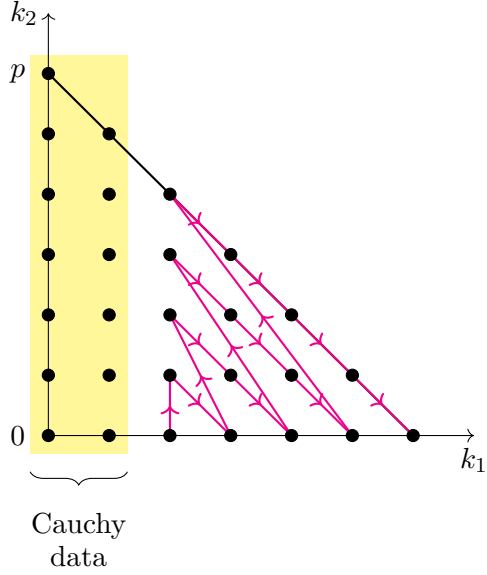
\begin{figure}[h!]
		\begin{tikzpicture}[scale=.8]
			\fill[yellow!50!white](-.3,-.3)--(-.3,6.3)--(1.3,6.3)--(1.3,-0.3)--(-.3,-.3);
			
			\draw[->](0,0)--(7,0); 
			\draw[->](0,0)--(0,7); 
			\draw[thick](0,6)--(6,0);
			\draw(-.5,6)node{$p$};
			\draw(-.5,0)node{$0$};
			
			\draw[middlearrow={>},thick,magenta] (2,0) to (2,1);
			\draw[middlearrow={>},thick,magenta] (2,1) to (3,0);
			\draw[middlearrow={>}, thick,magenta] (3,0) to (2,2);
			\draw[middlearrow={>}, thick,magenta] (2,2) to (3,1);
			\draw[middlearrow={>}, thick,magenta] (3,1) to (4,0);
			\draw[middlearrow={>},thick,magenta] (4,0) to (2,3);
			\draw[middlearrow={>},thick,magenta] (2,3) to (3,2);
			\draw[middlearrow={>},thick,magenta] (3,2) to (4,1);
			\draw[middlearrow={>},thick,magenta] (4,1) to (5,0);
			\draw[middlearrow={>},thick,magenta] (5,0) to (2,4);
			\draw[middlearrow={>},thick,magenta] (2,4) to (3,3);
			\draw[middlearrow={>},thick,magenta] (3,3) to (4,2);
			\draw[middlearrow={>},thick,magenta] (4,2) to (5,1);
			\draw[middlearrow={>},thick,magenta] (5,1) to (6,0);

			\draw[fill](0,0)circle(.1);
			\draw[fill](1,0)circle(.1);
			\draw[fill](2,0)circle(.1);
			\draw[fill](3,0)circle(.1);
			\draw[fill](4,0)circle(.1);
			\draw[fill](5,0)circle(.1);
			\draw[fill](6,0)circle(.1);
			\draw[fill](0,1)circle(.1);
			\draw[fill](1,1)circle(.1);
			\draw[fill](2,1)circle(.1);
			\draw[fill](3,1)circle(.1);
			\draw[fill](4,1)circle(.1);
			\draw[fill](5,1)circle(.1);
			\draw[fill](0,2)circle(.1);
			\draw[fill](1,2)circle(.1);
			\draw[fill](2,2)circle(.1);
			\draw[fill](3,2)circle(.1);
			\draw[fill](4,2)circle(.1);
			\draw[fill](0,3)circle(.1);
			\draw[fill](1,3)circle(.1);
			\draw[fill](2,3)circle(.1);
			\draw[fill](3,3)circle(.1);
			\draw[fill](0,4)circle(.1);
			\draw[fill](1,4)circle(.1);
			\draw[fill](2,4)circle(.1);
			\draw[fill](0,5)circle(.1);
			\draw[fill](1,5)circle(.1);
			\draw[fill](0,6)circle(.1);

			\draw[decorate, decoration={brace,amplitude=5}] (1.3,-.6) -- (-0.3,-.6);
			\draw(0.5,.-1.5)node{Cauchy};
			\draw(0.5,.-2)node{data};
			
			\draw(7,-.4)node{$ k _1$};
			\draw(-.4,7)node{$ k _2$};
			
		\end{tikzpicture}
		\caption{An illustration of the ordering of the coefficients of a quasi-Trefftz basis function $\refb$ computed through Algorithm \ref{algo:Basis2D}. 
			All the coefficients in the non-shaded region are computed with formula \eqref{final} 
			in a double loop: first across diagonals $\nearrow$, and then along diagonals $\searrow$. The ordering is shown by the arrows	$\SEARROW[->,thick,magenta]$.}
		\label{fig:IndexTriangle2}
	\end{figure}
	\begin{algorithm}[H]
		{\sc Algorithm}\\
		\SetAlgoLined
		Data: $p$, $\bx_T$, $h_T$, $D^{\bm\ell}\bk(\bx_T)$, $D^{\bm\ell}\bbeta(\bx_T)$ for $\abs{\bm\ell} \leq p-1$, $D^{\bm\ell}\sigma(\bx_T)$ for $\abs{\bm\ell} \leq p-2$.\\
		Fix coefficients $a_{0,k_2}$, $a_{1,k_2}$, choosing polynomial bases $\left\{\widehat f_J\right\}$, $\left\{\widetilde f_J\right\}$.\\
		For each $J=1,\dots,2p+1$, we construct $\refb$ as follows:\\
		\For{$r=2$ to $p$\quad (loop across diagonals $\nearrow$)\quad}{
			\For{$ i _1=0$ to $r-2$\quad (loop along diagonals $\searrow$)\quad}{
				set $ i _2=r- i _1-2$ and compute
				\begin{equation*}
					\begin{split}
						a_{\bi+2\be_1}=&\frac{{h_T^{\abs{\bi}+2}}}{\bk_{11}(\bx_{T})(\bi+2\be_{1})!}
						\Biggl(
						-\sum_{\substack{j,m\in\{1,2\}\\ \bm{\ell}\leq \bm{i}+\bm{e}_{j}\\(j,m,\bm{\ell})\neq (1,1,\bm{0})}}
						\frac{(\bm{i}+\bm{e}_{j})!}{\bm{\ell}!} \times \\&D^{\bm{\ell}} \bm{K}_{jm} (\bx_T)
						\frac{a_{\bm{i}+\bm{e}_{j}-\bm{\ell}+\bm{e}_{m}}(i_m+(\bm{e}_{j})_m-\ell_m+1)}{h_T^{\abs{\bi-\bm{\ell}}+2}}\\
						&+\sum_{\substack{j=1, 2\\ \bm{\ell}\leq \bm{i}+\bm{e}_{j}}}
						\frac{(\bm{i}+\bm{e}_{j})!}{\bm{\ell}!}  D^{\bm{\ell}} \bm{\beta}_{j} (\bx_T)
						\frac{a_{\bm{i}+\bm{e}_{j}-\bm{\ell}}}{h_T^{\abs{\bi-\bm{\ell}}+1}}\\&
						+ \sum_{\bm{\ell}\leq \bm{i}} \frac{\bi !}{\bm{\ell}!} D^{\bm{\ell}}\sigma (\bx_T) \frac{a_{\bm{i}-\bm{\ell}}}{h_T^{\abs{\bi-\bm{\ell}}}}\Biggr),
					\end{split}
				\end{equation*}
		}}
		$\displaystyle\refb(x_1,x_2)
		=\sum_{0\leq k _1+ k _2\leq p} a_{ k _1, k _2} \bigg(\frac{x_1-(\bx_{T})_1}{h_T}\bigg)^{ k _1} \bigg(\frac{x_2-(\bx_{T})_2}{h_T}\bigg)^{ k _2}.$
		\vspace{2mm}
		\caption{The algorithm for the construction of $\refb$ in the case $d=2$.}
		\label{algo:Basis2D}
	\end{algorithm}

	\subsection{Algorithm for the case \texorpdfstring{$d>2$}{d>2 }}
	We describe the extension of Algorithm \ref{algo:Basis2D} to the general $d$-dimensional case.
	The structure is the same, the only difference is that we need to perform three loops instead of two.
	In fact, for each value of $r=\abs{(i_2,\dots,i_d)}+i_1+2$ and of  $i_1$ we have not only one coefficient to compute but one for each $(i_2,\dots,i_d)\in\IN^{d-1}$ with $\abs{(i_2,\dots,i_d)}=r-i_1-2$.
	For this reason we need a further inner loop over $(i_2,\dots,i_d)$. We note that each coefficient of  this loop can be computed independently of the others and we preserve the requirement that at every step of the iterations all the values needed have already been computed.
	In Figure \ref{fig:IndexPyramid}  the coefficients $a_{\mk}$ are represented by the dots forming a simplex in the space of indices $\mk\in\mathbb N^d$ for the case $d=3$.
	The coefficients in the two yellow triangles are assigned by the “Cauchy data”. 
	The general coefficient, highlighted by the red diamond, is computed with the relation \eqref{final} using the coefficients corresponding to the black dots inside the blue stencil. 
	This procedure is described in Algorithm \ref{algo:BasisND}.
	\newcommand{\DDT}{\draw[dashed,thick]}
	\newcommand{\NodeThreeD}[3]{\draw (#1,#2,#3) node[circle,fill,inner sep=2pt] {};
		\DDT(#1,#2,0)--(#1,#2,#3);\DDT(0,#2,0)--(#1,#2,0);\DDT(#1,0,0)--(#1,#2,0);
		\DDT(#1,0,#3)--(#1,#2,#3);\DDT(0,0,#3)--(#1,0,#3);\DDT(#1,0,0)--(#1,0,#3);
		\DDT(0,#2,#3)--(#1,#2,#3);\DDT(0,#2,0)--(0,#2,#3);\DDT(0,0,#3)--(0,#2,#3);}
	
	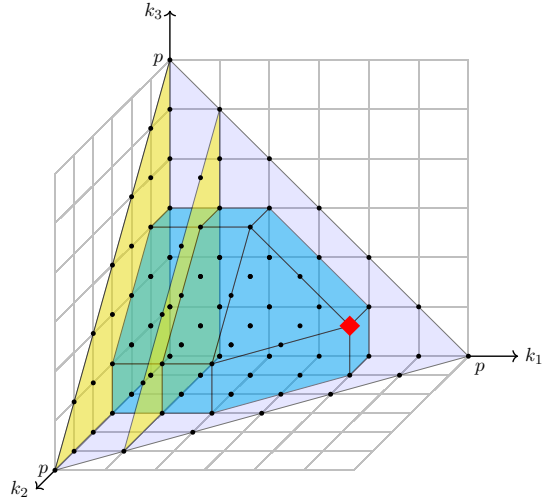
\begin{figure}[h!]\centering
		\resizebox{.49\linewidth}{!}{
			\begin{tikzpicture}
				[grid/.style={very thin,gray},axis/.style={->,thick}]
				%draw the axes
				
				\draw[axis] (0,0,0) -- (7,0,0) node[anchor=west]{$k_1$};
				\draw[axis] (0,0,0) -- (0,7,0) node[anchor=east]{$k_3$};
				\draw[axis] (0,0,0) -- (0,0,7) node[anchor=east]{$k_2$};

				\foreach \x in {0,1,...,6}
				\foreach \y in {0,1,...,6}
				\foreach \z in {0,1,...,6}
				{
					\draw[grid,lightgray] (\x,0,0) -- (\x,6,0);
					\draw[grid,lightgray] (0,\y,0) -- (6,\y,0);
					\draw[grid,lightgray] (0,\y,0) -- (0,\y,6);
					\draw[grid,lightgray] (0,0,\z) -- (0,6,\z);
					\draw[grid,lightgray] (\x,0,0) -- (\x,0,6);
					\draw[grid,lightgray] (0,0,\z) -- (6,0,\z);
					
				}

				\draw[opacity=.5] (0,6,0)--(0,0,6)--(6,0,0)--(0,6,0);
				\draw[fill=blue,opacity=.1] (0,6,0)--(0,0,6)--(6,0,0)--(0,6,0);
				\draw[fill=yellow,opacity=0.5] (0,0,0)--(0,6,0)--(0,0,6)--(0,0,0);
				\draw[fill=cyan,opacity=.5] 
				(1,1,3)--(1,0,3)--(0,0,3)--(0,1,3)--(1,1,3);
				\draw[fill=cyan,opacity=.5] 
				(0,1,3)--(0,3,1)--(1,3,1)--(1,1,3)--(0,1,3);
				\draw[fill=cyan,opacity=.5] (0,3,0)--(1,3,0)--(1,3,1)--(0,3,1)--(0,3,0);
				
				\draw[fill=yellow,opacity=0.5] (1,0,0)--(1,0,5)--(1,5,0)--(1,0,0);

				\draw[fill=cyan,opacity=.5] (2,1,3)--(1,1,3)--(1,0,3)--(2,0,3)--(2,1,3);
				\draw[fill=cyan,opacity=.5] (2,1,3)--(4,1,1)--(4,0,1)--(2,0,3)--(2,1,3);
				\draw[fill=cyan,opacity=.5] (4,1,1)--(4,0,1)--(4,0,0)--(4,1,0)--(4,1,1);
				\draw[fill=cyan,opacity=.5]
				(2,1,3)--(2,3,1)--(4,1,1)--(2,1,3);
				\draw[fill=cyan,opacity=.5] (2,3,1)--(4,1,1)--(4,1,0)--(2,3,0)--(2,3,1);
				\draw[fill=cyan,opacity=.5] (1,3,0)--(2,3,0)--(2,3,1)--(1,3,1)--(1,3,0);
				\draw[fill=cyan,opacity=.5] 
				(2,1,3)--(1,1,3)--(1,3,1)--(2,3,1)--(2,1,3);

				\foreach \x in {0,1,2,3,4}
				\foreach \y in {0,1}
				\foreach \z in {0,1}
				{
					\draw[] (\x,\y,\z) node[circle,fill,inner sep=1pt] {};
				}
				\foreach \x in {0,1,2,3}
				\foreach \y in {0,1}
				\foreach \z in {2}
				{
					\draw[] (\x,\y,\z) node[circle,fill,inner sep=1pt] {};
				}
				\foreach \x in {0,1,2,3}
				\foreach \y in {2}
				\foreach \z in {0,1}
				{
					\draw[] (\x,\y,\z) node[circle,fill,inner sep=1pt] {};
				}
				\foreach \x in {0,1,2}
				\foreach \y in {0,1}
				\foreach \z in {3}
				{
					\draw[] (\x,\y,\z) node[circle,fill,inner sep=1pt] {};
				}
				\foreach \x in {0,1,2}
				\foreach \y in {0,1,2}
				\foreach \z in {0,1,2}
				{
					\draw[] (\x,\y,\z) node[circle,fill,inner sep=1pt] {};
				}
				
				\foreach \x in {0,1,2}
				\foreach \y in {3}
				\foreach \z in {0,1}
				{
					\draw[] (\x,\y,\z) node[circle,fill,inner sep=1pt] {};
				}

				\draw[] (5,0,0) node[circle,fill,inner sep=1pt] {};
				\draw[] (6,0,0) node[circle,fill,inner sep=1pt] {};
				\draw[] (5,1,0) node[circle,fill,inner sep=1pt] {};
				\draw[] (5,0,1) node[circle,fill,inner sep=1pt] {};
				\draw[] (4,0,2) node[circle,fill,inner sep=1pt] {};
				\draw[] (4,2,0) node[circle,fill,inner sep=1pt] {};
				\draw[] (3,0,3) node[circle,fill,inner sep=1pt] {};
				\draw[] (3,3,0) node[circle,fill,inner sep=1pt] {};
				\draw[] (2,0,4) node[circle,fill,inner sep=1pt] {};
				\draw[] (2,4,0) node[circle,fill,inner sep=1pt] {};
				\draw[] (1,0,5) node[circle,fill,inner sep=1pt] {};
				\draw[] (1,5,0) node[circle,fill,inner sep=1pt] {};
				\draw[] (1,4,1) node[circle,fill,inner sep=1pt] {};
				\draw[] (1,4,0) node[circle,fill,inner sep=1pt] {};
				\draw[] (1,3,2) node[circle,fill,inner sep=1pt] {};
				\draw[] (1,2,3) node[circle,fill,inner sep=1pt] {};
				\draw[] (1,1,4) node[circle,fill,inner sep=1pt] {};
				\draw[] (0,6,0) node[circle,fill,inner sep=1pt] {};
				\draw[] (0,0,6) node[circle,fill,inner sep=1pt] {};
				\draw[] (0,5,1) node[circle,fill,inner sep=1pt] {};
				\draw[] (0,1,5) node[circle,fill,inner sep=1pt] {};
				\draw[] (0,4,2) node[circle,fill,inner sep=1pt] {};
				\draw[] (0,2,4) node[circle,fill,inner sep=1pt] {};
				\draw[] (0,3,3) node[circle,fill,inner sep=1pt] {};
				\draw[] (0,0,5) node[circle,fill,inner sep=1pt] {};
				\draw[] (0,5,0) node[circle,fill,inner sep=1pt] {};
				\draw[] (0,4,0) node[circle,fill,inner sep=1pt] {};
				\draw[] (0,0,4) node[circle,fill,inner sep=1pt] {};
				\draw[] (0,4,1) node[circle,fill,inner sep=1pt] {};
				\draw[] (0,1,4) node[circle,fill,inner sep=1pt] {};
				\draw[] (0,3,2) node[circle,fill,inner sep=1pt] {};
				\draw[] (0,2,3) node[circle,fill,inner sep=1pt] {};

				\draw[] (4,1,1) node[color=red,diamond,fill,inner sep=3pt] {};
				
				\draw (0,0,6) node[left] {$p$};
				\draw (6,0,0) node[below right] {$p$};
				\draw (0,6,0) node[ left] {$p$};
			\end{tikzpicture}
		}
		\caption{
			A graphical representation of the coefficients  $a_{\mk}$ of a quasi-Trefftz basis function $\refb$ in the three dimensional case (here $p=6$).
			The black dots $\tikzdot{}$ in the $(k_1,k_2,k_3)$ space represent the coefficients $a_{k _1, k _2,k_3}$ of the function $\refb$.
			The two yellow triangles in the planes $k_1=0$ and $k_1=1$ contain the coefficients whose values are determined from the choice of $\{\widehat f_s\}_{s=1,\dots,S_{d-1,p}}$ and $\{\widetilde f_s\}_{s=1,\dots,S_{d-1,p-1}}$.
			The coefficient $a_{i_1+2,i_2,i_3}$, represented by the  red diamond $\red{\blacklozenge}$, is computed with formula \eqref{final} for $(i_1,i_2,i_3)=(2,1,1)$ using all the coefficients inside the blue stencil.
			The coefficients whose dots are outside the two yellow triangles are computed  with formula \eqref{final}  from the coefficients inside their corresponding blue stencil.
			Algorithm \ref{algo:BasisND} performs three loops: the first loop through triangles parallel to the violet one (corresponding to $r=6$), the second loop through triangles parallel to the yellow ones, and the third one along the segments determined by the intersection between the two triangles.}
		\label{fig:IndexPyramid}
	\end{figure}
	\vspace{2cm}
	\pagebreak
	\begin{algorithm}[H]
		{\sc Algorithm}\\
		\SetAlgoLined
		Data: $p$, $\bx_T$, $h_T$, $D^{\bm\ell}\bk(\bx_T)$, $D^{\bm\ell}\bbeta(\bx_T)$ for $\abs{\bm\ell} \leq p-1$, $D^{\bm\ell}\sigma(\bx_T)$ for $\abs{\bm\ell} \leq p-2$.\\
		Fix coefficients $a_{0,k_2,\dots,k_d}$, $a_{1,k_2,\dots,k_d}$,
		choosing polynomial bases $\left\{\widehat f_s\right\}$, $\left\{\widetilde f_s\right\}$.\\
		For each $J=1,\ldots,N_{d,p}$, we construct $\refb$ as follows:\\
		\For{$r=2$ to $p$\qquad (loop across $\{i_1+i_2+\dots+i_d=r-2\}$ hyperplanes $\nearrow$)}{
			\For{$ i _1=0$ to $r-2$\qquad (loop across constant-$i_1$ hyperplanes $\rightarrow$)}{
				\For{ $(i_2,\dots,i_d)$ with $\abs{(i_2,\dots,i_d)}=r- i _1-2$\qquad
				}{
					\begin{equation*}
						\begin{split}
							a_{\bi+2\be_1}=&\frac{{h_T^{\abs{\bi}+2}}}{\bk_{11}(\bx_{T})(\bi+2\be_{1})!}
							\Biggl(
							-\sum_{\substack{j,m=1,\dots,d\\ \bm{\ell}\leq \bm{i}+\bm{e}_{j}\\(j,m,\bm{\ell})\neq (1,1,\bm{0})}}
							\frac{(\bm{i}+\bm{e}_{j})!}{\bm{\ell}!} \times \\&D^{\bm{\ell}} \bm{K}_{jm} (\bx_T)
							\frac{a_{\bm{i}+\bm{e}_{j}-\bm{\ell}+\bm{e}_{m}}(i_m+(\bm{e}_{j})_m-\ell_m+1)}{h_T^{\abs{\bi-\bm{\ell}}+2}}\\
							&+\sum_{\substack{j=1, \dots,d\\ \bm{\ell}\leq \bm{i}+\bm{e}_{j}}}
							\frac{(\bm{i}+\bm{e}_{j})!}{\bm{\ell}!}  D^{\bm{\ell}} \bm{\beta}_{j} (\bx_T)
							\frac{a_{\bm{i}+\bm{e}_{j}-\bm{\ell}}}{h_T^{\abs{\bi-\bm{\ell}}+1}}\\&
							+ \sum_{\bm{\ell}\leq \bm{i}} \frac{\bi !}{\bm{\ell}!} D^{\bm{\ell}}\sigma (\bx_T) \frac{a_{\bm{i}-\bm{\ell}}}{h_T^{\abs{\bi-\bm{\ell}}}}\Biggr),
						\end{split}
					\end{equation*}
		}}}
		$\displaystyle\refb(\bx)=\sum_{\mk\in\IN^{d},|\mk|\leq p} a_\mk \left(\frac{\mathbf{x}-\mathbf{x}_{T}}{h_{T}}\right)^{\mk}$.
		\vspace{2mm}
		\caption{The algorithm for the construction of $\refb$ in the $d$-dimensional case.}
		\label{algo:BasisND}
	\end{algorithm}
	
	\chapter{Numerical experiments}\label{Chapter6}
	In this chapter we present some numerical tests in two dimensions. 
	We begin with a simple example with polynomial solution in Section \ref{polysol}. In Section \ref{h-conv} we consider two test cases with smooth solutions in order to validate the $h$-convergence of the quasi-Trefftz DG method and compare the three different formulations SIPG, IIPG and NIPG introduced in Section \ref{s:Diffvar} that result for a different choice of the symmetrization parameter $\epsilon$.
	Moreover, in Section \ref{p-conv} we compare the quasi-Trefftz space against the full polynomial one in terms of $p$-convergence and of number of degrees of freedom.
	Finally, in Section \ref{dom} we consider advection-dominated and reaction-dominated examples to show the behaviour of the quasi-Trefftz DG method also in these cases.
	
	The DG scheme \eqref{variational} with as discrete space the global quasi-Trefftz space \eqref{globalQT} for the diffusion-advection-reaction equation has been implemented in MATLAB.
	We remark that we consider the homogeneous diffusion-advection-reaction equation, i.e., with $f=0$. 
	In all the experiments we use the SIPG method, i.e., the symmetrization parameter $\epsilon=-1$, except in the example of  Section \ref{h-conv} where the three formulations are compared.
	
	All the experiments are performed on the unit square $ \Omega = (0,1)^2$, which is partitioned using conforming meshes of triangular elements generated by 
	DistMesh, see \cite{persson2004simple}, and satisfying the assumptions of Section \ref{sec:domain}.
	For all the experiments we will say, with an abuse of notation, the meshsize $h$ used, though this is just the input of the mesh generator and usually the meshsize as defined in \ref{meshsizEE} is a little larger than this input.
	We calculate the “correct” $h$ and we use this value in the plots of the $h$-convergence in Section \ref{polysol} and in Section \ref{h-conv}.
	
	We compute all the volume integrals in the assembly of the matrix and of the right-hand side using the following quadrature rule.
	First, we use the Gauss--Legendre quadrature rule on the interval $[0,1]$, then we consider the tensor product of the nodes in $[0,1]^2$ and then with the Duffy transformation \cite[p. 21]{antonietti2016high} we map the points into the reference triangle, whose vertices are $(0, 0)$, $(1, 0)$ and $(0, 1)$.
	Finally, the quadrature points are mapped onto the generic triangle $T\in\calT_h$ using the classical affine map, the reference element onto $T$.
	For computing the integrals over edges we simply use the Gauss--Legendre quadrature rule combined with this affine map.
	Since the Gauss--Legendre quadrature rule with $n$ points is exact for polynomials of degree less than or
	equal to $2n-1$, we use $p+1$ points in $[0,1]$, hence $(p+1)^2$ on each triangles, where $p$ is the polynomial degree used.
	The linear system is solved using the backslash command.
	
	To construct the quasi-Trefftz basis functions, we use the symbolic computations only once to calculate the derivative of the PDE coefficients that we need for Algorithm \ref{algo:Basis2D} and to evaluate them in the barycentre of each element. Their values on the barycentres are stored and  passed in input to the quasi-Trefftz algorithm.
	To initialize Algorithm \ref{algo:Basis2D}, we choose scaled monomials as polynomial bases:
	\begin{align*}
		\left\{\widehat f_s(x_2)=\bigg(\frac{x_2-(\bx_T)_2}{h_T}\bigg)^{s-1}\right\}_{s=1,\dots,p+1}
		&\text{ basis for }\mathbb P^p(\mathbb R),\\
		\left\{\widetilde f_s(x_2)=\frac{1}{h_T}\bigg(\frac{x_2-(\bx_T)_2}{h_T}\bigg)^{s-1}\right\}_{s=1,\dots,p}
		&\text{ basis for }\mathbb P^{p-1}(\mathbb R).
	\end{align*}
	
	\section{Polynomial solution}\label{polysol}
	The aim of this simple example is to test the correctness of the quasi-Trefftz DG method implemented in the particular case when the exact solution belongs to the discrete space.
	Consider the problem \eqref{eq}-\eqref{Neumann} with $f=0$ and with coefficients
	$$\bk(x_1,x_2)=Id, \qquad  \bbeta(x_1,x_2)=\bm{0}, \qquad \sigma(x_1,x_2)=\frac{4}{x_1^2+x_2^2+1},$$
	where $Id$ denotes the $2\times 2$ identity matrix.
	
	We impose only Dirichlet boundary condition $g_D$ on $\partial \Omega$ such that the exact solution is given by $u(x_1,x_2)=x_1^2+x_2^2+1$.
	Since $u$ is a polynomial of degree $2$ which satisfy the partial differential equation, we have that $u\in\QT^2(\calT_h)$ for  the definition of the quasi-Trefftz space with $p=2$.
	This follows also from Theorem \ref{prop:Approx} which states that the Taylor polynomial of degree $2$ of $u$, which is $u$ itself, belongs to $\QT^2(T)$ for all $T\in\calT_h$.
	
	We use $\QT^2(\calT_h)$ as discrete space and choose the penalty parameter $\gamma=32$ and the symmetrization parameter $\epsilon=-1$. In Figure  \ref{fig:polysol} we show the quasi-Trefftz DG error against the meshsize $h$.
	We consider the values $h=2^{-j}$ for $j=1,\dots,5$ and we  measure the error in the $L^2$-norm, in the $H^1$-norm and in the  $L^{\infty}$-norm. 
	We observe that the error is not null but reaches almost the machine precision.
	Even though from a theoretical point of view the error should be zero, this is what we expect because of round-off errors.
	\begin{figure}[h]
		\includegraphics[width=0.55\textwidth]{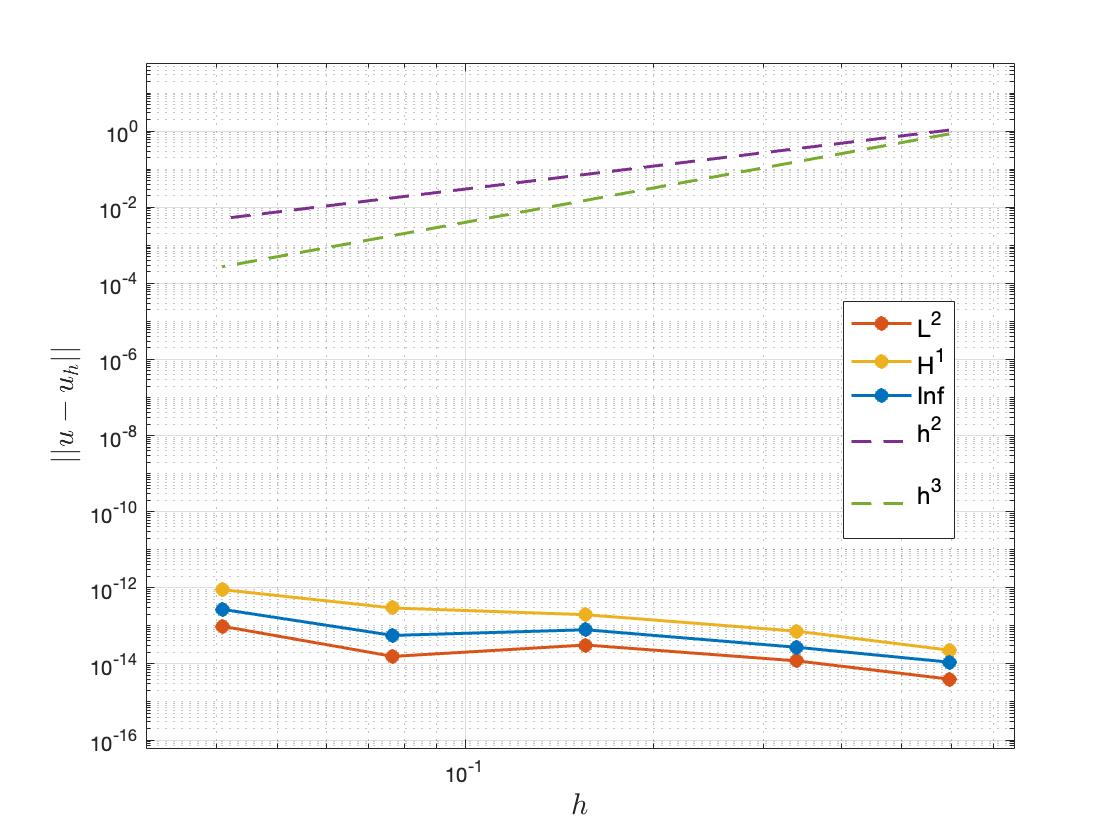}
		\caption{The quasi-Trefftz DG error for the problem of Section \ref{polysol} using $p=2$.
			The error in the $L^2$-norm (red line), in the $H^1$-norm (yellow line)  and in the $L^{\infty}$-norm (blue line) is plotted against the meshsize $h$ for $h=2^{-j}$ for $j=1,\dots,5$.
			Since the exact solution $u\in\QT^2(\calT_h)$ the error almost reaches the machine precision.}
		\label{fig:polysol}
	\end{figure}
	
	\section{\textit{h}-convergence}\label{h-conv}
	In this section we presents two test cases with smooth solutions in order to validate the optimal convergence rates with respect to the mesh size $h$ of the quasi-Trefftz DG method that we have proved in Theorem \ref{finaleerror}.
	
	As first example, we consider the problem \eqref{eq}-\eqref{Neumann} with $f=0$ and with coefficients
	\begin{equation}\label{example2}
		\bk(x_1,x_2)=\left[\begin{matrix}
			e^{x_1-x_2}& 0\\
			0 &e^{x_1-x_2}
		\end{matrix}\right], \qquad \bbeta(x_1,x_2)=\bm{0},  \qquad \sigma(x_1,x_2)=0.
	\end{equation}
	We assign only Dirichlet boundary condition $g_D$ on $\partial \Omega$ such that the exact solution is $u(x_1,x_2)=e^{-x_1+x_2}$.
	
	If the solution is smooth, such as in this case, the approximation bounds of Theorem \ref{finaleerror} for the $\N{\cdot}_{dar}$-norm of the quasi-Trefftz DG error hold.
	Here we compute the $L^2$-norm and the $H^1$-norm of the quasi-Trefftz DG error, which are both bounded above by the $\N{\cdot}_{dar}$-norm.
	In Figure \ref{fig:hconv}  we plot the error against the meshsize $h$ for polynomial degree $p=1,2,3,4$.
	
	\begin{figure}[h!]
		\centering
		\begin{subfigure}[b]{0.49\textwidth}
			\centering
			\includegraphics[width=\textwidth]{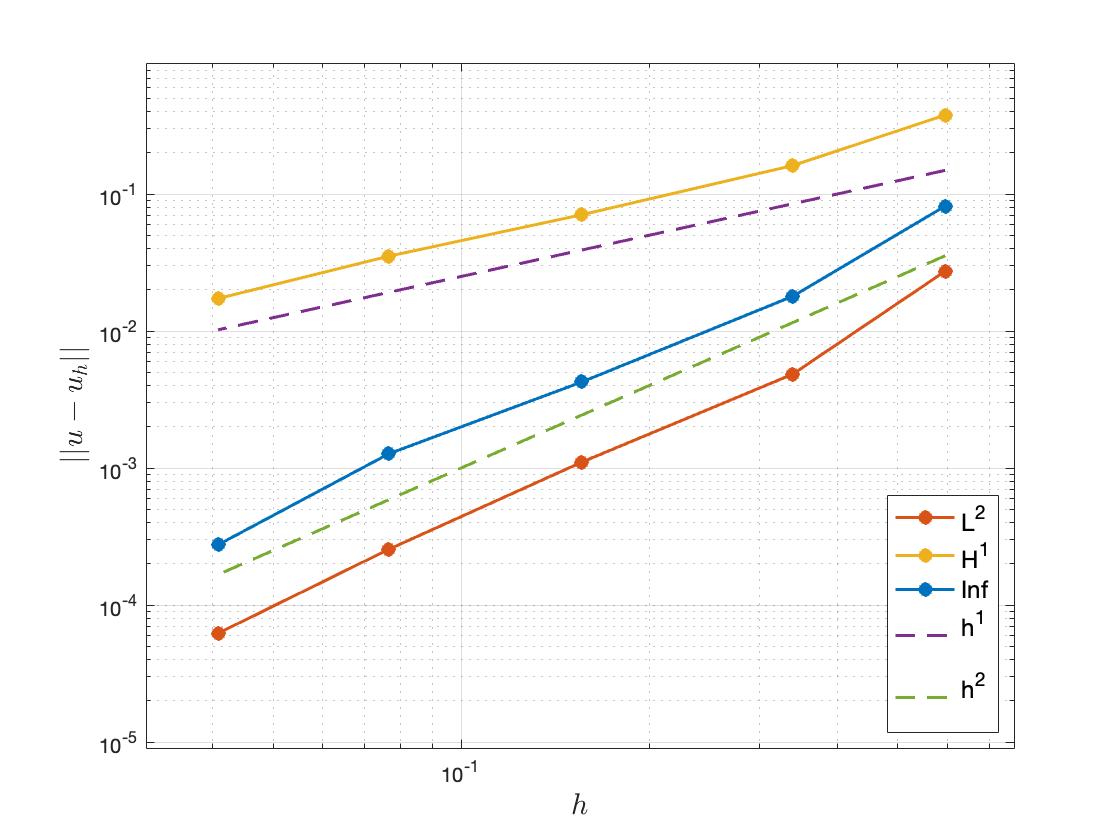}
			\caption{$p=1$}
		\end{subfigure}
		\hfill
		\begin{subfigure}[b]{0.49\textwidth}
			\centering
			\includegraphics[width=\textwidth]{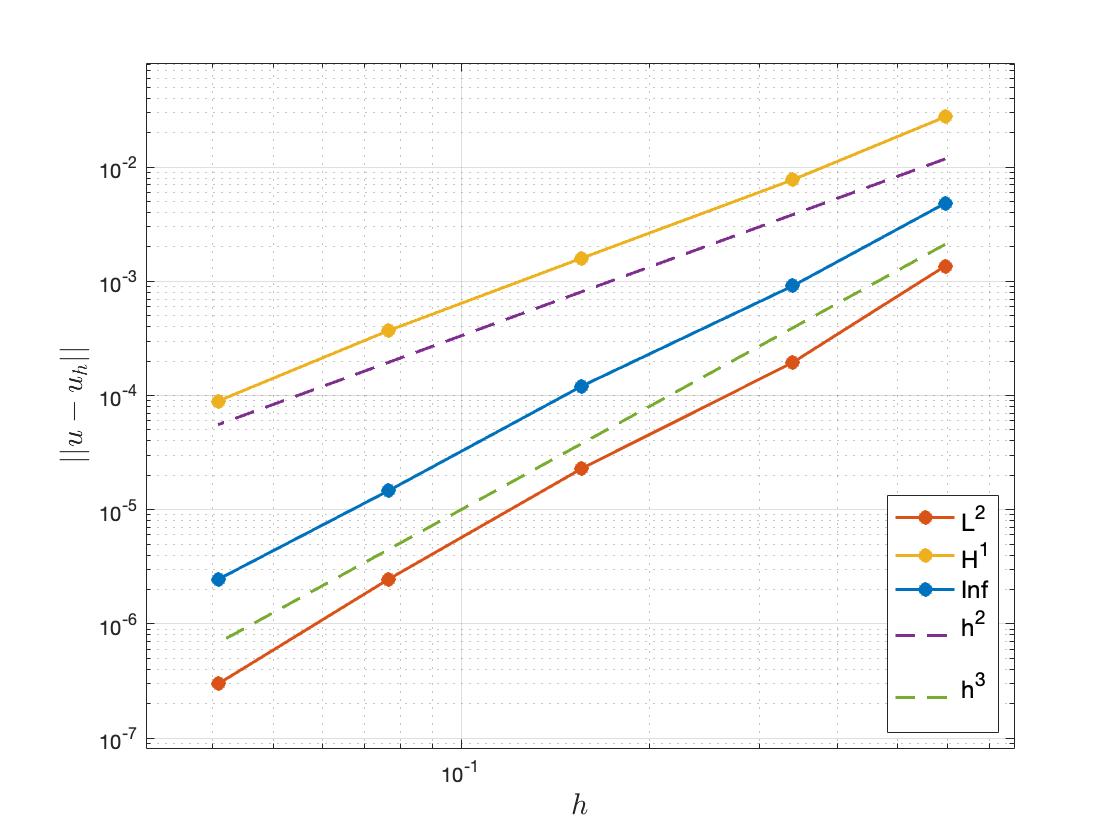}
			\caption{$p=2$}
		\end{subfigure}
		\vfill
		\begin{subfigure}[b]{0.49\textwidth}
			\centering
			\includegraphics[width=\textwidth]{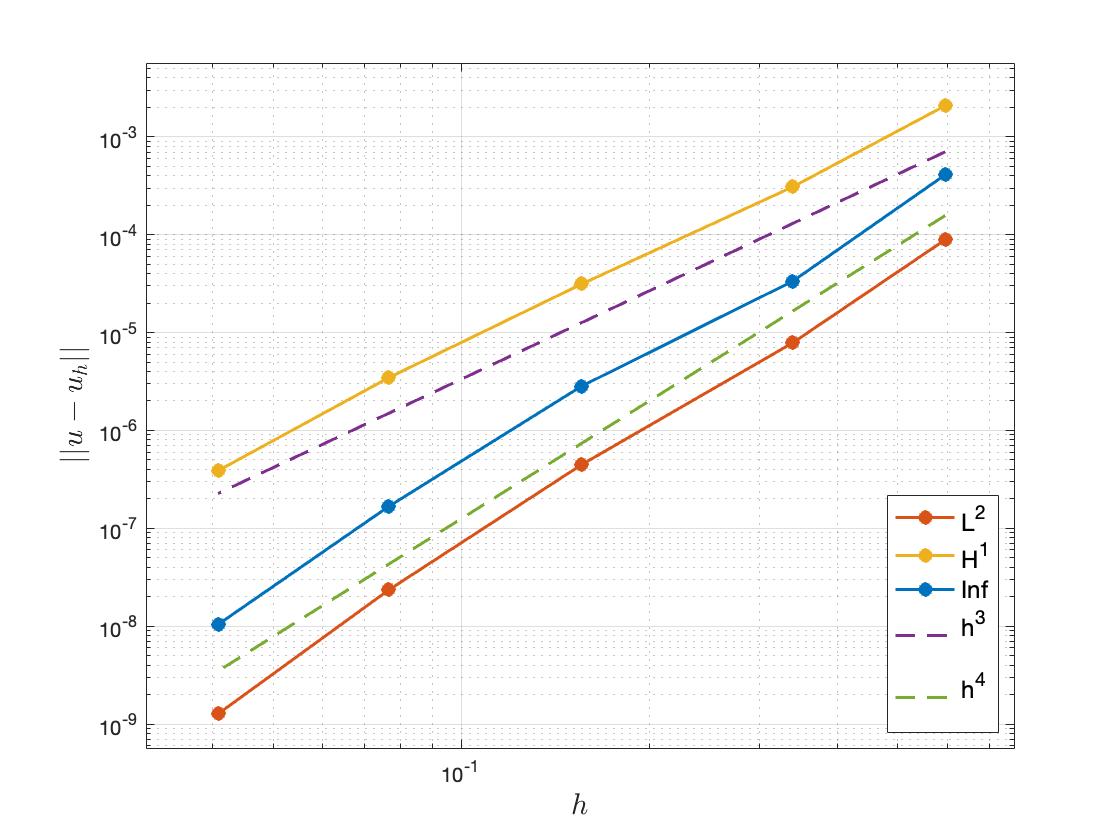}
			\caption{$p=3$}
		\end{subfigure}
		\hfill
		\begin{subfigure}[b]{0.49\textwidth}
			\centering
			\includegraphics[width=\textwidth]{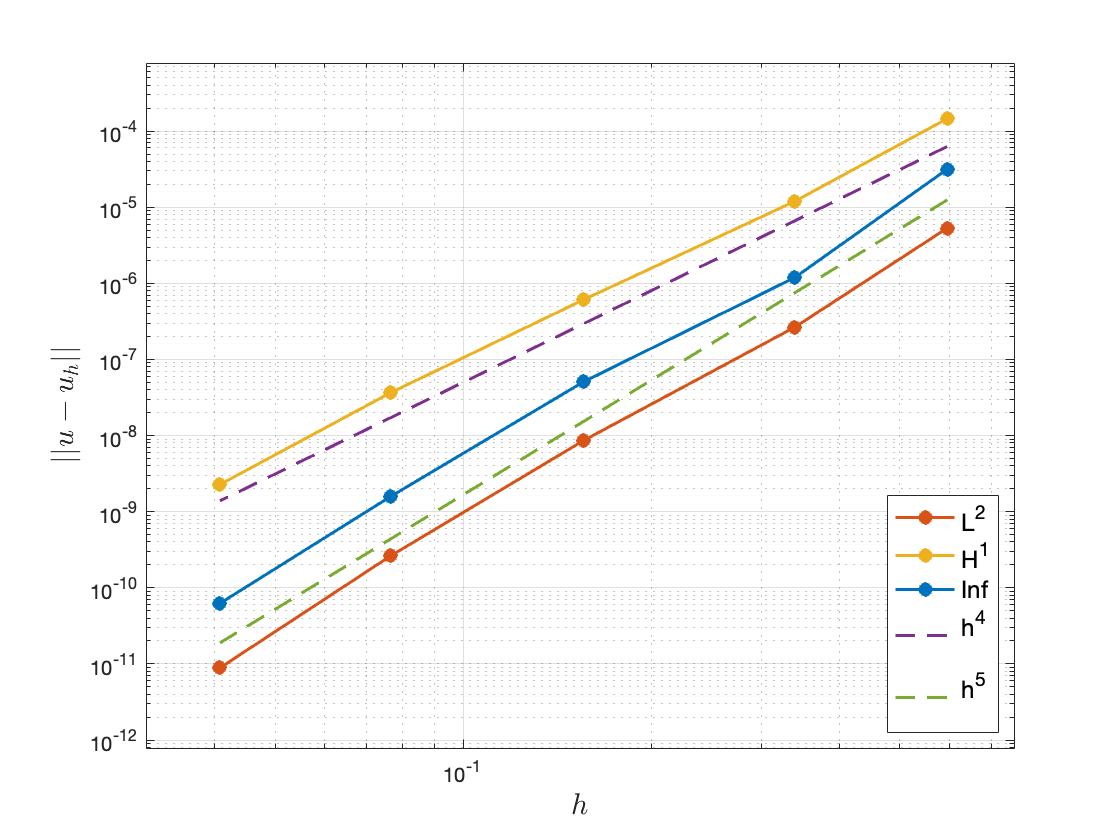}
			\caption{$p=4$}
		\end{subfigure}
		\caption{$h$-convergence for the test case with coefficients \eqref{example2}. The quasi-Trefftz DG error is plotted against the meshsize $h$ for the polynomial degree $p=1,2,3,4$.
			The red lines represent the error measured in the $L^2$-norm, the yellow lines in the $H^1$-norm and the blue lines in the $L^{\infty}$-norm.
			The dots correspond to $h=2^{-1},\dots,2^{-5}$.
		}
		\label{fig:hconv}
	\end{figure}
	
	For each polynomial degree we consider the values $h=2^{-j}$ for $j=1,\dots,5$ and we measure the error in the $L^2$-norm (red line) in the $H^1$-norm (yellow line) and in the $L^{\infty}$-norm (blue line).
	We choose $\gamma=8 p^2$ as penalty parameter and $\epsilon=-1$ as symmetrization parameter.
	We observe rates of convergence of order $\mathcal{O}(h^p)$ (dashed violet line) in the $H^1$-norm and of order  $\mathcal{O}(h^{p+1})$ (dashed green line) in the $L^2$-norm and in the $L^{\infty}$-norm.
	The rates correspond to the theoretical approximation bounds: rates of convergence of order $\mathcal{O}(h^{p})$ for the $\N{\cdot}_{dar}$-norm of the error.
	Hence, we have the optimal approximation properties that one achieve with the full polynomial space.
	
	We now consider a second experiment with a smooth solution in order to test the quasi-Trefftz DG method in a more general case with also Neumann boundary condition and with all the terms of diffusion, advection and reaction.
	As second example, we consider the problem \eqref{eq}-\eqref{Neumann} with $f=0$ and with coefficients
	\begin{equation}\label{smooth2}
		\bk(x_1,x_2)=\left[\begin{matrix}
			x_1+x_2+1 & 0\\
			0 & x_1+x_2+1
		\end{matrix}\right], \quad  \bbeta(x_1,x_2)=\left[\begin{matrix}
			1\\
			0
		\end{matrix}\right],  \quad \sigma(x_1,x_2)=\frac{3}{x_1+x_2+1}.
	\end{equation}
	We assign Dirichlet boundary condition $g_D$ on the inflow $\Gamma_{-}=\{\bx\in\ (0,1)^2 \mid x_1=0\}$ and Neumann boundary condition $g_N$ on the outflow $\Gamma_{+}=\{\bx\in\ (0,1)^2 \mid x_1=1 \text{ or } x_2=0 \text{ or } x_2=1\}$ such that the exact solution is given by  $u(x_1,x_2)=\frac{1}{x_1+x_2+1}$.
	Computing the quasi-Trefftz DG error, we obtain results similar to the ones shown in Figure \ref{fig:hconv} for the previous case.
	
	We now use this problem to compare the three formulations that arise varying the value of the symmetrization parameter $\epsilon$.
	We recall that for $\epsilon=-1$ we have the SIPG method, for $\epsilon=0$ we have the IIPG method and for $\epsilon=+1$ we have the NIPG method, as described in Section \ref{s:Diffvar}.
	In Table \ref{table:1} and in Table \ref{table:2} we show the quasi-Trefftz DG error measured  in the $L^2$-norm and in the $H^1$-norm, respectively.
	We consider the polynomial degrees $ p=2,3,4$ and the penalty parameter $\gamma=8 p^2$.
	For each polynomial degree we compute the error for the three methods for $h=2^{-3},2^{-4},2^{-5}$.
	
	\begingroup 
	\renewcommand{\arraystretch}{1.2}
	\begin{table}[h!]
		\begin{center}
			\begin{tabular}{  |c|c|c c|c c|c c| } 
				\hline
				\multicolumn{8}{|c|}{quasi-Trefftz DG error in $L^2$-norm} \\
				\hline
				&	$h$ & SIPG & rate & IIPG & rate & NIPG & rate \\
				\hline 
				\multirow{3}{4em}{$p=2$, $\gamma=32$ }  & $2^{-3}$ & $ 1.445\times 10^{-5}$   & - &   $3.159\times 10^{-5} $ & -&$ 4.990\times 10^{-5}  $         &- \\
				& $2^{-4}$ & $1.704 \times 10^{-6} $&  3.085     &  $6.488\times 10^{-6}  $  &2.284   & $1.105\times 10^{-5} $&    2.175   \\
				& $2^{-5}$ & $ 1.991\times 10^{-7}  $  &  3.098      & $    1.491\times 10^{-6}  $   &  2.122     & $  2.616\times 10^{-6}$  &     2.079      \\
				\hline
				\multirow{3}{4em}{$p=3$, $\gamma=72$ }  & $2^{-3}$ &  $8.496\times 10^{-7} $&-  & $ 7.530\times 10^{-7}$ &- & $7.183 \times 10^{-7}$ &- \\
				& $2^{-4}$ & $4.954\times 10^{-8} $& 4.100   & $4.276 \times 10^{-8}$& 4.138  &$4.044 \times 10^{-8}$&  4.151   \\
				& $2^{-5}$ &$2.947 \times 10^{-9}$  &  4.071    & $ 2.479\times 10^{-9} $& 4.108   &$ 2.289\times 10^{-9}$ &4.143  \\
				\hline
				\multirow{3}{4em}{$p=4$, $\gamma=128$ }  & $2^{-3}$ &$ 5.753\times 10^{-8}$  & - & $5.929 \times 10^{-8} $  & -&  $6.578\times 10^{-8}  $   &- \\
				& $2^{-4}$ &$1.860\times 10^{-9} $ &  4.951      &  $2.104\times 10^{-9}  $&4.817& $2.681\times 10^{-9}  $  &  4.617    \\
				& $2^{-5}$ & $5.991\times 10^{-11}  $  &  4.956         &$   8.486\times 10^{-11}$ & 4.632      &$ 1.296\times 10^{-10}   $    &      4.371   \\
				\hline
			\end{tabular}
		\end{center}
		\caption{Comparison of the three formulation SIPG, IIPG and NIPG in terms of $L^2$ error for polynomial degree $p=2,3,4$ for the example \eqref{smooth2}.}
		\label{table:1}
	\end{table}
	\endgroup
	
	All the three methods present optimal convergence rates of order $\mathcal{O}(h^p)$ in the $H^1$-norm, as shown in  Table  \ref{table:2}. 
	Observing the rates in the $L^2$-norm in Table \ref{table:1}, we see that the SIPG method has convergence rates of order 
	$\mathcal{O}(h^{p+1})$ while the IIPG and the NIPG methods present rates of order  $\mathcal{O}(h^{p+1})$ if the polynomial degree is odd and of order  between $\mathcal{O}(h^{p})$ and $\mathcal{O}(h^{p+1})$ if the polynomial degree is even.
	This is also observed numerically in \cite[p. 57]{riviere2008discontinuous} for a problem with only the diffusion term and  in \cite[Theorem 2.14]{riviere2008discontinuous} is proved that the $L^2$ optimal error estimate holds for the SIPG method while the others two have suboptimal rates.
	For this reason and also since in the problems with only diffusion it preserves the symmetry of the original problem, the SIPG formulation is preferred.
	\begingroup 
	\renewcommand{\arraystretch}{1.2}
	\begin{table}[h!]
		\begin{center}
			\begin{tabular}{   |c|c|c c|c c|c c| } 
				\hline
				\multicolumn{8}{|c|}{quasi-Trefftz DG error in $H^1$-norm} \\
				\hline
				&	$h$ & SIPG & rate & IIPG & rate & NIPG & rate \\
				\hline
				\multirow{3}{4em}{$p=2$, $\gamma=32$ }  & $2^{-3}$ & 	$  1.145\times 10^{-3} $   & - &  $1.113\times 10^{-3}  $  & -&$1.096\times 10^{-3} $     &- \\
				& $2^{-4}$ &$ 2.785\times 10^{-4} $  &	2.039     &  $ 2.722\times 10^{-4} $ &  2.032  &$2.687\times 10^{-4}$&     2.028 \\
				& $2^{-5}$ & $6.640\times 10^{-5} $    &	 2.069     &     $6.528\times 10^{-5}  $     &   2.060    & $6.467\times 10^{-5}$&     2.055   \\
				\hline
				\multirow{3}{4em}{$p=3$, $\gamma=72$ }  & $2^{-3}$ & $6.721\times 10^{-5}  $& - &   $6.587\times 10^{-5}$   &- & $6.501\times 10^{-5}$ &- \\
				& $2^{-4}$ &$ 7.992\times 10^{-6} $ & 3.072      &  $7.882\times 10^{-6}$   &      3.063      &$7.813\times 10^{-6}$&  3.057   \\
				& $2^{-5}$ 	&   $ 9.229\times 10^{-7} $   & 3.114 &   $9.157\times 10^{-7}  $    &    3.106  &    $ 9.111\times 10^{-7} $ &    3.100  \\
				\hline
				\multirow{3}{4em}{$p=4$, $\gamma=128$ }  & $2^{-3}$ & 		$4.770\times 10^{-6} $  &  -&  $4.652\times 10^{-6}   $   & -&$  4.562\times 10^{-6} $   &- \\
				& $2^{-4}$ &$	2.978\times 10^{-7} $ &  4.002     &   $2.909\times 10^{-7}  $  &3.9995 &$2.854\times 10^{-7} $   &    3.999   \\
				& $2^{-5}$ &$ 1.819\times 10^{-8} $  &   4.033          &    $  1.777 \times 10^{-8}  $   & 4.033   & $1.743\times 10^{-8}  $ &   4.034    \\
				\hline
			\end{tabular}
		\end{center}
		\caption{Comparison of the three formulation SIPG, IIPG and NIPG in terms of $H^1$ error for polynomial degree $p=2,3,4$ for the example  \eqref{smooth2}.}
		\label{table:2}
	\end{table}
	\endgroup
	
	\section{\textit{p}-convergence and number of DOFs}\label{p-conv}
	We now compare the quasi-Trefftz space $\QT^p(\calT_h)$ with the full polynomial space $\mathbb{P}_2^p(\calT_h)$ in terms of $p$-convergence and of number of degrees of freedom (DOFs).
	We consider the first example in Section \ref{h-conv} with coefficients \eqref{example2}, we fix the meshsize $h=2^{-4}$ and choose the parameters $\gamma=8 p^2$ and $\epsilon=-1$.
	
	\begin{figure}[h!]
		\centering
		\begin{subfigure}[b]{0.49\textwidth}
			\centering
			\includegraphics[width=\textwidth]{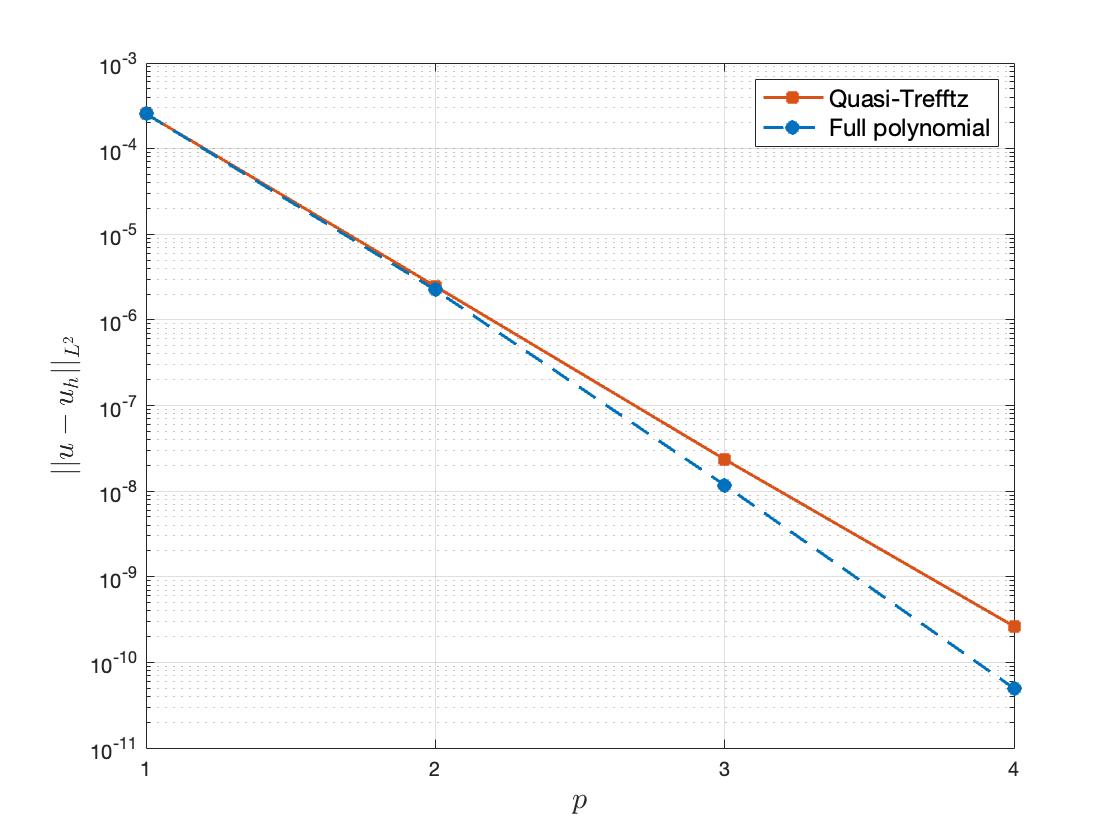}
			\caption{Error in norm $L^2$}
		\end{subfigure}
		\hfill
		\hfill
		\begin{subfigure}[b]{0.49\textwidth}
			\centering
			\includegraphics[width=\textwidth]{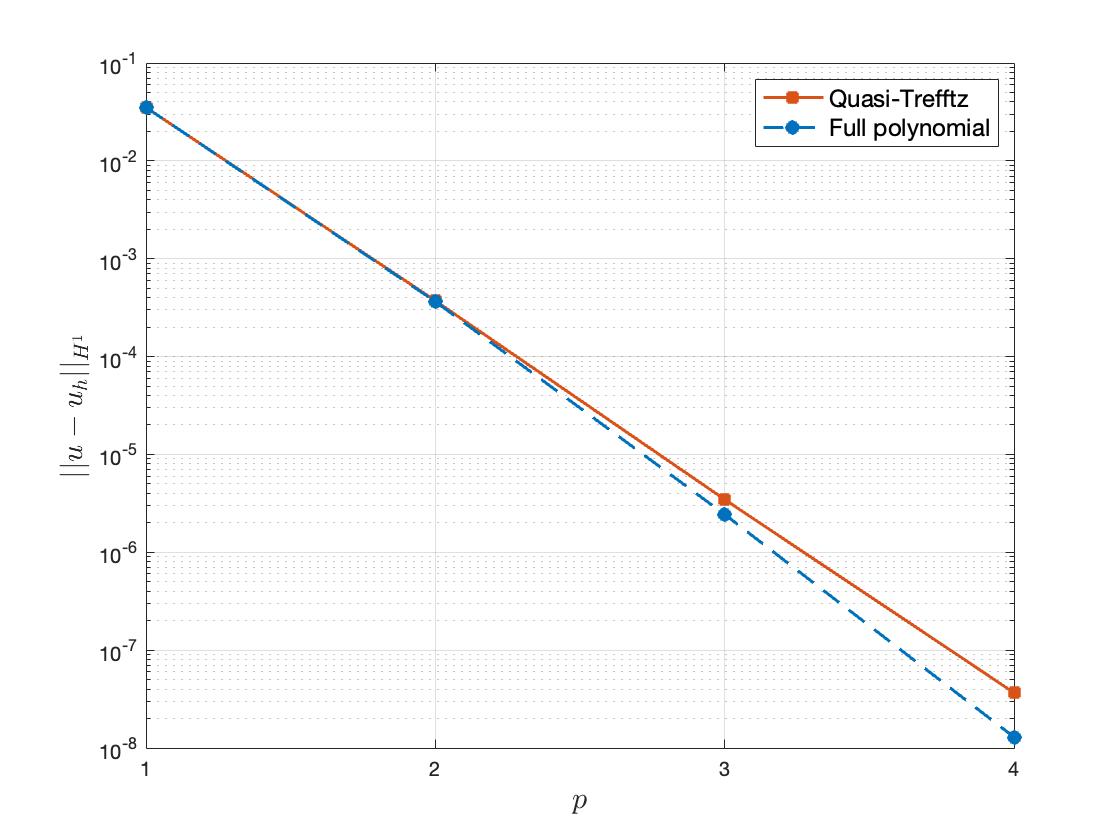}
			\caption{Error in norm $H^1$}
		\end{subfigure}
		\caption{$p$-convergence for the problem with coefficients \eqref{example2}. The red lines represent the DG error for the $\QT^p(\calT_h)$ while the dashed blue lines  the DG error for $\mathbb{P}^p_2(\calT_h)$. The error is measured  in the  $L^2$-norm (a) and in the  $H^1$-norm (b) and the meshsize is $h=2^{-4}$.}
		\label{fig:pconvergence}
	\end{figure}
	
	In Figure \ref{fig:pconvergence} we  plot the error, in the $L^2$-norm (a) and in the $H^1$-norm (b), obtained using the quasi-Trefftz space (red line) and the full polynomial space (dashed blue line) for polynomial degree $p=1,2,3, 4$. 
	These plots show that the quasi-Trefftz space, as the full polynomial space, leads to exponential convergence in terms of polynomial degree $p$.
	However, we observe that the full polynomial space presents a slightly smaller error.
	The advantages of the quasi-Trefftz space are shown in Figure \ref{fig:dofs}, where the error is compared in terms of global number of degrees of freedom.
	We clearly see that, for comparable number of degrees of freedom, with the quasi-Trefftz space we can reach a higher accuracy, especially for high polynomial degrees $p$.
	\begin{figure}[h!]
		\centering
		\begin{subfigure}[b]{0.49\textwidth}
			\centering
			\includegraphics[width=\textwidth]{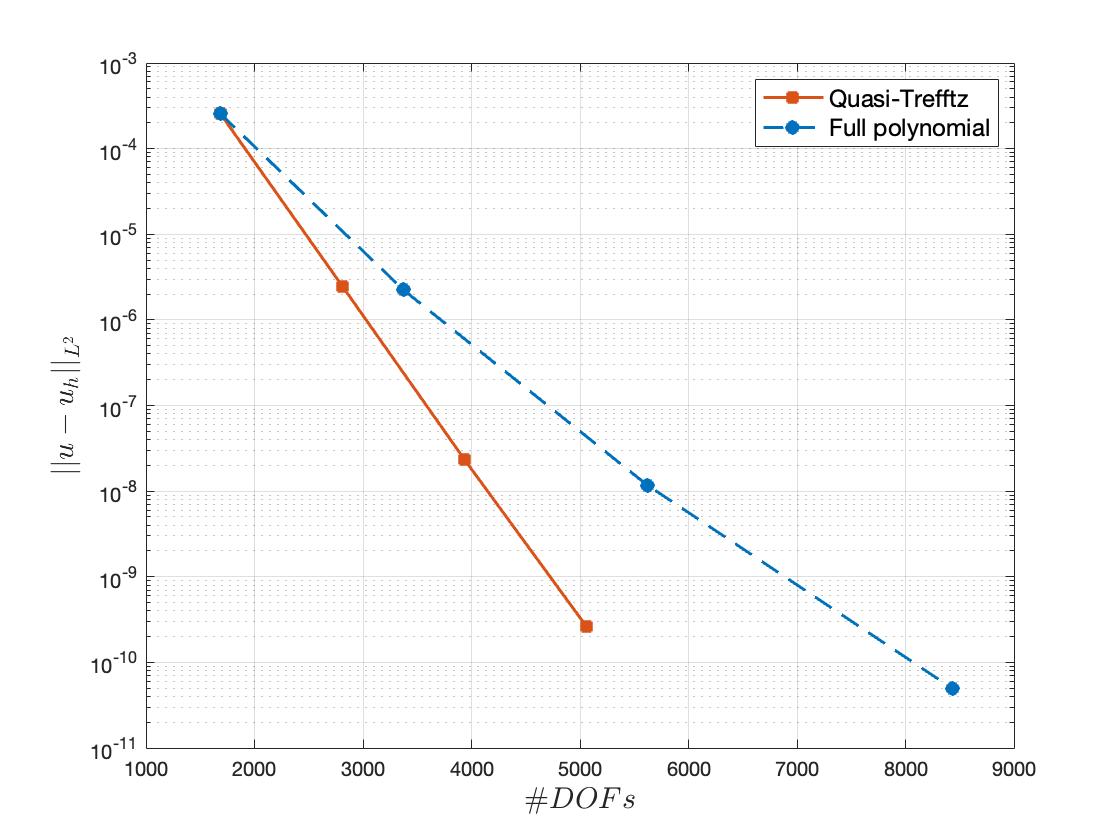}
			\caption{Error in norm $L^2$}
		\end{subfigure}
		\hfill
		\hfill
		\begin{subfigure}[b]{0.49\textwidth}
			\centering
			\includegraphics[width=\textwidth]{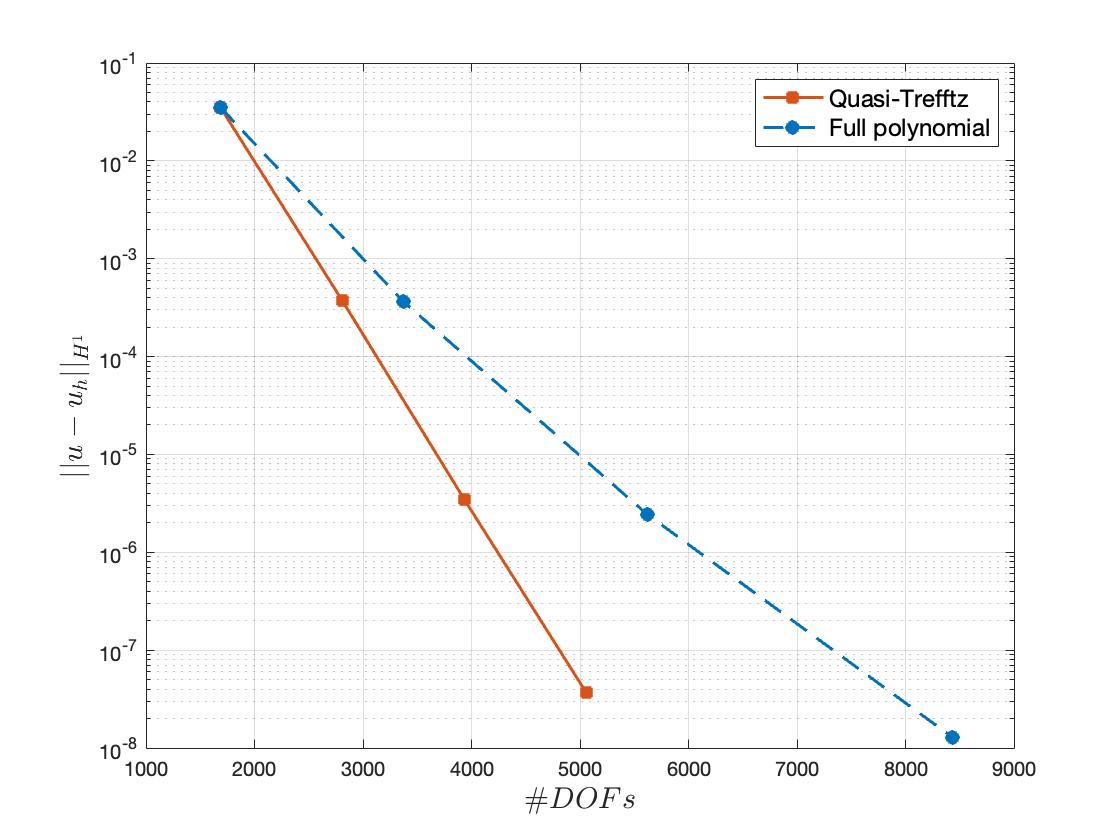}
			\caption{Error in norm $H^1$}
		\end{subfigure}
		\caption{Comparison between  $\QT^p(\calT_h)$ and $\mathbb{P}_2^p(\calT_h)$ in terms of degrees of freedom for the problem with coefficients \eqref{example2} using $h=2^{-4}$.
			The dots correspond to $p=1,2,3,4$ and the errors are measured  in the $L^2$-norm (a) and in the $H^1$-norm (b). These plots show that with the quasi-Trefftz space we have the same level of accuracy for much fewer degrees of freedom.
		}
		\label{fig:dofs}
	\end{figure}
	\vspace{-1.3cm}
	\section{Advection-dominated and reaction-dominated problems}\label{dom}
	In this section we consider advection-dominated and reaction-dominated problems in order to illustrate the approximate solutions obtained with the quasi-Trefftz DG method in these particular cases.
	We recall that the exact solution of this kind of problems can present boundary and internal layers, which are narrow regions where the solution and its gradient change rapidly.
	For this reason the numerical treatment of such problems can be difficult. It is known that, if we use the standard Finite Element Method, the approximate solution, in general, presents spurious oscillations and can be completely different from the exact one.
	
	First, we consider the problem \eqref{eq}-\eqref{Neumann} with $f=0$ and with coefficients  
	\begin{equation}\label{coeffdom}
		\bk(x_1,x_2)=\nu Id,\qquad \bbeta(x_1,x_2)=\left[\begin{matrix}
			x_2+1\\
			-x_1+2
		\end{matrix}\right],\qquad\sigma(x_1,x_2)=0,
	\end{equation}
	for different values of the diffusion parameter
	$\nu=10^{-j}$ for $j=1,\dots,4$.
	We impose the Dirichlet boundary condition on the inflow $\Gamma_{-}=\{\bx\in\ \partial\Omega \mid x_1=0\text{ or } x_2=0\}$:
	\vspace{0.5cm}
	
	\begin{equation}\label{DIRICHLETADVECTIONDOM}
		g_D(x_1,x_2)=	\begin{cases}
			1 \quad \text{ on $\{x_1=0, 0\leq x_2<1\}$,}\\
			1 \quad \text{ on $ \{x_2=0, 0\leq x_1 \leq \frac{1}{3}\}$,}\\
			0 \quad  \text{ on $ \{x_2=0, \frac{1}{3}< x_1 <1\}$,} \\
		\end{cases}
	\end{equation}
	and the Neumann boundary condition $g_N=0$ on the outflow $\Gamma_{+}=\{\bx\in\ \partial \Omega \mid x_1=1 \text{ or } x_2=1\}$.
	We note that, as the value of the diffusion parameter $\nu$ decreases, the advection term $\bbeta$ becomes dominant and for $0 < \nu \ll 1$ the exact solution presents an internal layer.
	
	In Figure \ref{advectiondomNeumann}  we plot the quasi-Trefftz DG approximate solutions obtained using the meshsize $h=2^{-4}$ and the polynomial degree $p=3$.
	
	\begin{figure}[h!]
		\centering
		\begin{subfigure}[b]{0.49\textwidth}
			\centering
			\includegraphics[width=\textwidth]{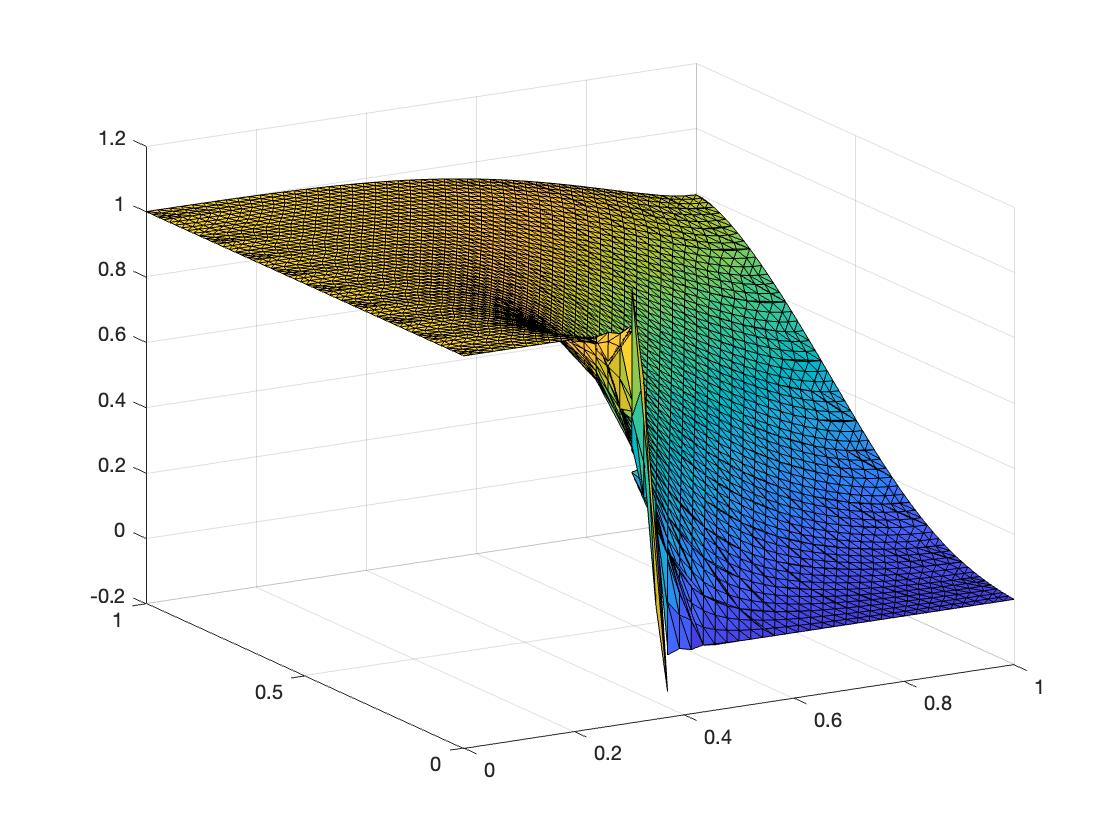}
			\caption{$\nu=10^{-1}$, $\gamma=10$}
		\end{subfigure}
		\hfill
		\begin{subfigure}[b]{0.49\textwidth}
			\centering
			\includegraphics[width=\textwidth]{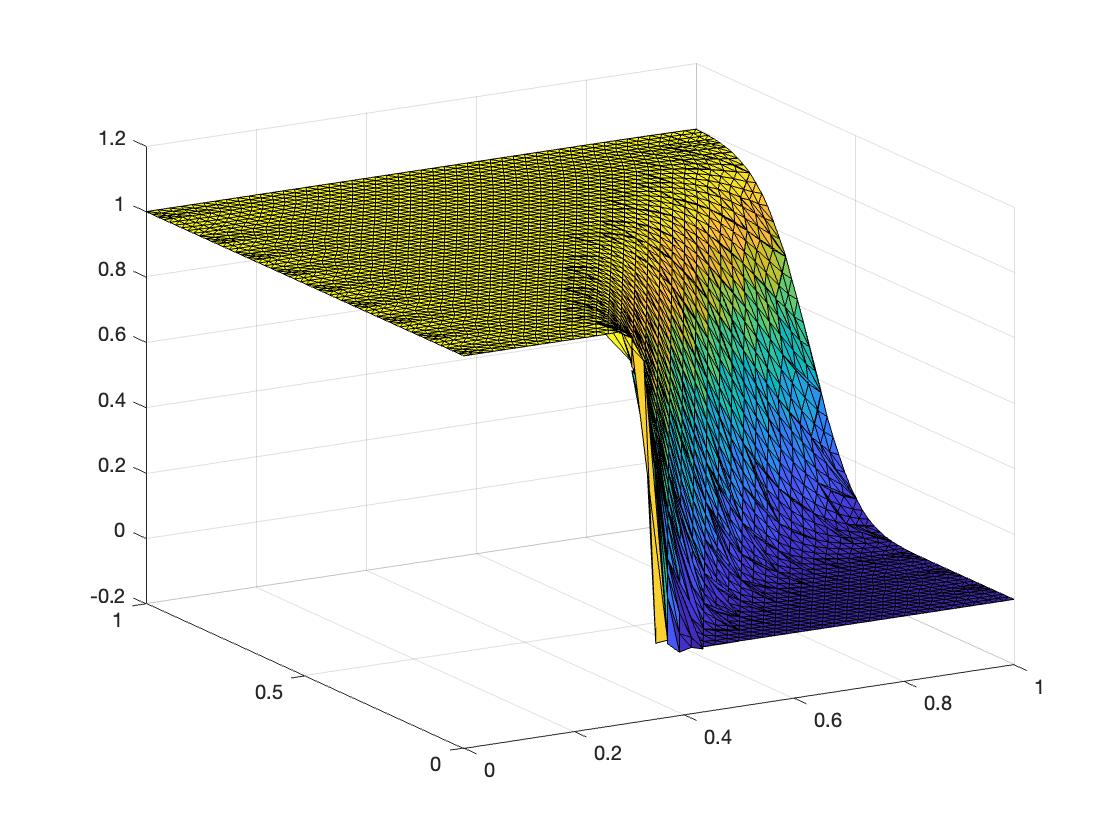}
			\caption{$\nu=10^{-2}$, $\gamma=1$}
		\end{subfigure}
		\vfill
		\begin{subfigure}[b]{0.49\textwidth}
			\centering
			\includegraphics[width=\textwidth]{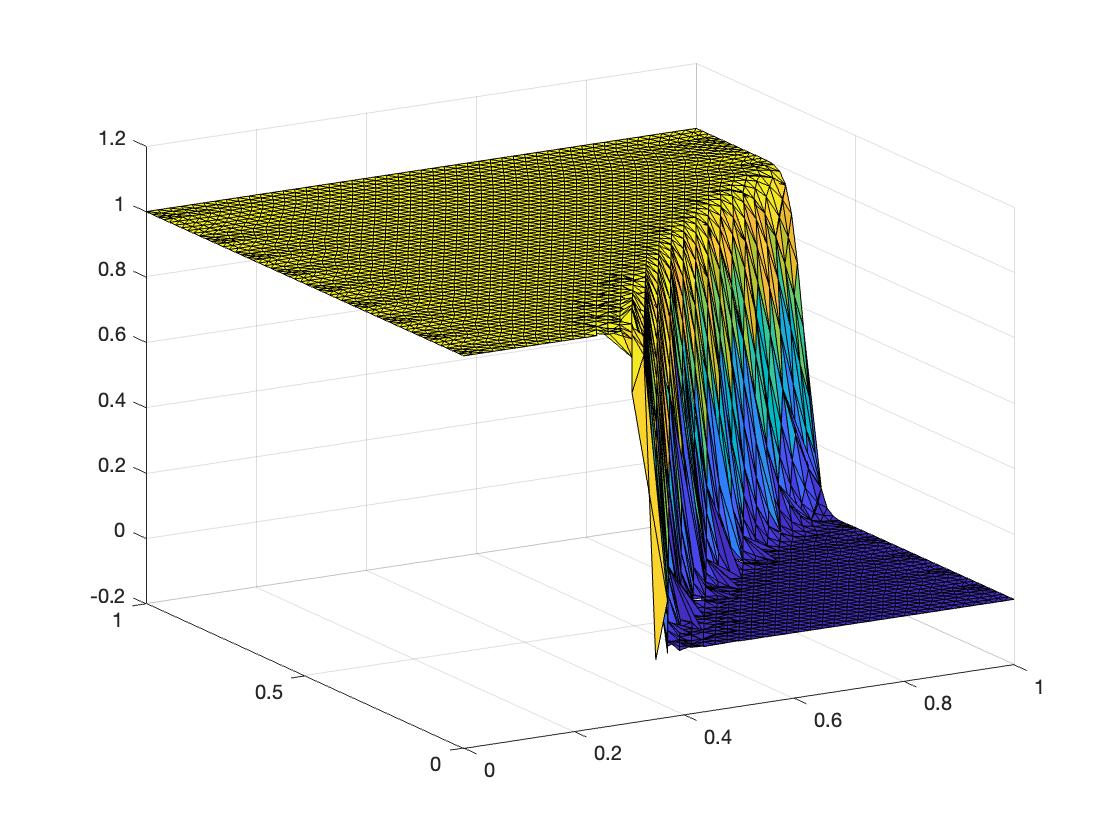}
			\caption{$\nu=10^{-3}$, $\gamma=10^{-1}$}
		\end{subfigure}
		\hfill
		\begin{subfigure}[b]{0.49\textwidth}
			\centering
			\includegraphics[width=\textwidth]{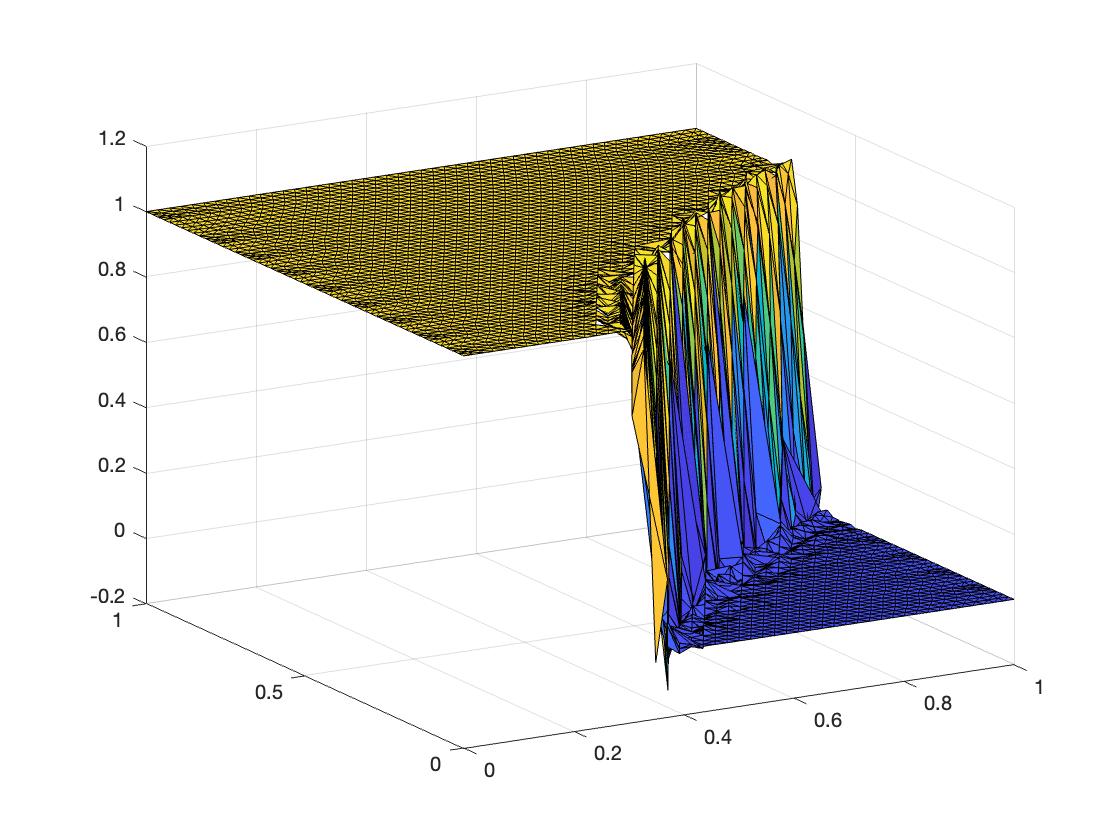}
			\caption{$\nu=10^{-4}$, $\gamma=10^{-2}$}
		\end{subfigure}
		\caption{Approximate solutions of the problem with coefficients \eqref{coeffdom}, obtained using the meshsize $h=2^{-4}$ and the quasi-Trefftz space with polynomial degree $p=3$.}
		\label{advectiondomNeumann}
	\end{figure}
	
	We fix the symmetrization parameter $\epsilon=-1$ and for each value of $\nu=10^{-1},\dots,10^{-4}$ we choose the penalty parameter $\gamma=10,1,10^{-1},10^{-2}$, respectively.
	These choices of the penalty parameter are made taking into consideration the two following aspects. 
	On the one hand, we cannot choose a too small $\gamma$ because 
	from the theory of Section \ref{coercivitydiscrete} we know that, in order to have the discrete coercivity,  the penalty parameter needs to be greater than  $\gamma_{\epsilon}$, defined in \eqref{gammaeps}, which in this case is equal to $3 \nu C_{tr}$, with $C_{tr}$ defined in \eqref{discretetraceinequality}.
	On the other hand, if the penalty parameter is too large we are imposing the continuity across the mesh facets and we obtain similar results to ones of the standard Galerkin method, which, as previously stated, it does not perform well in the advection-dominated regime due to the spurious oscillations.
	We notice that we do not have spurious oscillations spread over all the domain but the flat part of the solution is well approximated and 
	the discontinuity at the boundary and the internal layer are well captured with small oscillations.

	We now consider the same example with coefficients \eqref{coeffdom} but with Dirichlet boundary condition \eqref{DIRICHLETADVECTIONDOM} on the  inflow $\Gamma_{-}$ and homogeneous Dirichlet condition on the outflow $\Gamma_{+}$.
	This is done in order to assess the behaviour of the quasi-Trefftz DG method in a test case not covered by the theory, since the assumption \eqref{gammaDgammaN} is not valid. In this experiment in the advection-dominated regime the exact solution presents a boundary layer on the edge $\{x_2=0\}$ and also an internal layer.
	
	\begin{figure}[h!]
		\centering
		\begin{subfigure}[b]{0.49\textwidth}
			\centering
			\includegraphics[width=\textwidth]{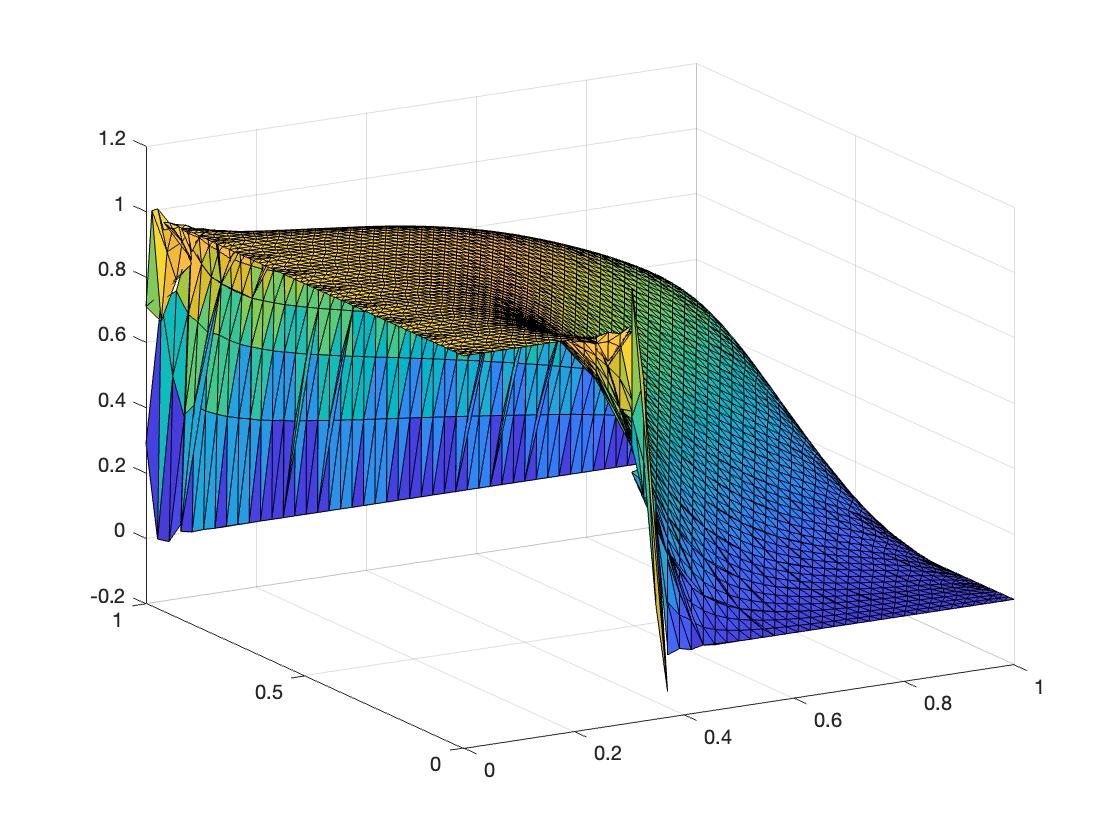}
			\caption{$\nu=10^{-1}$, $\gamma=10$}
		\end{subfigure}
		\hfill
		\begin{subfigure}[b]{0.49\textwidth}
			\centering
			\includegraphics[width=\textwidth]{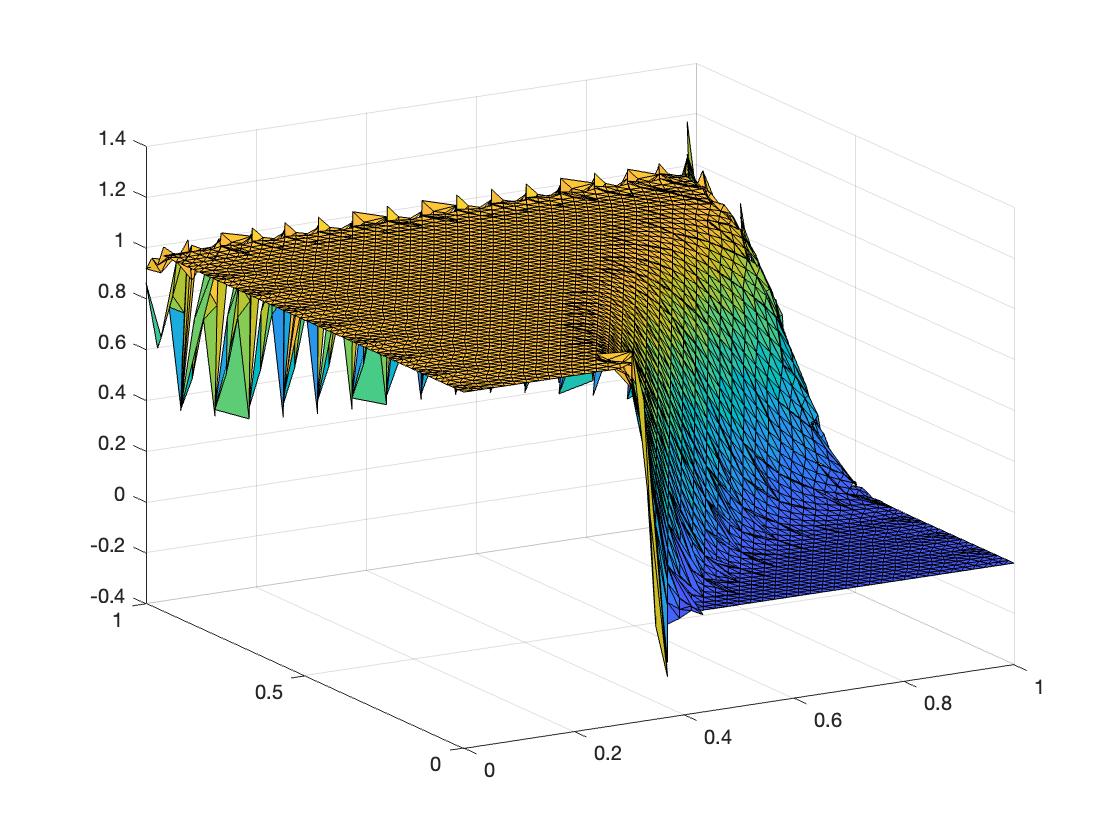}
			\caption{$\nu=10^{-2}$, $\gamma=10^{-1}$}
		\end{subfigure}
		\vfill
		\begin{subfigure}[b]{0.49\textwidth}
			\centering
			\includegraphics[width=\textwidth]{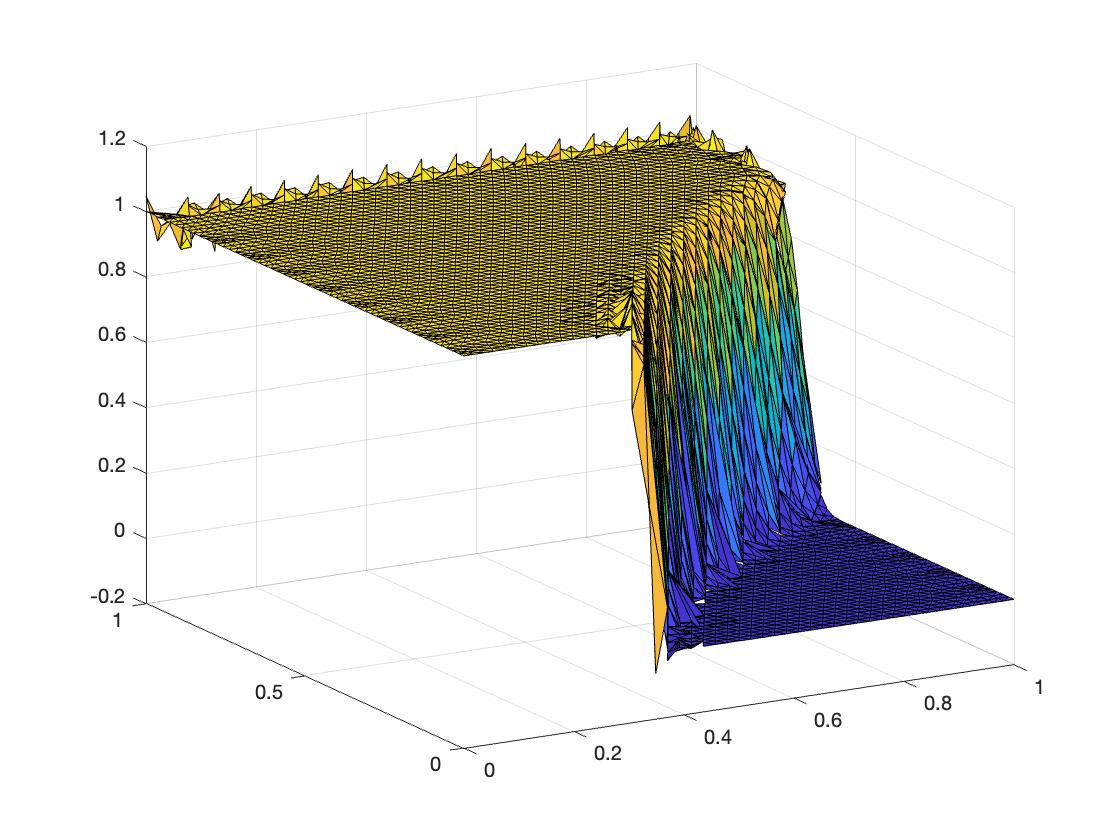}
			\caption{$\nu=10^{-3}$, $\gamma=10^{-2}$}
		\end{subfigure}
		\hfill
		\begin{subfigure}[b]{0.49\textwidth}
			\centering
			\includegraphics[width=\textwidth]{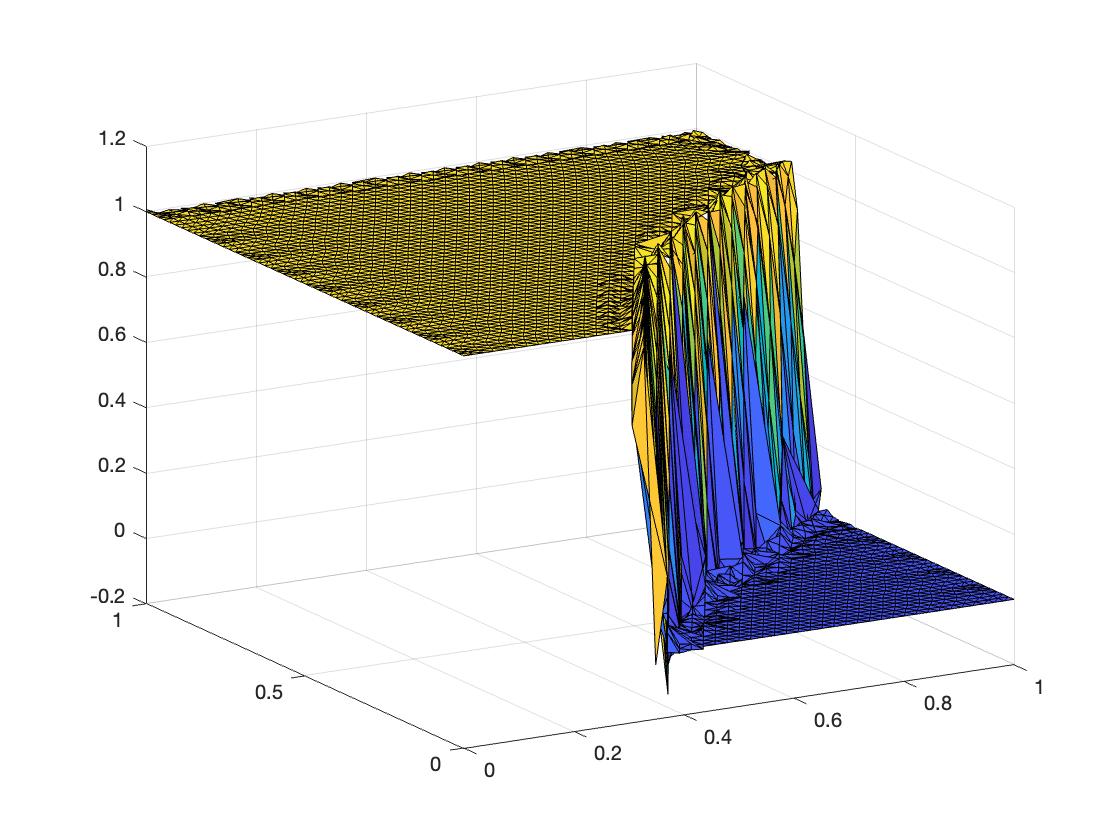}
			\caption{$\nu=10^{-4}$, $\gamma=10^{-3}$}
		\end{subfigure}
		\caption{Approximate solutions of the problem with coefficients \eqref{coeffdom} with only Dirichlet boundary condition, obtained using the meshsize $h=2^{-4}$ and the quasi-Trefftz space with polynomial degree $p=3$.}
		\label{advectiondom}
	\end{figure}
	
	In Figure \ref{advectiondom} we plot the quasi-Trefftz DG approximate solutions obtained using the meshsize $h=2^{-4}$ and the polynomial degree $p=3$.
	We fix  the symmetrization parameter $\epsilon=-1$ and for each value of $\nu=10^{-1},\dots,10^{-4}$ we choose the penalty parameter $\gamma=10,10^{-1},10^{-2},10^{-3}$, respectively.
	The internal layer is well captured with only very small oscillations. The boundary layer, instead, is completely missed in the advection-dominated regime but this is what we expect when using DG methods due to the weak enforcement of boundary conditions \cite{ayuso2009discontinuous}.
	In fact, we obtain very similar results if we use the DG scheme with the classical full polynomial space.
	
	We now consider the problem \eqref{eq}-\eqref{Neumann} with $f=0$ and with coefficients  
	\begin{equation}\label{coeffreac}
		\bk(x_1,x_2)=\nu Id,\qquad \bbeta(x_1,x_2)=\bm{0},\qquad\sigma(x_1,x_2)=x_1+x_2+1,
	\end{equation}
	for different values of the diffusion parameter
	$\nu=10^{-j}$ for $j=1,\dots,4$.
	We assign only Dirichlet boundary condition  $g_D=1$ on $\partial \Omega$.
	In this case, as the value of the diffusion parameter $\nu$ decreases, the reaction term $\sigma$ becomes dominant and for $0 < \nu \ll 1$ the exact solution presents boundary layers on $\partial \Omega$.
	In Figure \ref{advectiondom} we plot the quasi-Trefftz DG approximate solutions obtained using the meshsize $h=2^{-4}$ and the polynomial degree $p=3$.
	We fix the symmetrization parameter $\epsilon=-1$ and for each value of $\nu=10^{-1},\dots,10^{-4}$ we choose the penalty parameter $\gamma=10,1,10^{-1},10^{-2}$, respectively.
	As in the examples before, the flat part of the solution is well approximated and the boundary layers are captured with small oscillations.
	\begin{figure}[h!]
		\centering
		\begin{subfigure}[b]{0.49\textwidth}
			\centering
			\includegraphics[width=\textwidth]{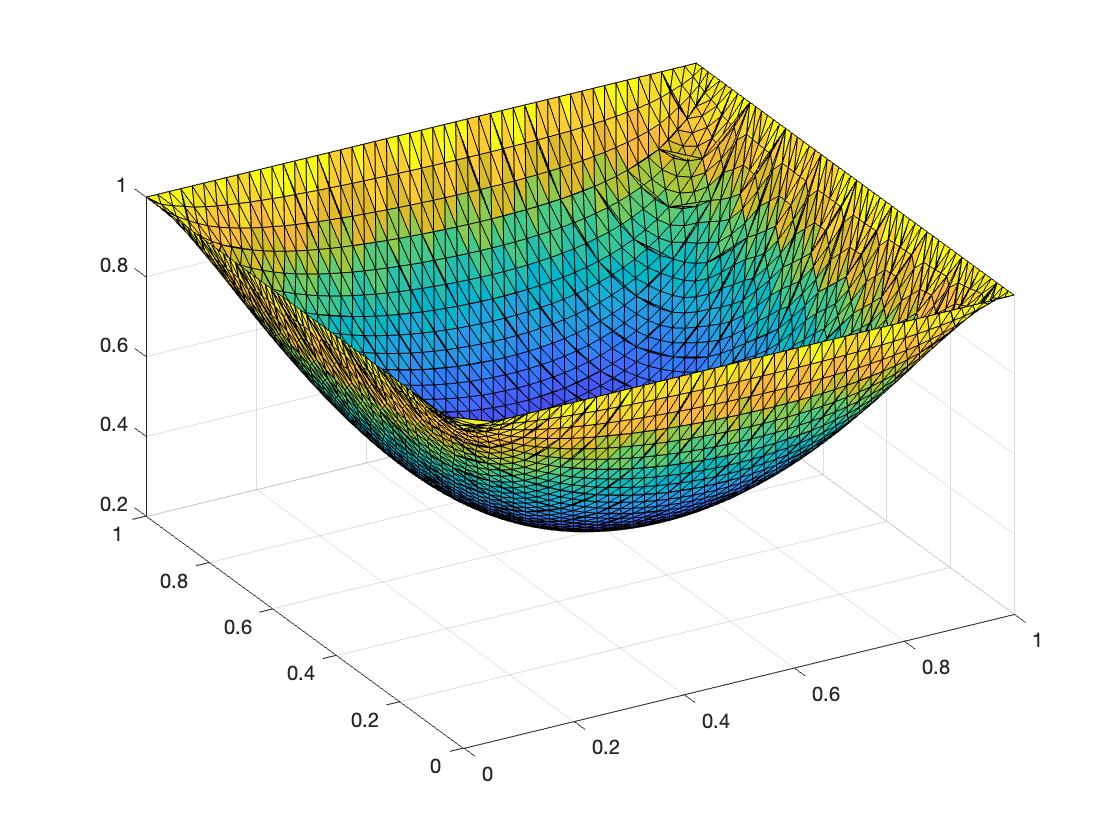}
			\caption{$\nu=10^{-1}$, $\gamma=10$}
		\end{subfigure}
		\hfill
		\begin{subfigure}[b]{0.49\textwidth}
			\centering
			\includegraphics[width=\textwidth]{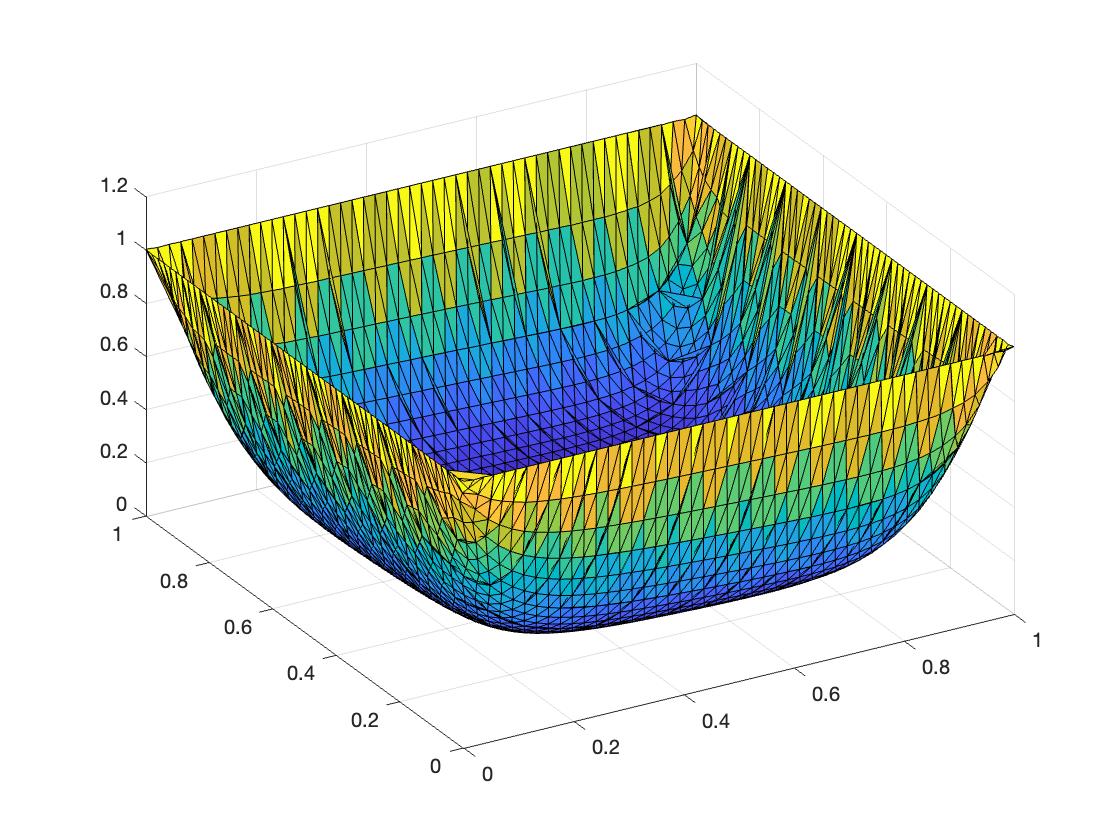}
			\caption{$\nu=10^{-2}$, $\gamma=1$}
		\end{subfigure}
		\vfill
		\begin{subfigure}[b]{0.49\textwidth}
			\centering
			\includegraphics[width=\textwidth]{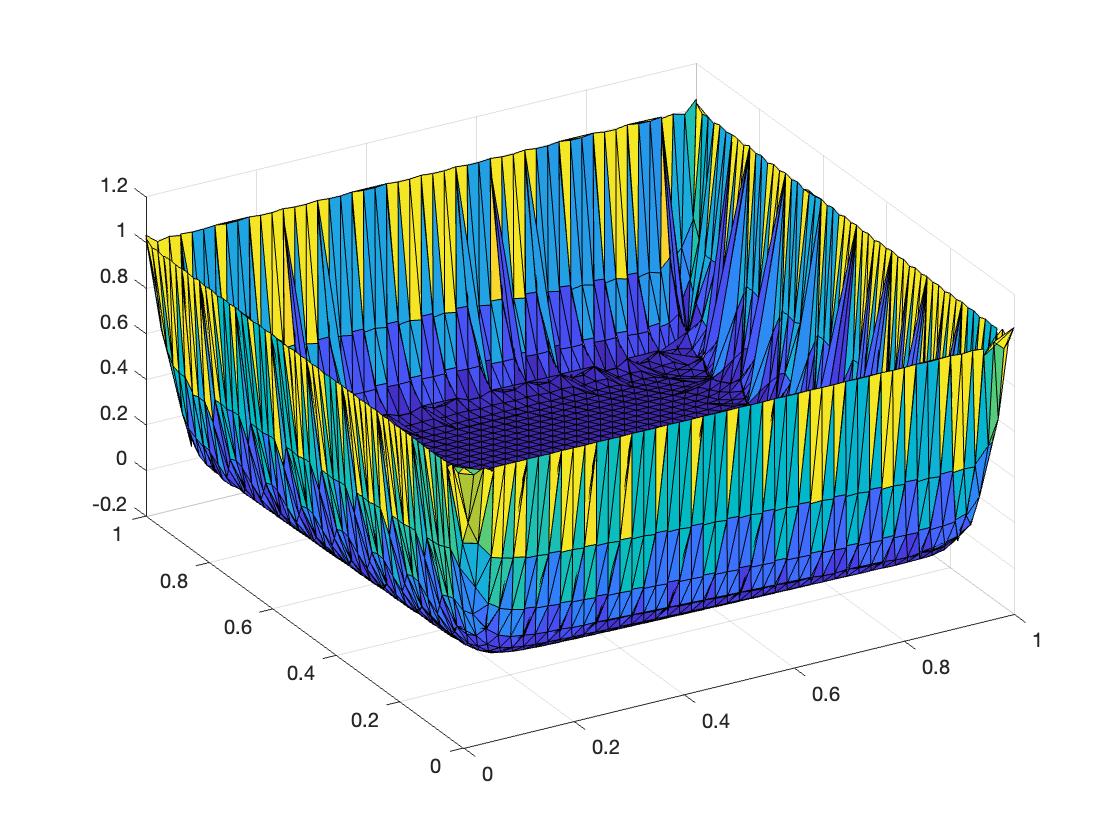}
			\caption{$\nu=10^{-3}$, $\gamma=10^{-1}$}
		\end{subfigure}
		\hfill
		\begin{subfigure}[b]{0.49\textwidth}
			\centering		\includegraphics[width=\textwidth]{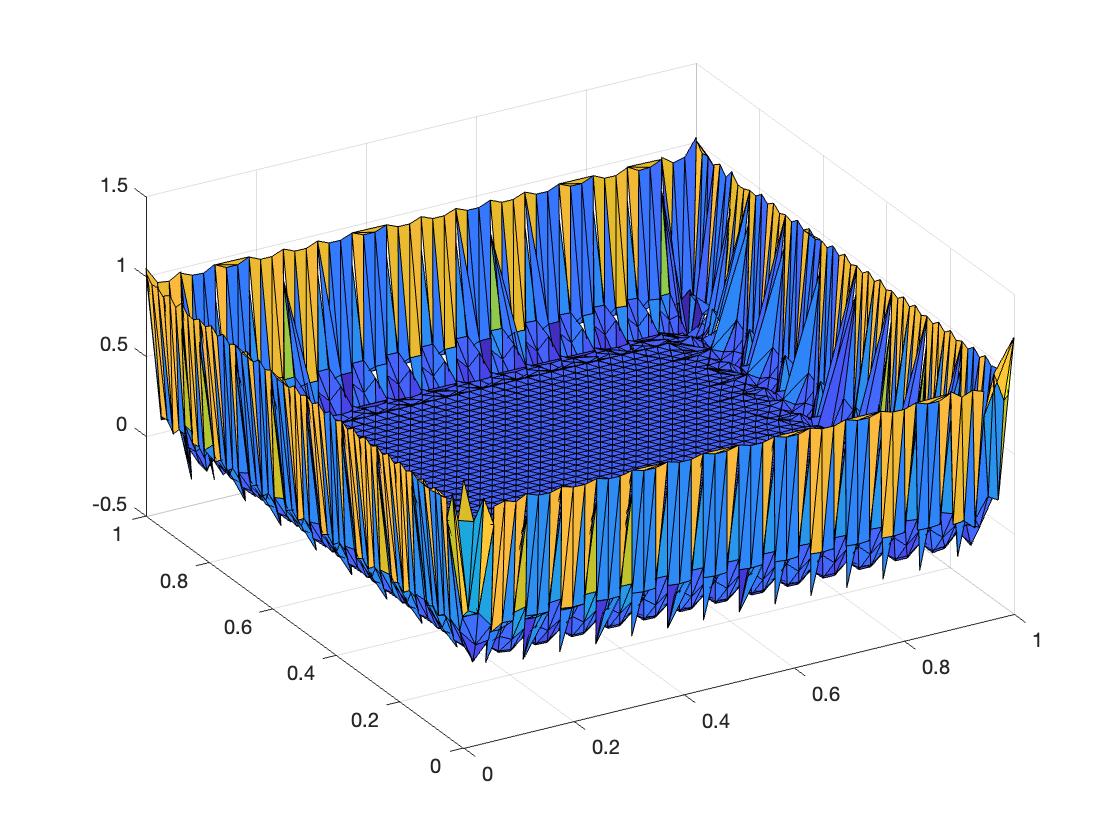}
			\caption{$\nu=10^{-4}$, $\gamma=10^{-2}$}
		\end{subfigure}
		\caption{Approximate solutions of the problem with coefficients \eqref{coeffreac}, obtained using the meshsize $h=2^{-4}$ and the quasi-Trefftz space with polynomial degree $p=3$.}
		\label{reactiondom}
	\end{figure}

	\chapter{Conclusions and future developments}\label{Chapter7}
	We have introduced a quasi-Trefftz discontinuous Galerkin method for approximating the homogeneous diffusion-advection-reaction equation with varying coefficients.
	The DG variational formulation, its well-posedness and the quasi-optimality in a mesh-dependent norm have been analysed.
	We have presented the polynomial quasi-Trefftz space for a general homogeneous linear partial differential equation with sufficiently smooth coefficients and proved that it reaches the same approximation orders with respect to the mesh size
	as the full polynomial space but with fewer degrees of freedom.
	For the homogeneous diffusion-advection-reaction equation, we have  proved optimal $h$-convergence of the quasi-Trefftz DG scheme 
	and described an algorithm for the construction of the quasi-Trefftz basis functions in dimension $d\in\mathbb{N}$, $d\ge1$.
	We have presented some numerical experiments in two dimensions that validate the accuracy of the method.
	
	We now consider some possible improvements of the method and future research directions.
	
	To better assess the computational performance of the method, the method should be implemented 
	in the three-dimensional case and in a high-performance open-source software such as NGSolve, where the Trefftz method and quasi-Trefftz method for the wave equation have already been implemented.
	
	The algorithm for the construction of the quasi-Trefftz basis functions is initialized assigning two polynomial bases. Those conditions could be optimized, through a more accurate study of the properties of the basis functions, in order to improve conditioning, computing time and accuracy.
	
	The construction of non-polynomials quasi-Trefftz functions might also be analysed; they could be useful, for example, to efficiently approximate solutions that present boundary layers, such as in the case of advection-dominated or reaction-dominated problems.
	
	We have proved optimal convergence rate with respect to the mesh size (\textit{h}-convergence) and a challenging extension is the analysis of the approximation properties for increasing polynomial degrees (\textit{p}-convergence). 
	This is a difficult task since in the context of Trefftz schemes has been achieved for time-harmonic equations (Helmholtz \cite{hiptmair2011plane} and Maxwell \cite{hiptmair2013error}) but not yet for the wave equation and neither in the context of the quasi-Trefftz methods.
	
	The analysis of the DG scheme considered includes the source term while the quasi-Trefftz space is designed for the homogenous equation.
	A significant improvement of the method is the extension to the inhomogeneous advection-diffusion-reaction equation.
	We need to construct a particular approximate local solution, considering a variant of the algorithm presented, and then apply the quasi-Trefftz DG method in an adequate way to solve for the difference.
	
	Another interesting extension is the combination of the method presented and the quasi-Trefftz DG method already implemented for the wave equation with piecewise-smooth coefficients \cite{imbert2023space}, in order to approximate PDEs whose nature changes in the domain, such as the Euler-Tricomi equation ($\partial^2_{x}u +x \partial^2_{y}u= 0$), useful in the study of transonic flow.
	Other similar equations are used in plasma physics, for example, $-\Delta u+\sigma u=0$, with $\sigma$ that changes sign,
	which is a Helmholtz equation in a part of the domain and a diffusion-reaction equation in the other  \cite{imbert2013analyse}.
	
	\nocite{*}
	\printbibliography[heading=bibintoc,title={Bibliography}]
\end{document}